\documentclass[opre,nonblindrev]{informs3a}

\OneAndAHalfSpacedXI 

\usepackage{endnotes}
\let\footnote=\endnote
\let\enotesize=\normalsize
\def\notesname{Endnotes}%
\def\makeenmark{$^{\theenmark}$}
\def\enoteformat{\rightskip0pt\leftskip0pt\parindent=1.75em
  \leavevmode\llap{\theenmark.\enskip}}

\usepackage{natbib}
 \bibpunct[, ]{(}{)}{,}{a}{}{,}%
 \def\bibfont{\small}%

\usepackage{algorithm}
\usepackage{algorithmicx}
\usepackage{algpseudocode}
\algnewcommand{\algorithmicand}{\textbf{and }}
\algnewcommand{\algorithmicor}{\textbf{or }}
\algnewcommand{\OR}{\algorithmicor}
\algnewcommand{\AND}{\algorithmicand}
\makeatletter
\newenvironment{breakablealgorithm}
  {
  \begin{center}
     \refstepcounter{algorithm}
     \hrule height.8pt depth0pt \kern2pt
     \renewcommand{\caption}[2][\relax]{
      {\raggedright\textbf{\ALG@name~\thealgorithm} ##2\par}%
      \ifx\relax##1\relax 
         \addcontentsline{loa}{algorithm}{\protect\numberline{\thealgorithm}##2}%
      \else 
         \addcontentsline{loa}{algorithm}{\protect\numberline{\thealgorithm}##1}%
      \fi
      \kern2pt\hrule\kern2pt
     }
  }{
     \kern2pt\hrule\relax
  \end{center}
  }
\usepackage{makecell}

\usepackage{ulem}
\usepackage{booktabs}
\usepackage{multirow}
\usepackage{comment}
\usepackage{subcaption}
\usepackage{verbatimbox}

\newcommand{\lovasz}{Lov\'{a}sz~}

\usepackage{xcolor}

\definecolor{tumb}{RGB}{0,101,189}

\newcommand{\revise}[1]{{\color{black}#1}}
\newcommand{\revisee}[1]{{\color{black}#1}}

\TheoremsNumberedThrough     
\ECRepeatTheorems

\EquationsNumberedThrough    

\MANUSCRIPTNO{} 

\begin{document}


\RUNAUTHOR{Zhang, Zheng, and Lavaei}

\RUNTITLE{Localization Methods for Convex Discrete Optimization via Simulation}

\TITLE{Stochastic Localization Methods for \\ Convex Discrete Optimization via Simulation}


\ARTICLEAUTHORS{%
\AUTHOR{Haixiang Zhang}
\AFF{Department of Mathematics, University of California, Berkeley, CA 94720, \EMAIL{haixiang\_zhang@berkeley.edu}} 
\AUTHOR{Zeyu Zheng}
\AFF{Department of Industrial Engineering and Operations Research, University of California, Berkeley, CA 94720, \EMAIL{zyzheng@berkeley.edu}}
\AUTHOR{Javad Lavaei}
\AFF{Department of Industrial Engineering and Operations Research, University of California, Berkeley, CA 94720,
\EMAIL{lavaei@berkeley.edu}}
} 

\ABSTRACT{%

We develop and analyze a set of new sequential simulation-optimization algorithms for large-scale multi-dimensional \textit{discrete optimization via simulation problems} with a convexity structure. The ``large-scale" notion refers to that the decision variable has a large number of values to choose from on each dimension. The proposed algorithms are targeted to identify a solution that is close to the optimal solution given any precision level with any given probability. 
To achieve this target, utilizing the convexity structure, our algorithm design does not need to scan all the choices of the decision variable, but instead sequentially draws a subset of choices of the decision variable and uses them to ``localize" potentially near-optimal solutions to an adaptively shrinking region.

To show the power of the localization operation, we first consider one-dimensional large-scale problems. We propose the shrinking uniform sampling algorithm, which is proved to achieve the target with an optimal expected simulation cost under an asymptotic criterion. For multi-dimensional problems, we combine the idea of localization with subgradient information and propose a framework to design stochastic cutting-plane methods and the dimension reduction algorithm, whose expected simulation cost 
have a low dependence on the scale and the dimension of the problems. 
The proposed algorithms do not require prior information about the Lipschitz constant of the objective function and the simulation costs are upper bounded by a value that is independent of the Lipschitz constant. Finally, we propose an adaptive algorithm to deal with the unknown noise variance case under the assumption that the randomness of the system is Gaussian. 
%
We implement the proposed algorithms on both synthetic and queueing simulation optimization problems, and demonstrate better performances compared to benchmark methods especially for large-scale examples.
}%


%

\maketitle

%


\section{Introduction}
In the areas of operations research and management science, many decision-making problems involve complex stochastic systems and discrete decision variables. In presence of stochastic uncertainties, many replications of stochastic simulation are often needed to accurately evaluate the objective function associated with a discrete decision variable. Such problems are sometimes referred to as \textit{Discrete Optimization via Simulation}, \textit{Discrete Simulation Optimization}, or \textit{Simulation/Stochastic Optimization with Integer Decision Variables} (see \cite*{nelson2010optimization,hong2015discrete,ragavan2021adaptive}).  For complex stochastic systems, even one replication of simulation can be time consuming or costly; see also \cite*{xu2010industrial,sun2014balancing,xu2016accelerated} for related discussions. When the decision space is large, it is often computationally impractical to run simulations for all choices of the decision variables, creating a challenge in finding the optimal or near-optimal choice of decision variables. To circumvent this challenge, problem structure such as convexity or local convexity of the objective function may need to be exploited to lower costs and improve the efficiency to find an optimal or near-optimal choice.

In this paper, we consider large-scale \revise{discrete optimization via simulation} problems with a convex objective function. The notion of ``large-scale" refers to a large number of choices for the discrete decision variable on each dimension. Optimization problems with such features naturally arise in many operations research and management science applications, including queueing networks, supply chain networks, sharing economy operations, financial markets, etc.; see \cite{shaked1988stochastic}, \cite{wolff2002convexity}, \cite{altman2003discrete},   \cite{singhvi2015predicting,jian2016simulation,freund2017minimizing} for example. Particularly in the area of supply chain management, a significant amount of models are proved to be discrete convex: lost-sales inventory systems with positive lead time \citep{zipkin2008structure}; serial inventory systems \citep{huh2010optimal}; single-stage inventory systems with positive order lead time \citep{pang2012note}; capacitated inventory systems with remanufacturing \citep{gong2013optimal}; more applications are discussed in \citet{chen2020discrete}. Overall in these papers, the authors consider various decision-making settings and prove convexity for commonly used objective functions in the corresponding settings. 
In these applications, the convexity is proved, but finer structure such as strong convexity often does not hold or is very difficult to prove. In addition, there may be many choices of decision variables whose associated objective values are close to the optimal objective value, and the gap between optimal and sub-optimal solutions is hard to measure or estimate a priori. For the algorithms designed in this work, we take the view that this gap information is not available and the algorithms are designed to work for arbitrarily small unknown gap. 

In this work, we develop provably efficient simulation-optimization algorithms that are guaranteed with an arbitrary probability $1-\delta$ to find a near-optimal choice of the decision variable that renders an objective value  $\epsilon$-close to the optimal solution, where $\epsilon$ is an arbitrary user-specified precision level. This criterion is called \textit{$(\epsilon,\delta)$-Probability of Good Selection} ($(\epsilon,\delta)$-PGS) in the simulation literature; see \cite{ma2017efficient} and \cite{hong2020review}. \revisee{Although the asymptotic regime $\delta \ll1$ is of more interest in many theoretical works, this work provides bounds on the simulation cost that hold for all $\epsilon\geq0$ and $\delta\in(0,1]$.} To quantify the computational cost for the proposed algorithms that are guaranteed to find PGS solutions, we take the view that the simulation cost is the dominant contributor to the computational cost; see also \cite{mahen17b}. The simulation cost of an algorithm is measured as the total number of simulation replications run at all possible decisions visited by the algorithm until it stops. When designing algorithms to solve large-scale discrete optimization via simulation problems, the dependence of the simulation cost on the problem size (or, the number of alternatives/solutions/systems in the area of ranking and selection) is crucial to understand; see also discussions in \cite{zhong2019knockout}.

Three most recent papers \cite{wang2020optimal}, \cite{Nelson2020} and \cite{zhang2020discrete} also discussed the use of the convexity structure in simulation. \cite{wang2020optimal} considered a discrete simulation optimization problem with a specific polynomial functional form for the objective function, and focus on how to strategically use gradient information to accelerate the selection of the best. Their focus and problem settings are different from ours. \cite{Nelson2020} utilized the convexity structure to select a feasible region that contains the optimal given existing simulation samples at different choices; see also \cite{eckman2021flat}. Because they do not consider an optimization problem and their goal is not to find an optimal or near-optimal solution, the focus of \cite{Nelson2020} is different from ours. For example, they do not provide simulation-optimization algorithms that can find an optimal or near-optimal decision, nor do they analyze simulation costs and their dependence on problem scale. On the other hand, the method and analysis provided by \cite{Nelson2020} and \cite{eckman2021flat} can serve effectively as a module to help solve other general simulation problems, such as multi-objective simulation optimization, which is not the focus of our work.

\cite{zhang2020discrete} proposed subgradient descent algorithms for problems with a high-dimension decision space. Roughly speaking, \revisee{their algorithms scale well to high-dimensional problems, but are computationally expensive for large-scale problems.}
However, in practice, many problem settings have a large scale but a low dimension or even a single dimension. For example, large delivery companies often need to decide the total number of trucks that should be recruited for operations in a self-contained region. A service system may needs to decide the total number of staff members needed to host a special event. In our work, our focus is on designing algorithms that work well for large-scale problems.
Furthermore, the subgradient descent algorithms in \cite{zhang2020discrete} require prior knowledge about the upper bounds on the Lipschitz constant $L$ and the variance $\sigma^2$. In addition, the simulation cost of the subgradient descent algorithm has a polynomial dependence on the upper bounds $L$ and $\sigma^2$. For many real-world discrete simulation via optimization problems, the Lipschitz constant and the variance are unknown and hard to estimate. As a result, both upper bounds are likely to be over-estimated, which will lead to worse simulation costs. \revisee{In this work, algorithms that do not rely on prior information about $L$ and $\sigma^2$ are proposed, which solve the aforementioned issues.}

\subsection{Contributions}
\label{sec:contributions}

The major methodology in algorithm design in this paper can be classified as \textit{stochastic localization methods}, in the sense that we ``localize" potentially near-optimal solutions in a subset and adaptively shrink the subset at each step. The design of algorithms relies on and addresses the challenge from the fact that the feasible set is a discrete set. Intuitively, if the feasible set has a finite number of discrete points, the subset of potentially near-optimal solutions can only be shrunk for a finite number of times, and the number of localization operations cannot exceed the size of the feasible set. The proposed algorithms generally do not rely on prior estimates of the Lipschitz constant and the variance. In addition, the simulation cost of achieving the PGS guarantee does not depend on the Lipschitz constant. We note that the dependence on the variance $\sigma^2$ is inevitable. \revisee{To avoid requiring prior knowledge about the variance in the Gaussian case,} after designing algorithms that do not require information about the Lipschitz constant, we propose in the last section before numerical experiments an adaptive scheme to address the challenge of unknown variances. \revisee{The idea of localization also appears in prior literature of discrete optimization via simulation, such as empirical stochastic branch-and-bound \citep{xu2013empirical}, nested partition \citep{shi2000nested} and COMPASS \citep{hong2006discrete,xu2010industrial}. However, existing works do not utilize the convexity structure and do not provide complexity analysis of the proposed algorithms. In contrast, we propose specially-designed algorithms for discrete convex objective functions and provide an estimate of the simulation costs.}


\revisee{To show the usefulness of the localization operation,} we first consider an important case of discrete simulation via optimization problems, where the decision space is the ``one-dimensional'' set $\{1,2,\ldots,N\}$. Here, 
$N$ is an arbitrary positive integer that represents the problem scale. Without the convexity structure, the problem setting is mathematically equivalent to the problem of \textit{ranking and selection}; see \cite{hong2020review} for a comprehensive review. In this work, the objective function is assumed to be discrete convex on the decision space, but no other structure information such as strong convexity or the knowledge of a minimal gap between the optimal and sub-optimal solutions is known. \revisee{Utilizing the idea of \textit{localization}, we overcome the shortcoming of the subgradient descent algorithm that its simulation cost has a quadratic dependence on the problem scale. We propose two localization algorithms.} As a natural generalization of the classical bi-section algorithm, we design the tri-section sampling (TS) algorithm to find a $(\epsilon,\delta)$-PGS solution. We prove that, when $\delta$ is small, $O(\log(N) \epsilon^{-2}\log(1/\delta))$ serves as an upper bound on the simulation cost for the TS algorithm for any one-dimensional convex problem, which represents the same logarithmic dependence on the scale as the bi-section algorithm. Note that when the convexity structure is not exploited, the optimal dependence on $N$ can be linear. We then design the shrinking uniform sampling (SUS) algorithm that beats the TS algorithm. The SUS algorithm is proved to enjoy the upper bound on the simulation cost as $O[ \epsilon^{-2}(\log(N) + \log(1/\delta))]$ when $\delta$ is small.
Using the asymptotic criterion (namely, $\delta\rightarrow0$ with other parameters fixed) in \citet{kaufmann2016complexity}, the SUS algorithm asymptotically achieves the optimal performance and, therefore, is the first algorithm to achieve a matching upper bound on simulation costs for ranking and selection problems with general convex structure. This theoretical superiority of the SUS algorithm is also verified in numerical experiments. We remark that our major contribution is the SUS algorithm rather than the TS algorithm, though the analysis provided for these two algorithms may be separately useful in broader settings.


Next, we turn to the settings of large-scale multi-dimensional problems with the ``$d$-dimensional'' discrete decision space $\{1,2,\ldots,N\} \times \{1,2,\ldots,N\} \times\ldots \times \{1,2,\ldots,N\}$. We note that the scale $N$ can easily be relaxed to be different in each dimension in our algorithm design (e.g., after linear constraints are applied on the decision space), but we unify the use of $N$ in each dimension in the analysis, so as to clearly demonstrate the impact of the scale $N$. A natural definition of discrete convexity on the multi-dimensional decision space is the $L^\natural$-convexity~\citep{murota2003discrete}, which guarantees that a local optimum is globally optimal; see \cite{dyer1977note,freund2017minimizing} for examples of $L^\natural$-convex functions. We observe that even though the TS algorithm and the SUS algorithm designed for one-dimensional problems can be extended to the multi-dimensional case, the dependence of their simulation cost on the dimension $d$ can be large, even up to an exponential order of dependence, which may prohibit their practical use in high-dimensional problems. This motivates us to consider alternative approaches to design stochastic localization algorithms that have a low dependence on the dimension $d$.

\revisee{In this work, we combine the idea of localization with the \textit{subgradient information} in the multi-dimensional case.
The subgradient information is constructed by taking simulation samples and plays a a crucial role in reducing the dependence of simulation cost on the dimension $d$.
The cutting-plane methods~\citep{vaidya1996new,bertsimas2004solving,lee2015faster,jiang2020improved} is based on a similar idea and is known to have lower order or no dependence on the Lipschitz constant. However, the cutting-plane methods are not robust to noise.} Therefore, we develop a novel framework to design stochastic cutting-plane (SCP) algorithms based on deterministic cutting-plane algorithms, with the goal of achieving the PGS guarantee. A novel stochastic separation oracle is designed and analyzed. 
A straightforward application of the proposed framework leads to SCP algorithms that have an $O(d^3)$ dependence on the dimension and a logarithmic dependence on $L$. 

Utilizing the discrete natural of the problem, we further develop the dimension reduction algorithm whose simulation cost is upper bounded by a constant that is independent of $L$ and has an $O(d^4)$ dependence on the dimension. This is the first algorithm for convex discrete optimization via simulation in the literature that does not require the knowledge about the Lipschitz constant $L$. In contrast, the subgradient-based search algorithms developed in \cite{zhang2020discrete} has a higher order dependence on $L$ and requires the knowledge about the Lipschitz constant, although it has a lower dependence ($O(d^2)$) on the dimension compared to the dimension reduction algorithm. Our developed SCP algorithms may particularly be preferable when the Lipschitz parameter $L$ for a given problem is large or hard to estimate. The idea of gradually reducing the problem dimension was proposed in parallel in \cite{jiang2020minimizing}, where the author made the algorithm more practical by reducing the number of arithmetic operations to be polynomial. We numerically verify that the dimension reduction algorithm has a better performance than the subgradient descent algorithm in \cite{zhang2020discrete} both on the synthetic and the queueing simulation optimization examples, especially for the large-scale case.

In terms of dependence on the scale $N$, we theoretically show that the subgradient descent algorithm and the SCP algorithms all present an $O(N^2)$ dependence on $N$ for their simulation costs. However, the SCP algorithms empirically perform better than the subgradient descent algorithm on examples where $N$ is large. On the other hand, the SUS algorithm, when extended to multi-dimensional problems, still present no dependence on $N$ under the asymptotic criterion \citep{kaufmann2016complexity}, but however incurs an exponential dependence on $d$. These analyses can assist practitioners to choose which algorithm to use depending on the knowledge or partial knowledge on $d$, $N$ and $L$ in the specific problems.


We remark that the design of localization algorithms that satisfy the PGS guarantee is the main focus of this paper. If, in addition, for scenarios when extra information on the indifference zone parameter $c>0$ is available, i.e., the gap between the objective function values of the best decision and the second best decision is known, our algorithms can naturally be extended to identify the exact best decision with high probability $1-\delta$. This criterion is referred to as \textit{Probability of Correct Selection with Indifference Zone} (PCS-IZ). We also provide performance analysis for our proposed algorithms in the appendix when they are used to achieve the PCS-IZ criterion.

Finally, we propose a novel algorithm that is able to adaptively estimate the variance of the randomness at each feasible decision in the case when the noise is Gaussian. The design of the algorithm is based on the property that the lower tail for $\chi^2$-random variables is sub-Gaussian \citep{wainwright2019high}. The adaptive algorithm is suitable for the case when an upper bound on the variance is hard to estimate and over-estimation is inevitable. In addition, the adaptive algorithm provides an approach to improve the simulation cost in the case when location-dependent upper bounds of the variance $\sigma_x^2$ is available for all feasible decision $x$. This is because the uniform upper bound $\sigma^2 = \max_x \sigma_x^2$ is in general attained by extreme choices of the decision variable and may be much larger than the variance of a large proportion of feasible decisions. In contrast to common two-stage procedures for the unknown variance case in ranking and selection literature, the proposed adaptive algorithm does not require simulating all choices of the decision variable (which requires $O(N^d)$ simulations) to get an upper bound on the variance. Moreover, using the novel algorithm, the simulation cost is at most increased by a constant factor compared to the known variance case.

The remainder of the paper is outlined as follows. Section \ref{sec:notation} summarizes the notation. Section \ref{sec:modelandframework} introduces the model, framework, optimality criterion, and simulation costs. Section \ref{sec:one-dim} discusses the algorithms and performance analysis developed for one-dimensional large-scale problems. Section \ref{sec:multi-dim} discusses the algorithms and performance analysis developed for multi-dimensional large-scale problems. Section \ref{sec:var} introduces the adaptive algorithm for estimation the variance in the Gaussian case. Section \ref{sec:numerical} provides numerical experiments to compare the proposed algorithms to benchmark methods. Section \ref{sec:cls} gives the concluding remarks.

\subsection{Notation} \label{sec:notation}

For a stochastic system labeled by its decision variable $x$, we denote $\xi_x$ as the random object associated with the decision variable. We write $\xi_{x,1},\xi_{x,2},\ldots,\xi_{x,n}$ as independent and identically distributed (i.i.d.) copies of $\xi_x$. The empirical mean of the $n$ independent evaluations for a decision variable labeled by $x$ is denoted as $\hat{F}_n(x) := \frac{1}{n}\sum_{j=1}^{n} F(x,\xi_{x,j})$. The indices set $[N]:=\{1,2,\dots,N\}$ is defined for every positive integer $N$. For any set $S$ and positive integer $d$, we define the product set $S^d$ as $\{ (x_1,x_2\dots,x_d): x_i\in\mathcal{S},i\in[d] \}$. 
For two vectors $x,y\in\mathbb{R}^d$, the maximum operation $x \vee y$, minimum operation $x \wedge y$, the ceiling function $\lceil x\rceil$ and the flooring function $\lfloor x\rfloor$ are all considered as component-wise operations.
To compare simulation costs, we omit terms that are independent of $d,N,\epsilon,\delta,c$ in $O(\cdot)$ and omit terms independent of $\delta$ in $\tilde{O}(\cdot)$. To be more concrete, the notation $f = O(g)$ means that there exist constants $c_1,c_2>0$ independent of $N,d,\epsilon,\delta,c$ such that $f \leq c_1 g + c_2$. Similarly, the notation $f = \tilde{O}(g)$ means that there exist constants $c_1>0$ independent of $N,d,\epsilon,\delta,c$ and constant $c_2>0$ independent of $\delta$ such that $f \leq c_1 g + c_2$. 
\revise{The notation $f = \Theta(g)$ means that there exist constants $c_1,c_2,c_3>0$ independent of $N,d,\epsilon,\delta,c$ such that $c_3 g\leq f \leq c_1 g + c_2$. The notation $f = \tilde\Theta(g)$ means that there exist constants $c_1,c_3>0$ independent of $N,d,\epsilon,\delta,c$ and constants $c_2,c_4>0$ independent of $\delta$ such that $c_3 g + c_4\leq f \leq c_1 g + c_2$.}

\section{Model and Framework} \label{sec:modelandframework}
We consider a complex stochastic system that involves discrete decision variables in a $d$-dimensional subspace $\mathcal{X} = [N_1]\times [N_2]\times\cdots\times [N_d]$ in which the $N_i$'s are positive integers. The objective function $f(x)$ for $x\in\mathcal{X}$ is given by
\begin{equation*}
    f(x) := \mathbb{E}[F(x, \xi_x)], 
\end{equation*}
in which $\xi_x$ is a random object belongs to probability space $(\mathsf{Y}, \mathcal{B}_{\mathsf{Y}})$ and $F:\mathcal{X} \times \mathsf{Y}\rightarrow \mathbb{R}$ is a measurable function. Specifically, the function $F$ captures the full operations logic in the stochastic system and measures the performance of the system. \revisee{For example, in a queueing system, $\xi_x$ is the arrival times and the service times of customers, and $F(\cdot,\xi_x)$ is the average waiting time of all customers under the situation described by $\xi_x$.} We consider scenarios when the objective function $f(x)$ is not in closed-form and needs to be evaluated by averaging over simulation replications of $F(x,\xi_x)$. The random objects $\xi_x$'s can be different for different choices of decision variables.
In this work, we focus on identifying the optimal decision, i.e., finding the decision that has the minimal objective value:
\begin{equation} \label{eqn:obj}
    \min_{x\in\mathcal{X}}~f(x).
\end{equation}
We assume that the objective function has a convex structure.
\begin{assumption}\label{asp:1}
The objective function $f(x)$ is a convex function on the discrete set $\mathcal{X}$.
\end{assumption}
For the exact definition of discrete convexity, we describe in details in Section \ref{sec:one-dim} for the one-dimensional cases and Section \ref{sec:multi-dim} for the multi-dimensional cases. 

\subsection{Optimality Guarantees and Classes of Algorithms}

Our general goal is to design algorithms that guarantee the selection of a good decision that yields a close-to-optimal performance with high probability. Formally, this criterion is defined as \textit{Probability of Good Selection}.
\begin{itemize}
     \item \textbf{$(\epsilon,\delta)$-Probability of good selection (PGS).} The solution $x$ returned by an algorithm has an objective value at most $\epsilon$ larger than the optimal objective value with probability at least $1-\delta$.
\end{itemize}

This PGS guarantee is also referred to as the probably approximately correct selection (PAC) guarantee in the literature~\citep{even2002pac,kaufmann2016complexity,ma2017efficient}. While our main focus is to design algorithms that satisfy the PGS optimality guarantee, we also consider the optimality guarantee of \textit{Probability of Correct Selection with Indifference Zone} for comparison.
\begin{itemize}
    \item \textbf{Probability of correct selection with indifference zone (PCS-IZ).} (See \cite{hong2020review}) The problem is assumed to have a unique solution that renders the optimal objective value. The optimal objective value is assumed to be at least $c>0$ smaller than the objective values at sub-optimal choices of decisions. The gap width $c$ is called the \textbf{indifference zone parameter} in~\citet{bechhofer1954single}. The PCS-IZ guarantee requires that the solution returned by an algorithm be the optimal solution with probability at least $1-\delta$.
\end{itemize}

In general, by choosing $\epsilon < c$, algorithms satisfying the PGS guarantee can be readily applied to satisfy the PCS-IZ guarantee. However, algorithms satisfying the PCS-IZ guarantee may fail to satisfy the PGS guarantee; see \cite{eckman2018fixed} and \cite{hong2020review}. The failing probability $\delta$ in either PGS or PCS-IZ is usually chosen to be small to ensure a high probability result. \revise{Hence, we assume henceforth that $\delta$ is small enough and focus on the asymptotic expected simulation cost.}
In addition, we assume that the probability distribution for the stochastic simulation output $F(x,\xi_x)$ is sub-Gaussian.
\begin{assumption}\label{asp:3}
The distribution of $F(x,\xi_x) - f(x)$ is zero-mean sub-Gaussian with the known upper bound $\sigma^2$ on the parameter for any $x\in\mathcal{X}$.
\end{assumption}
We note that a special case of Assumption \ref{asp:3} is when the distribution follows the Gaussian distribution. In that case, the parameter $\sigma^2$ can be chosen as the upper bound on the variance of the distribution. For more general distributions with a finite variance, the mean estimator in \citet{lee2020optimal} can be used in place of the empirical mean estimator and the results in this work can be directly generalized. 
We assume that Assumption \ref{asp:3} holds in the remainder of the paper except Section \ref{sec:var}, where we propose a novel algorithm to adaptively estimate the variance $\sigma^2$ in the Gaussian case.
The triad of the decision space $\mathcal{X}$, the space of randomness $(\mathsf{Y},\mathcal{B}_\mathsf{Y})$ and the function $F(\cdot,\cdot)$ is called the \textbf{model} of problem \eqref{eqn:obj}. We define the set of all models for which function $f(\cdot)$ is convex on set $\mathcal{X}$ as $\mathcal{MC}(\mathcal{X})$, or simply $\mathcal{MC}$. The set $\mathcal{MC}_c$ includes all convex models with the indifference zone parameter $c$. Next, we define the class of simulation-optimization algorithms that are proved to find solutions satisfying certain optimality guarantee for a given set of models. 
\begin{definition}
Given an optimality guarantee $\mathcal{O}$ and a set of models $\mathcal{M}$, a simulation-optimization algorithm is called an $(\mathcal{O},\mathcal{M})$-algorithm if, for any model $M\in\mathcal{M}$, the algorithm returns a solution to $M$ that satisfies the optimality guarantee $\mathcal{O}$. 
\end{definition}
For example, the class of $(\text{PGS},\mathcal{MC})$-algorithms guarantees a PGS solution for any convex model. 

\subsection{Simulation Costs}

For \revise{optimization via simulation problems}, the view that the simulation cost of generating replications of $F(x,\xi_x)$ is the dominant contributor to the computational cost is widely hold; see \citet{luo2017fully,nietal17,ma2017efficient,mahen17b}. Therefore, for the purpose of comparing different simulation-optimization algorithms that satisfy certain optimality guarantee, the performance of each algorithm is measured by the total number of evaluations of $F(x,\xi_x)$ at different points $x$. The number of evaluations during an optimization process is called the \textit{simulation cost}.
Besides providing a measure to compare different algorithms, simulation costs can provide insights into how the computational cost depends on the scale and dimension of the problem. Moreover, understanding the simulation costs can provide information to facilitate the setup of parallel procedures for large-scale problems. 
%
%
The main focus of this paper is to develop provably efficient simulation-optimization algorithms for a certain optimality guarantee and provide an upper bound on the simulation cost to achieve that guarantee. We note that our proposed algorithms do not require additional structures of the selection problem in addition to convexity.
%
%
Now, we give the rigorous definition of the expected simulation cost for a given set of models $\mathcal{M}$ and given optimality guarantee $\mathcal{O}$. 
\begin{definition}
Given the optimality guarantee $\mathcal{O}$ and a set of models $\mathcal{M}$, the \textbf{expected simulation cost} is defined as
\[ T(\mathcal{O},\mathcal{M}) := \inf_{A~\text{is }(\mathcal{O},\mathcal{M})} ~ \sup_{M\in\mathcal{M}}  \mathbb{E}\left[\tau_A\right], \]
where $\tau_A$ is a random variable that represents the number of simulation evaluations of $F(\cdot,\cdot)$ for each implementation of algorithm $A$.
\end{definition}
The notion of simulation cost in this paper is largely focused on
\[  \quad T(\epsilon,\delta,\mathcal{MC}) := T((\epsilon,\delta)\text{-}PGS,\mathcal{MC}),\quad T(\delta,\mathcal{MC}_c) := T((c,\delta)\text{-}PCS\text{-}IZ,\mathcal{MC}_c). \] 
%
We mention that the upper bounds derived in this paper also hold almost surely, while the lower bounds only hold in expectation.

To better present the dependence of the expected simulation cost on the scale and dimension of the problem, we assume that $N_1=N_2=\cdots=N_d$.
\begin{assumption}\label{asp:4}
The feasible set of decision variables is $\mathcal{X} = [N]^d$, where $N\geq2$ and $d\geq1$. 
\end{assumption}
%
With Assumption \ref{asp:4} in hand, we will present the dependence of the expected simulation cost on $N$ and $d$. We note that the results in this work can be naturally extended to the case when each dimension has a different number of feasible choices of decision variables. \revise{Furthermore, if the objective function $f$ is defined on a $L^\natural$-convex set (i.e., the indicator function of the set is a $L^\natural$-convex function, which we will define later), the algorithms proposed in this paper can be directly extended with small modifications. A typical example of a $L^\natural$-convex set is the capacity-constrained set
\[ \left\{ (x_1,\dots,x_d)~\Big|~ x_i\in[N],~\forall i\in[d],~ {\textstyle\sum}_i x_i \leq c \right\} \]
under a linear transform, where $c > 0$ is the capacity constraint; see Section \ref{sec:numerical} for more details.
}

\section{Simulation-optimization Algorithms and Complexity Analysis: One-dimensional Case}
\label{sec:one-dim}

We first consider a special class of optimization via simulation problems where the dimension of the decision variable is one, but there are a large number of choices of decision variable. This class of one-dimensional problems, despite of the less generality compared to multi-dimensional large-scale problems, have applications when the one-dimensional decision variable is a choice of overall resource level. For example, large delivery companies often need to decide the total number of trucks that should be recruited for operations in a self-contained region. A service system may needs to decide the total number of staff members needed to host a special event. Such decisions often involve a trade-off between service satisfaction and resource costs. The convexity in the objective function often comes from the marginal decay of contribution to service satisfaction as the resource level increase; see the optimal allocation example and Figure \ref{fig:landscape} in Section \ref{sec:numerical} for more details.

In the one-dimensional case, the feasible set is $\mathcal{X}=[N]=\{1,2,\ldots,N\}$. This setting is mathematically equivalent to the problem of ranking and selection with convexity structure. The discrete convexity for a function $f$ can be defined similarly to the ordinary continuous convexity through the discrete midpoint convexity property, namely,
\[ f(x+1) + f(x-1) \geq 2 f(x),\quad\forall x \in \{2,\dots,N-1\}. \]
If the function $f(x)$ is convex on $\mathcal{X}$, it has a convex linear interpolation on the continuous interval $[1,N]$, defined as
\begin{align}\label{eqn:linear-int} \tilde{f}(x) := [f(x_0+1)-f(x_0)] \cdot(x-x_0) + f(x_0),\quad\forall x\in[x_0,x_{0}+1],~ x_0\in[N-1]. \end{align}
The definition of discrete convexity in a multi-dimensional decision space is called the $L^\natural$-convexity~\citep{murota2003discrete}. We defer the discussion of $L^\natural$-convex functions for the multi-dimensional case to Section \ref{sec:multi-dim}.

In this section, we propose simulation-optimization algorithms that are guaranteed to find solutions that satisfy the PGS guarantee, provided that the objective function has a convex structure. 
For every developed simulation-optimization algorithm, we provide an upper bound on the expected simulation cost to achieve the PGS guarantee. 
We also provide a lower bound on the expected simulation cost that reflects the best achievable performance for any algorithm. Under the asymptotic criterion in \citet{kaufmann2016complexity}, one of our proposed algorithms can attain the best achievable asymptotic performance.

In contrast to the multi-dimensional case, where the subgradient descent algorithm achieves satisfying performance \citep{zhang2020discrete}, the subgradient descent algorithm is not efficient for large-scale one-dimensional problems. This is because of the $O(N^2)$ dependence in the simulation cost. In addition, the subgradient descent algorithm relies on the Lipschitz constant of the objective function, which is shown to be unnecessary for discrete problems in this section. Utilizing the localization operation, the algorithms proposed in this section do not have the aforementioned issues. Therefore, the algorithms in this section provide better alternatives to the subgradient descent algorithm for one-dimensional problems. The analysis of the one-dimensional case also shows the limitation of subgradient-based search methods and provides a hint on how to improve algorithms for multi-dimensional problems.

\subsection{Tri-section Sampling Algorithm and Upper Bound on Expected Simulation Cost}\label{sec:adaptive}

We first propose the tri-section sampling algorithm for the PGS guarantee. The idea of the tri-section sampling algorithm is from the classical bi-section method and the golden section method. 
%
A similar tri-section sampling algorithm is proposed in \cite{agarwal2011stochastic} for stochastic continuous convex optimization, which controls the regret instead of the objective value.
However, 
their algorithm does not utilize the prior information that the optimal solution is an integral point and thus the simulation cost has a polynomial dependence on the Lipschitz constant. In addition, although an algorithm that minimizes the regret can be used to minimize the objective function value, the resulting simulation cost may be larger than that of specialized optimization algorithms and has an inferior dependence on the dimension $d$ in the multi-dimensional case.
%
The pseudo-code of the proposed tri-section sampling algorithm is listed in Algorithm \ref{alg:one-dim}. 
%
\bigskip
\begin{breakablealgorithm}
\caption{Tri-section sampling algorithm for the PGS guarantee}
\label{alg:one-dim}
\begin{algorithmic}[1]
\Require{Model $\mathcal{X}=[N], (\mathsf{Y},\mathcal{B}_\mathsf{Y}),F(x,\xi_x)$, optimality guarantee parameters $\epsilon,\delta$.}
\Ensure{An $(\epsilon,\delta)$-PGS solution $x^*$ to problem \eqref{eqn:obj}.}
\State Set upper and lower bounds of the current interval $\revise{x_L}\leftarrow 1,\revise{x_U}\leftarrow N$. 
\State Set maximal number of comparisons $T_{max}\leftarrow\log_{1.5}(N) + 2$.
\While{$\revise{x_U} - \revise{x_L} > 2$} \Comment{Iterate until there are at most $3$ points.}
    \State Compute $3$-quantiles of the interval $\revise{q_{1/3}}\leftarrow \lfloor 2\revise{x_L}/3+\revise{x_U}/3\rfloor$ and $\revise{q_{2/3}}\leftarrow \lceil \revise{x_L}/3+2\revise{x_U}/3\rceil$. 
    \Repeat{ simulate an independent copy of $F(\revise{q_{1/3}},\xi_{1/3})$ and an independent copy of  $F(\revise{q_{2/3}},\xi_{2/3})$  } 
    \State Compute the empirical mean (using all of the simulated samples) $\hat{F}_n(\revise{q_{1/3}}),\hat{F}_n(\revise{q_{2/3}})$.
    \State Compute $1-\delta/(2T_{max})$ confidence intervals at each quantile:
    \[ \left[\hat{F}_n(\revise{q_{1/3}})-h_{1/3},\hat{F}_n(\revise{q_{1/3}})+h_{1/3}\right]\quad \text{ and }\quad \left[\hat{F}_n(\revise{q_{2/3}})-h_{2/3},\hat{F}_n(\revise{q_{2/3}})+h_{2/3}\right], \]
    \hspace{3.5em}where $h_{1/3}$ and $h_{2/3}$ are the half-widths of confidence intervals.
    \Statex \Comment{A possible choice of $h_{1/3}$ and $h_{2/3}$ is $h(n,\sigma,\alpha)$ defined below.}
    \Until{the first time that one of the following three conditions holds:
    \begin{align*}  (i)\quad& \hat{F}_n(\revise{q_{1/3}})-h_{1/3}\geq\hat{F}_n(\revise{q_{2/3}})+h_{2/3},\\
                    (ii)\quad& \hat{F}_n(\revise{q_{1/3}})+h_{1/3}\leq\hat{F}_n(\revise{q_{2/3}})-h_{2/3},\\
                    (iii)\quad& h_{1/3} \leq \epsilon/8\quad\text{ and }\quad h_{2/3} \leq \epsilon/8. \end{align*}}
    \If{$\hat{F}_n(\revise{q_{1/3}})-h_{1/3}\geq\hat{F}_n(\revise{q_{2/3}})+h_{2/3}$}
        \State Update $\revise{x_L}\leftarrow \revise{q_{1/3}}$.
    \ElsIf{$\hat{F}_n(\revise{q_{1/3}})+h_{1/3}\leq\hat{F}_n(\revise{q_{2/3}})-h_{2/3}$}
        \State Update $\revise{x_U}\leftarrow \revise{q_{2/3}}$.
    \ElsIf{$h_{1/3} \leq \epsilon/8 $ and $h_{2/3} \leq \epsilon/8$}
        \State Update $\revise{x_L}\leftarrow \revise{q_{1/3}}$ and $\revise{x_U}\leftarrow \revise{q_{2/3}}$.
    \EndIf
\EndWhile
\State Simulate $F(x,\xi_x)$ for $x\in\{\revise{x_L},\dots,\revise{x_U}\}$ until the $1-\delta/(2T_{max})$ confidence half-widths are smaller than $\epsilon/2$.
\Statex\Comment{Now $\revise{x_U}-\revise{x_L}\leq 2$.}
\State Return the point in $\{\revise{x_L},\dots,\revise{x_U}\}$ with the minimal empirical mean.
\end{algorithmic}
\end{breakablealgorithm}
\bigskip
In the procedure of Algorithm \ref{alg:one-dim}, one step is to compute confidence intervals that satisfy certain confidence guarantees. We now provide one feasible approach to construct such confidence intervals, which is based on Hoeffding's inequality for sub-Gaussian random variables. Define 
\begin{equation*}
    h(n,\sigma,\alpha) := \sqrt{\frac{2\sigma^2}{n} \cdot \log(2/\alpha)}.
\end{equation*}
Recall that $\sigma^2$ is the upper bound on the sub-Gaussian parameters of all choices of decision variables. With this function $h(\cdot)$ in hand, whenever $n$ independent simulations of the decision $x$ are available, one can construct a $(1-\alpha)$ confidence interval for $f(x)$ as
\begin{equation*}
     \left[\hat{F}_n(x)-h(n,\sigma,\alpha),\hat{F}_n(x)+h(n,\sigma,\alpha)\right].
\end{equation*}
If the variance $\sigma_x^2$ of a single choice of decision variable $x$ is known, the confidence interval may be sharpened by replacing $\sigma$ with $\sigma_x$; see Section \ref{sec:var}. 
\revise{We note that the analysis in this work can be generalized to more general distributions, such as the sub-exponential distributions, by replacing $h(n,\sigma,\alpha)$ with other concentration bounds.}

Intuitively, the algorithm iteratively shrinks the size of the set containing a potentially near-optimal choice of decision variables. Specifically, the algorithm shrinks the length of the current interval by at least $1/3$ for each iteration. Thus, the total number of iterations is at most $O(\log_{1.5}(N))$ to shrink the set until there are at most $3$ points. Then, the algorithm solves a sub-problem with at most $3$ points. We can prove that Algorithm \ref{alg:one-dim} achieves the PGS guarantee for any given convex problem without knowing further structural information, i.e., Algorithm \ref{alg:one-dim} is a $[(\epsilon,\delta)$-PGS$,\mathcal{MC}]$-algorithm. By estimating the simulation cost of the algorithm, an upper bound on the expected simulation cost to achieve the PGS guarantee follows.

\revise{
\begin{theorem}\label{thm:one-dim}
Suppose that Assumptions \ref{asp:1}-\ref{asp:4} hold. Algorithm \ref{alg:one-dim} is a $[(\epsilon,\delta)$-PGS$,\mathcal{MC}]$-algorithm. Furthermore, we have
\[ T(\epsilon,\delta,\mathcal{MC}) = O\left[ \frac{\log(N)}{\epsilon^2}\log\left( \frac{\log(N)}{\delta} \right) + \log(N) \right] = \tilde{O}\left[ \frac{\log(N)}{\epsilon^2}\log\left( \frac{1}{\delta} \right) \right]. \]
%
\end{theorem}
}
\begin{proof}{Proof of Theorem \ref{thm:one-dim}.}
The proof is provided in \ref{ec:one-dim}.
\hfill\Halmos\end{proof}
%

We provide an explanation on the additional $\log(N)$ term. 
We note that in practice, the number of simulation samples taken in each iteration must be an integer, while the simulation cost is treated as a real number in our complexity analysis. Hence, the practical simulation cost of each iteration should be the smallest integer larger than the theoretical simulation cost, which introduces an extra $O(1)$ term. Then, the total expected simulation cost of Algorithm \ref{alg:one-dim} should contain an extra $O(\log(N))$ term, which is not related to $\delta$ and is relatively small compared to the main term when $\delta$ is small.
%
\begin{remark}
The term in the $\tilde{O}(\cdot)$ notation reflects the asymptotic simulation cost when $\delta\rightarrow 0$. The asymptotic simulation cost is commonly used in multi-armed bandits literature to compare the computational complexities of different algorithms \citep{lai1985asymptotically,burnetas1996optimal,karnin2013almost,jamieson2014lil,chen2016pure,kaufmann2016complexity}. In practice, the failing probability $\delta$ is usually not small enough to enter the asymptotic regime and thus the simulation cost of algorithms may deviate from the asymptotic simulation cost. Therefore, we provide both the non-asymptotic and the asymptotic simulation costs for all algorithms.
\end{remark}

\subsection{Shrinking Uniform Sampling Algorithm and Upper Bound on Expected Simulation Cost} 
\label{sec:one-dim-uniform}

We have shown that the expected simulation cost of tri-section sampling algorithm for the PGS guarantee has a $\log(N)$ dependence on $N$. Then, one may naturally ask: is there any algorithm for the PGS guarantee whose simulation cost has a better dependence on $N$? The answer is affirmative. In this subsection, the shrinking uniform sampling algorithm for the PGS guarantee is proposed, which is proven to have a simulation cost as $O[ \epsilon^{-2}(\log(N) + \log(1/\delta))]$, which grows as $\epsilon^{-2}\log(1/\delta)$ in the asymptotic regime $\delta \rightarrow 0$. Similarly, utilizing the idea of localization, the shrinking uniform sampling algorithm maintains a set of active points and shrinks the set in each iteration until there are at most $2$ points. However, instead of only sampling at $3$-quantiles points of the current interval, the shrinking uniform sampling algorithm samples all points in the current active set but with much fewer simulations. We give the pseudo-code in Algorithm \ref{alg:one-dim-uni}.
\bigskip
\begin{breakablealgorithm}
\caption{Shrinking uniform sampling algorithm for the PGS guarantee}
\label{alg:one-dim-uni}
\begin{algorithmic}[1]
\Require{Model $\mathcal{X}=[N],(\mathsf{Y},\mathcal{B}_\mathsf{Y}),F(x,\xi_x)$, optimality guarantee parameters $\epsilon,\delta$.}
\Ensure{An $(\epsilon,\delta)$-PGS solution $x^*$ to problem \eqref{eqn:obj}.}
\State Set the active set $\mathcal{S}\leftarrow \mathcal{X}$.
\State Set the step size $s\leftarrow1$, maximal number of comparisons $T_{max}\leftarrow N$.
\While{the size of $\mathcal{S}$ is at least $3$} \Comment{Iterate until $\mathcal{S}$ has at most $2$ points.}
    \Repeat{ simulate an independent copy of $F(x,\xi_x)$ for all $x\in \mathcal{S}$ }
        \State Compute the empirical mean (using all of the available simulated samples) $\hat{F}_n(x)$ for all $x\in\mathcal{S}$.
        \State Compute $1-\delta/(2T_{max})$ confidence intervals at each point:
        \[ [\hat{F}_{n_x}(x) - h_{x},\hat{F}_{n_x}(x) + h_{x}],\quad \forall x\in S.\]
        \If{the confidence half-width $h_x\leq |\mathcal{S}|\cdot\epsilon/80$ for some $x\in S$}
            \State Stop sampling point $x$.
        \EndIf
    \Until{at least one of the following conditions holds
    \begin{align*}
        (i)\quad&\exists x,y\in\mathcal{S}\quad\mathrm{s.t.}\quad \hat{F}_{n_x}(x) + h_{x} \leq \hat{F}_{n_y}(y) - h_{y}\\
        (ii)\quad&\forall x\in\mathcal{S} \hspace{4.5em} h_x \leq |\mathcal{S}|\cdot \epsilon/80.
    \end{align*}
    }
    \If{$\hat{F}_{n_x}(x) + h_{x} \leq \hat{F}_{n_y}(y) - h_{y}$ for some $x,y\in\mathcal{S}$} \Comment{\textbf{Type-I Operation}}
        \If{$x < y$}
            \State Remove all points $z\in\mathcal{S}$ with the property $z\geq y$ from $\mathcal{S}$.
        \Else
            \State Remove all points $z\in\mathcal{S}$ with the property $z\leq y$ from $\mathcal{S}$.
        \EndIf
    \ElsIf{$h_x \leq |\mathcal{S}|\cdot \epsilon/80$ for all $x\in\mathcal{S}$}\Comment{\textbf{Type-II Operation}} 
        \State Update the step size $s\leftarrow 2s$. 
        \State Update $\mathcal{S}\leftarrow\{x_{min},x_{min}+s,\dots,x_{min}+ks \}$, where $x_{min}=\min_{x\in \mathcal{S}}~x$ and $k = \lceil|\mathcal{S}|/2\rceil-1$.
    \EndIf
\EndWhile \Comment{Now $\mathcal{S}$ has at most $2$ points.}
\State Simulate $F(x,\xi_x)$ for $x\in\mathcal{S}$ until the $1-\delta/(2T_{max})$ confidence half-widths are smaller than $\epsilon/4$.
\State Return the point in $\mathcal{S}$ with minimal empirical mean.
\end{algorithmic}
\end{breakablealgorithm}
\bigskip
There are two kinds of shrinkage operations in Algorithm \ref{alg:one-dim-uni}, which we denote as Type-I and Type-II Operations. Intuitively, Type-I Operations are implemented when we can compare and differentiate the function values of two points with high probability, and Type-II Operations are implemented when all points have similar function values. In the latter case, we prove that there exists a neighboring point to the optimum that has a function value at most $\epsilon/2$ larger than the optimum. Hence, we can discard every other point in $\mathcal{S}$ (the set in the algorithm that contains a potential good selection) with at least one $\epsilon/2$-optimal point remained in the active set. We give a rough estimate to the expected simulation cost of Algorithm \ref{alg:one-dim-uni}. We assign an order to points in $\mathcal{X}$ by the time they are discarded from $\mathcal{S}$. Points discarded in the same iteration are ordered randomly. Then, for the last $k$-th discarded point $x_k$, there are at least $k$ points in $\mathcal{S}$ when $x_k$ is discarded. By the second termination condition in \revise{line 17}, the confidence half-width at $x_k$ is at least $k\epsilon/80$. If the Hoeffding bound is used, simulating $\tilde{O}(\epsilon^{-2}k^{-2}\log(1/\delta))$ times is enough to achieve the confidence half-width. Recalling the fact that $\sum_k~k^{-2} < \pi^2/6 <\infty$, if we sum the simulation cost over $k\in[N]$, the total expected simulation cost is bounded by $\tilde{O}(\epsilon^{-2}\log(1/\delta))$ and is independent of $N$. The following theorem proves that Algorithm \ref{alg:one-dim-uni} indeed achieves the PGS guarantee for any convex problem
and provides a rigorous upper bound on the expected simulation cost $T(\epsilon,\delta,\mathcal{MC})$.

\begin{theorem}\label{thm:one-dim-uni}
Suppose that Assumptions \ref{asp:1}-\ref{asp:4} hold. Algorithm \ref{alg:one-dim-uni} is a  $[(\epsilon,\delta)\text{-PGS},\mathcal{MC})]$-algorithm. Furthermore, we have
\[ T(\epsilon,\delta,\mathcal{MC}) = O\left[ \frac{1}{\epsilon^{2}}\log\left(\frac{N}{\delta}\right) + N \right] = \tilde{O}\left[ \frac{1}{\epsilon^{2}}\log\left(\frac{1}{\delta}\right) \right]. \]
%
\end{theorem}
\begin{proof}{Proof of Theorem \ref{thm:one-dim-uni}.}
The proof is provided in \ref{ec:one-dim-uni}.
\hfill\Halmos\end{proof}

If we consider the asymptotic regime $\delta\ll 1$ (which is considered in \citet{kaufmann2016complexity}), the expected simulation cost of the shrinking uniform sampling algorithm grows as $\epsilon^{-2}\log(1/\delta)$. This independence is asymptotic and holds in the sense that the required failing probability $\delta$ tends to be very small. When $\delta$ is moderately large, the cost can depend on $N$. We demonstrate in the numerical experiments this asymptotic independence.

\subsection{Lower Bound on Expected Simulation Cost}

In this subsection, we derive lower bounds on the expected simulation costs for all of the simulation-optimization algorithms that satisfy certain optimality guarantee for general convex problems. 
The lower bounds show the fundamental limit behind the simulation-optimization algorithms for general selection problems with a convex structure. In the one-dimensional case, the derived lower bound for the PCS-IZ guarantee also holds for the PGS guarantee by choosing $c=2\epsilon$. By comparing those lower bound with the upper bounds established for specific simulation-optimization algorithms, we can conclude that the shrinking uniform sampling algorithm is optimal up to a constant factor. In the proof for the lower bound result, we construct two convex models that have similar distributions at each point but have distinct optimal solutions. Then, the information-theoretical inequality in~\citet{kaufmann2016complexity} can be used to provide a lower bound on the simulation costs for all algorithms.

We first present the results in~\citet{kaufmann2016complexity} for completeness. Given a simulation-optimization algorithm and a model $M$, we define random variable $N_x(\tau)$ to be the number of times that $F(x,\xi_x)$ is sampled when the algorithm terminates, where $\tau$ is the stopping time of the algorithm. Then, it follows from the definition that
\[ \mathbb{E}_{M}[\tau] = \sum_{x\in\mathcal{X}}\mathbb{E}_{M}\left[N_x(\tau)\right], \]
where $\mathbb{E}_M$ is the expectation when the model $M$ is given. Similarly, we can define $\mathbb{P}_{M}$ as the probability when the model $M$ is given. \revise{We denote the filtration up to the stopping time $\tau$ as $\mathcal{F}_\tau$.} The following lemma is proved in~\citet{kaufmann2016complexity} and is the major tool for deriving lower bounds in this paper.
\begin{lemma}[\citet{kaufmann2016complexity}]
For any two models $M_1,M_2$ and any event $\mathcal{E}\in\mathcal{F}_{\tau}$, we have
\begin{align}\label{eqn:one-dim-low1} \sum_{x\in\mathcal{X}}\mathbb{E}_{M_1}\left[N_x(\tau)\right]\mathrm{KL}(\nu_{1,x},\nu_{2,x}) \geq d(\mathbb{P}_{M_1}(\mathcal{E}),\mathbb{P}_{M_2}(\mathcal{E})), \end{align}
where $d(x,y):=x\log(x/y)+(1-x)\log((1-x)/(1-y))$, $\mathrm{KL}(\cdot,\cdot)$ is the KL divergence, and $\nu_{k,x}$ is the distribution of model $M_k$ at point $x$ for $k=1,2$.
\end{lemma}

We first give a lower bound for the PCS-IZ guarantee.
\begin{theorem}\label{thm:one-dim-low}
Suppose that Assumptions \ref{asp:1}-\ref{asp:4} hold. We have 
\[  T(\delta,\mathcal{MC}_{c}) \geq \Theta\left[ \frac{1}{c^{2}}\log\left(\frac{1}{\delta}\right) \right]. \]
\end{theorem}
\begin{proof}{Proof of Theorem \ref{thm:one-dim-low}.}
The proof is provided in \ref{ec:one-dim-low}.
\hfill\Halmos\end{proof}
The lower bound on $T(\epsilon,\delta,\mathcal{MC})$, i.e., the expected simulation cost for achieving the PGS guarantee, can be derived in a similar way by substituting $c$ with $2\epsilon$ in the construction of two models.
\begin{corollary}
Suppose that Assumptions \ref{asp:1}-\ref{asp:4} hold. We have
\[  T(\epsilon,\delta,\mathcal{MC}) \geq \Theta\left[ \frac{1}{\epsilon^{2}}\log\left(\frac{1}{\delta}\right) \right]. \]
\end{corollary}

Combining with the upper bounds derived in Sections \ref{sec:adaptive}, \ref{sec:one-dim-uniform} and \ref{ec:pcsiz}, we conclude that the tri-section sampling algorithm is optimal up to a constant for the PCS-IZ guarantee, while having a $\log(N)$ order gap for the PGS guarantee. On the other hand, the shrinking uniform sampling algorithm is optimal for both guarantees up to a constant in the asymptotic regime $\delta\ll 1$. However, the space complexities of the tri-section sampling algorithm and the shrinking uniform sampling algorithms are $O(\log(N))$ and $O(N)$, respectively. This observation implies that the tri-section sampling algorithm is preferred for the PCS-IZ guarantee, while for the PGS guarantee we need to consider the trade-off between the simulation cost and the space complexity when choosing algorithms. 

Before concluding this section, we note that the subgradient descent algorithm in \citet{zhang2020discrete} requires the knowledge of the Lipschitz constant and has a simulation cost as $\tilde{O}(N^2\epsilon^{-2}\log(1/\delta))$, which is $O(N^2)$ larger than that of the TS and the SUS algorithms. This observation implies that subgradient-based search methods may not be able to fully utilize the discrete nature and the convex structure of problem \eqref{eqn:obj}, especially for low-dimensional problems. Therefore, the proposed algorithms in this section provide a non-trivial improvement for solving one-dimensional convex optimization via simulation problems and hint a potential improvement direction (namely, localization-based methods) for multi-dimensional problems.

\section{\revise{Simulation-optimization Algorithms} and Complexity Analysis: Multi-dimensional Case}
\label{sec:multi-dim}

In this section, we propose simulation-optimization algorithms to achieve the PGS guarantee for \revise{convex discrete optimization via simulation} problems with multi-dimensional decision variables. The decision space is considered as $\mathcal{X} = [N]^d$. 
In the multi-dimensional case, the discrete convexity of $f$ is defined by the so-called $L^\natural$-convexity, which is defined by the mid-point convexity for discrete variables. The exact definition will be given in Section \ref{sec:Lnatural}. The $L^\natural$-convexity can lead to the property that the discrete convex function has a convex extension along with an explicit subgradient defined on the convex hull of $\mathcal{X}$. 

We outline the intuition underlying the algorithm design of this section before discussing the details. \revisee{Since we have observed the power of localization from the one-dimensional case, the major approach is to design multi-dimensional algorithms based on the same idea.}
The first idea of applying the localization technique is to extend the tri-section sampling algorithm to the multi-dimensional case. A direct generalization of the tri-section sampling algorithm results in the zeroth-order stochastic ellipsoid method~\citep{agarwal2011stochastic} and the zeroth-order random walk method~\citep{liang2014zeroth}, whose computational complexities have $O(d^{33})$ and $O(d^{14})$ dependence on the dimension, respectively. On the other hand, we show that the shrinking uniform sampling method can be naturally extended to the multi-dimensional case. The multi-dimensional shrinking uniform sampling algorithm also has an expected simulation cost independent of the scale $N$ using the asymptotic criterion in \citet{kaufmann2016complexity} (i.e., when $\delta$ is sufficiently small). However, the expected simulation cost has an exponential dependence on the dimension $d$ and, therefore, the shrinking uniform sampling algorithm is only suitable for low-dimensional problems. 

\revisee{We thus take an alternative approach and combine the localization operation with the \textit{subgradient information}, which is known to be useful for high-dimensional problems.} In this work, we design stochastic cutting-plane methods, which utilize properties of $L^\natural$-convex functions and the \lovasz extension to evaluate unbiased stochastic subgradients at each point via finite difference.
More specifically, we develop a new framework to design stochastic cutting-plane methods and thus reduce the dependence of the simulation cost on $d$. A straightforward application our proposed framework leads to stochastic cutting-plane methods whose simulation cost has a $O(d^3)$ dependence on $d$. In addition, the stochastic cutting-plane methods have only a logarithmic dependence on the Lipschitz constant $L$, while the gradient-based method in~\citet{zhang2020discrete} has a higher-order dependence on $L$. 
Further utilizing the \textit{discrete nature} of problem \eqref{eqn:obj}, we develop the dimension reduction algorithm,
whose simulation cost is upper bounded by a constant that is independent of the Lipschitz constant. In addition, the dimension reduction algorithm does not require any prior knowledge about the Lipschitz constant, which makes it suitable for the case when prior knowledge about the objective function is limited.

\subsection{Discrete Convex Functions in Multi-dimensional Space}
\label{sec:Lnatural}

Similar to the one-dimensional case, discrete convex functions in multi-dimensional space are characterized by the discrete midpoint convexity property and have a convex piecewise linear extension. In~\citet{murota2003discrete}, this collection of functions is named $L^{\natural}$-convex functions and is proved to have the property that local optimality implies global optimality.  The exact definition of $L^\natural$-convex functions is given below.
\begin{definition}
A function $f(x):\mathcal{X}\mapsto\mathbb{R}$ is called a $L^\natural$-convex function if the discrete midpoint convexity property holds:
\[ f(x)+f(y)\geq f(\lceil (x+y)/2\rceil) + f(\lfloor (x+y)/2\rfloor),\quad \forall x,y\in\mathcal{X}. \]

The set of models such that $f(x)$ is $L^\natural$-convex on $\mathcal{X}$ is denoted as $\mathcal{MC}$. The set of models such that $f(x)$ is $L^\natural$-convex with an indifference zone parameter $c$ is denoted as $\mathcal{MC}_c$.
\end{definition}
\begin{remark}
As noted in~\citet{zhang2020discrete}, the feasible set $\mathcal{X}=[N]^d$ is a $L^\natural$-convex set and thus we only need the discrete midpoint convexity property to define $L^\natural$-convex functions on $\mathcal{X}$. In the case when $d=1$, the $L^\natural$-convexity is equivalent to the discrete convexity defined in Section \ref{sec:one-dim}. Hence, the definitions of $\mathcal{MC}$ and $\mathcal{MC}_c$ are consistent with Section \ref{sec:one-dim}. 
\end{remark}
%
%
The following property shows that $L^\natural$-convex functions can be viewed as a generalization of submodular functions.
\begin{lemma}[\citet{murota2003discrete}]\label{thm:l-cvx}
Suppose that the function $f(x):\mathcal{X}\mapsto\mathbb{R}$ is $L^\natural$-convex. Then, the translation submodularity holds:
\[ f(x) + f(y) \geq f( (x-\alpha\mathbf{1}) \vee y ) + f( x \wedge (y+\alpha\mathbf{1}) ),\quad \forall x,y\in\mathcal{X},~\alpha\in\mathbb{N}~\mathrm{s.t.}~(x-\alpha\mathbf{1}) \vee y,~ x \wedge (y+\alpha\mathbf{1})\in\mathcal{X}. \]
%
\end{lemma}

By the translation submodularity, the $L^\natural$-convex function restricted to a cube $x+\{0,1\}^d\subset\mathcal{X}$ is a submodular function. Therefore, the \lovasz extension~\citep{lovasz1983submodular} can be constructed as the convex piecewise linear extension inside each cube. In addition, $L^\natural$-convex functions are integrally convex functions~\citep{murota2003discrete}. Hence, we can obtain a continuous convex function on $[1,N]^d$ by piecing together the \lovasz extension in each cube. More importantly, we can calculate a subgradient of the convex extension with $O(d)$ function value evaluations. Hence, $L^\natural$-convex functions provide a good framework for extending the continuous convex optimization theory to the discrete case. 
%
%
%
%
In the remainder of this subsection, we specify this intuition of $L^\natural$-convex functions in a rigorous way. We first define the \lovasz extension of submodular functions and give an explicit subgradient of the \lovasz extension at each point. 
\begin{definition}
Suppose that $f(x):\{0,1\}^d\mapsto \mathbb{R}$ is a submodular function. For \revise{any $x=(x_1,\dots,x_d)\in[0,1]^d$}, we say that a permutation $\alpha_x:[d]\mapsto[d]$ is a \textbf{consistent permutation} of $x$, if
\[ x_{\alpha_x(1)} \geq x_{\alpha_x(2)} \geq \cdots \geq x_{\alpha_x(d)}. \]
\revise{We define $S^{x,0} := (0,\dots,0)\in\mathcal{X}$. For each $i\in[d]$, the \textbf{$i$-th neighbouring points} of $x$ is defined as
\[ S^{x,i} := \sum_{j=1}^i~e_{\alpha_x(j)} \in \mathcal{X}, \]
where vector $e_k$ is the $k$-th vector in the standard basis of $\mathbb{R}^d$.} We define the \textbf{\lovasz extension} $\tilde{f}(x):[0,1]^d\mapsto\mathbb{R}$ as
\begin{align}\label{eqn:submodular} \tilde{f}(x) := f\left(S^{x,0}\right) + \sum_{i=1}^d\left[f\left(S^{x,i}\right) - f\left(S^{x,i-1}\right)\right]x_{\alpha_{x}(i)}. \end{align}
\end{definition}
We note that the value of the \lovasz extension does not rely on the choice of the consistent permutation. We list several well-known properties of the \lovasz extension and refer their proofs to~\citet{lovasz1983submodular,fujishige2005submodular}. 
\begin{lemma}\label{lem:submodular}
Suppose that Assumptions \ref{asp:1}-\ref{asp:4} hold. Then, the following properties hold for $\tilde{f}(x)$:
\begin{itemize}
    \item[(i)] For any $x\in \mathcal{X}$, it holds that $\tilde{f}(x) = f(x)$.
    \item[(ii)] The minimizers of $\tilde{f}(x)$ satisfy $\argmin_{x\in[0,1]^d}~\tilde{f}(x) = \argmin_{x\in \mathcal{X}}~{f}(x)$.
    \item[(iii)] The function $\tilde{f}(x)$ is a convex function on $[0,1]^d$.
    \item[(iv)] A subgradient $g\in\partial\tilde{f}(x)$ is given by
    \begin{align}\label{eqn:subgrad} g_{\alpha_x(i)} := f\left(S^{x,i}\right) - f\left(S^{x,i-1}\right),\quad\forall i\in[d]. \end{align}
\end{itemize}
\end{lemma}
%
It is proved in~\citet{zhang2020discrete} that, with $2d$ simulation runs, we can generate a stochastic subgradient at point $x$ by
\begin{align}\label{eqn:stochastic-subgrad} \hat{g}_{\alpha_x(i)} := F\left(S^{x,i},\xi^1_i\right) - F\left(S^{x,i-1},\xi_{i-1}^2\right),\quad\forall i\in[d]. \end{align}
%
%
Then, we show that the \lovasz extension in the neighborhood of each point can be pieced together to form a convex function on $\mathrm{conv}(\mathcal{X}) = [1,N]^d$. We define the local neighborhood of each point $y\in[1,N-1]^d$ as the cube
\[ \mathcal{C}_y := y + [0, 1]^d. \]
We denote the objective function $f(x)$ restricted to $\mathcal{C}_y \cap \mathcal{X}$ as $f_y(x)$, which is submodular by the translation submodularity of $f(x)$.
%
%
For point $x\in\mathcal{C}_y$, we denote $\alpha_x$ as a consistent permutation of $x-y$ in $[0,1]^d$ and, for each $i\in\{0,1,\dots,d\}$, the corresponding $i$-th neighboring point of $x$ is defined as
\[ S^{x,i} := y + \sum_{j=1}^i~e_{\alpha_x(j)}. \]
Then, the \lovasz extension of $f_y(x)$ in $\mathcal{C}_y$ can be calculated as
\begin{align*} \tilde{f}_y(x) := f\left(S^{x,0}\right) + \sum_{i=1}^d\left[f\left(S^{x,i}\right) - f\left(S^{x,i-1}\right)\right]x_{\alpha_{x}(i)}. \end{align*}
Now, we piece together the \lovasz extension in each cube by defining
\begin{align}\label{eqn:convex-extension} \tilde{f}(x) := \tilde{f}_y(x),\quad \forall x\in[1,N]^d,~y\in[N-1]^d\quad\mathrm{s.t.}~ x\in \mathcal{C}_y. \end{align}
It is proved in~\citet{murota2003discrete} and \citet{zhang2020discrete} that $\tilde{f}(x)$ is well-defined and is a convex function.
\revise{\begin{lemma}\label{thm:consistent}
The function $\tilde{f}(x)$ in \eqref{eqn:convex-extension} is well-defined and convex on $\mathcal{X}$.
\end{lemma}}
Utilizing properties (i) and (ii) of Lemma \ref{lem:submodular}, problem \eqref{eqn:obj} is equivalent to the relaxed problem
\begin{align}\label{eqn:obj-relax} f^* := \min_{x\in[1,N]^d}~\tilde{f}(x), \end{align}
which is convex according to Lemma \ref{thm:consistent}. Moreover, the subgradient \eqref{eqn:subgrad} and stochastic subgradient \eqref{eqn:stochastic-subgrad} are valid for the convex extension $\tilde{f}(x)$. Similarly, (stochastic) subgradients can be computed in the neighboring cube of each point and it does not matter which cube is chosen for points belonging to multiple cubes. Finally, the linear-time rounding process proposed in~\citet{zhang2020discrete} reduces the problem of finding PGS solutions of problem \eqref{eqn:obj} to that of the relaxed problem \eqref{eqn:obj-relax}. The pseudo-code is provided in the following algorithm.
\bigskip
\begin{breakablealgorithm}
\caption{Rounding process to a feasible solution}
\label{alg:multi-dim-round}
\begin{algorithmic}[1]
\Require{Model $\mathcal{X},(\mathsf{Y},\mathcal{B}_\mathsf{Y}),F(x,\xi_x)$, optimality guarantee parameters $\epsilon$ and $\delta$, $(\epsilon/2,\delta/2)$-PGS solution $\bar{x}$ to problem \eqref{eqn:obj-relax}.}
\Ensure{An $(\epsilon,\delta)$-PGS solution $x^*$ to problem \eqref{eqn:obj}.}
\State Compute a consistent permutation of $\bar{x}$, denoted as $\alpha$.
\State Compute the neighbouring points of $\bar{x}$, denoted as $S^0,\dots,S^d$.
\State Simulate $F(S^0,\xi_0),\dots,F(S^d,\xi_d)$ until the $1-\delta/4$ confidence half-width is smaller than $\epsilon/4$.
\State Return the point $x^*$ with the minimal empirical mean.
\end{algorithmic}
\end{breakablealgorithm}
\bigskip
The following theorem verifies the correctness and estimates the simulation cost of Algorithm \ref{alg:multi-dim-round}. 
\revise{\begin{lemma}[\cite{zhang2020discrete}]\label{thm:round}
Suppose that Assumptions \ref{asp:1}-\ref{asp:4} hold. The solution returned by Algorithm \ref{alg:multi-dim-round} satisfies the $(\epsilon,\delta)$-PGS guarantee and the simulation cost of Algorithm \ref{alg:multi-dim-round} is at most
\[ {O}\left[ \frac{d}{\epsilon^2}\log\left(\frac{d}{\delta}\right) + d \right] = \tilde{O}\left[ \frac{d}{\epsilon^2}\log\left(\frac{1}{\delta}\right) \right]. \]
\end{lemma}}

The rounding process for the $(c,\delta)$-PCS-IZ guarantee follows by choosing $\epsilon = c/2$. 

\subsection{Multi-dimensional Shrinking Uniform Sampling Algorithm}
\label{sec:multi-uni}

In this subsection, we give the multi-dimensional version of the shrinking uniform sampling (SUS) algorithm designed in Section \ref{sec:one-dim-uniform}. Similar to the one-dimensional case, the asymptotic simulation cost of the multi-dimensional algorithm is upper bounded by a constant that does not depend on the problem scale $N$ and the Lipschitz constant of the objective function. Hence, the multi-dimensional algorithm provides a matching simulation cost to the one-dimensional case. However, the expected simulation cost is exponentially dependent on the dimension $d$. Therefore, the multi-dimensional SUS algorithm is \revise{mainly theoretical and} only suitable for low-dimensional problems. 

The main idea of the generalization to multi-dimensional problems is to view optimization algorithms as (usually biased) estimators to the optimal value, 
%
%
which is elaborated in the following definition.
\begin{definition}
Given a constant $C>0$, we say that an algorithm is \textbf{sub-Gaussian with dimension $d$ and parameter $C$} if for any $d$-dimensional $L^\natural$-convex problem, any $\epsilon>0$ and small enough $\delta>0$, the algorithm returns a PGS solution $\hat{x}$ along with an estimate $\hat{f}^*$ to the optimal value $f^*$ that satisfies $|\hat{f}^* - f^* | \leq \epsilon$ with probability at least $1-\delta$ using at most
\[ T(\epsilon,\delta) := \tilde{O}\left[\frac{2C}{\epsilon^2}\log(\frac{2}{\delta})\right] \]
simulation runs.
\end{definition}

For example, Theorem \ref{thm:one-dim-uni} shows that the one-dimensional SUS algorithm (Algorithm \ref{alg:one-dim-uni}) returns an $(\epsilon/2,\delta/2)$-PGS solution with $\tilde{O}[\epsilon^{-2}\log(1/\delta)]$ simulations. Then, we can simulate the function value at the solution for ${O}[\epsilon^{-2}\log(1/\delta)]$ times such that the $1-\delta/2$ confidence half-width becomes smaller than $\epsilon/4$. Then, the empirical mean of function values at the solution is at most $\epsilon$ distant from $f^*$ with probability at least $1-\delta$. \revise{Hence, we know that Algorithm \ref{alg:one-dim-uni} is sub-Gaussian with dimension $1$. We denote its associated parameter as $C$.} We note that if we treat algorithms as estimators, the estimators are generally ``\textit{biased}'' (but consistent). This fact implies that the empirical mean of several estimates to the optimal value does not produce a better optimality guarantee, while the empirical mean of several unbiased estimators usually has a tighter deviation bound.

Now, we inductively construct sub-Gaussian algorithms for multi-dimensional problems. We first define the marginal objective function as
\begin{align}\label{eqn:munti-dim-uni-2} f^{d-1}(x) := \min_{y\in[N]^{d-1}} f(y,x). \end{align}
Observe that each evaluation of $f^{d-1}(x)$ requires solving a $(d-1)$-dimensional $L^\natural$-convex sub-problem. Hence, if we have an algorithm for $(d-1)$-dimensional $L^\natural$-convex problems, we only need to solve the one-dimensional problem
\begin{align}\label{eqn:multi-dim-uni} \min_{x\in[N]}~f^{d-1}(x) = \min_{x\in[N]}\min_{y\in[N]^{d-1}}~f(y,x) = \min_{x\in\mathcal{X}}~f(x) \end{align}
Moreover, we can prove that problem \eqref{eqn:multi-dim-uni} is also a convex problem.
\begin{lemma}\label{lem:multi-dim-uni}
If function $f(x)$ is $L^\natural$-convex, then function $f^{d-1}(x)$ is $L^\natural$-convex on $[N]$.
\end{lemma}
\begin{proof}{Proof of Lemma \ref{lem:multi-dim-uni}.}
The proof is provided in \ref{ec:lem-multi-dim-uni}.
\hfill\Halmos\end{proof}

Based on the observations above, we can use sub-Gaussian algorithms for $(d-1)$-dimensional problems and Algorithm \ref{alg:one-dim-uni} to construct sub-Gaussian algorithms for $d$-dimensional problems. We give the pseudo-code in Algorithm \ref{alg:multi-dim-uni}.
\bigskip
\begin{breakablealgorithm}
\caption{Multi-dimensional shrinking uniform sampling algorithm}
\label{alg:multi-dim-uni}
\begin{algorithmic}[1]
\Require{Model $\mathcal{X},(\mathsf{Y},\mathcal{B}_\mathsf{Y}),F(x,\xi_x)$, optimality guarantee parameters $\epsilon$ and $\delta$, sub-Gaussian algorithm $\mathcal{A}$ with dimension $d-1$.}
\Ensure{An $(\epsilon,\delta)$-PGS solution $x^*$ to problem \eqref{eqn:obj}.}
\State Set the active set $\mathcal{S}\leftarrow [N]$.
\State Set the step size $s\leftarrow1$ and the maximal number of comparisons $T_{max}\leftarrow N$.
\State Set $N_{cur} \leftarrow +\infty$.
\While{the size of $\mathcal{S}$ is at least $3$} \Comment{Iterate until $\mathcal{S}$ has at most $2$ points.}
    \If{ $|\mathcal{S}| \leq N_{cur} / 2$ } \Comment{Update the confidence interval.}
        \State Record current active set size $N_{cur} \leftarrow |\mathcal{S}|$.
        \State Set the confidence half-width $h \leftarrow N_{cur}\cdot \epsilon / 160$.
        \State For each $x\in\mathcal{S}$, use algorithm $\mathcal{A}$ to get an estimate to $f^{d-1}(x)$ such that
        \[ \left|\hat{f}^{d-1}(x) - f^{d-1}(x)\right| \leq h \]
        \quad\quad\quad holds with probability at least $1-\delta/(2T_{max})$.
    \EndIf
    \If{$\hat{f}^{d-1}(x) + h \leq \hat{f}^{d-1}(y) - h$ for some $x,y\in\mathcal{S}$} \Comment{\textbf{Type-I Operation}}
        \If{$x < y$}
            \State Remove all points $z\in\mathcal{S}$ with the property $z\geq y$ from $\mathcal{S}$.
        \Else
            \State Remove all points $z\in\mathcal{S}$ with the property $z\leq y$ from $\mathcal{S}$.
        \EndIf
    \Else \Comment{\textbf{Type-II Operation}}
        \State Update the step size $s\leftarrow 2s$. 
        \State Update $\mathcal{S}\leftarrow\{x_{min},x_{min}+s,\dots,x_{min}+ks \}$, where $x_{min}=\min_{x\in \mathcal{S}}~x$ and $k = \lceil|\mathcal{S}|/2\rceil-1$.
    \EndIf
\EndWhile \Comment{Now $\mathcal{S}$ has at most $2$ points.}
\State For each $x\in\mathcal{S}$, use Algorithm $\mathcal{A}$ to obtain an estimate to $f^{d-1}(x)$ such that
    \[ \left| \hat{f}^{d-1}(x) - f^{d-1}(x) \right| \leq \epsilon / 4 \]
holds with probability at least $1-\delta/(2T_{max})$. 
\State Return $x^* \leftarrow \argmin_{x\in\mathcal{S}}~\hat{f}^{d-1}(x)$.
\end{algorithmic}
\end{breakablealgorithm}
\bigskip
We prove that Algorithm \ref{alg:multi-dim-uni} is sub-Gaussian with dimension $d$ and estimate its parameter.
\begin{theorem}\label{thm:multi-dim-uni}
Suppose that Assumptions \ref{asp:1}-\ref{asp:4} hold, and that Algorithm $\mathcal{A}$ is sub-Gaussian with dimension $d-1$ and parameter $C$. Then, Algorithm \ref{alg:multi-dim-uni} is a sub-Gaussian algorithm with dimension $d$ and parameter $MC$, where $M>0$ is an absolute constant.
\end{theorem}
\begin{proof}{Proof of Theorem \ref{thm:multi-dim-uni}.}
The proof is provided in \ref{ec:multi-dim-uni}.
\hfill\Halmos\end{proof}

If we treat $F(x,\xi_x)$ as a sub-Gaussian algorithm with dimension $0$ and parameter $\sigma^2$, then Theorem \ref{thm:multi-dim-uni} implies that there exists a sub-Gaussian algorithm with dimension $1$ and parameter $\sigma^2M$. However, the parameter $C$ of Algorithm \ref{alg:one-dim-uni} is usually smaller than $\sigma^2M$ and therefore Algorithm \ref{alg:one-dim-uni} is preferred in the one-dimensional case. Using the results of Theorem \ref{thm:multi-dim-uni} and the fact that Algorithm \ref{alg:one-dim-uni} is sub-Gaussian with dimension $1$, we can inductively construct sub-Gaussian algorithms with any dimension $d$.
\revise{
\begin{theorem}\label{thm:multi-dim-uni-2}
Suppose that Assumptions \ref{asp:1}-\ref{asp:4} hold \revise{and Algorithm \ref{alg:one-dim-uni} is sub-Gaussian with dimension $1$ and parameter $C$}. There exists an $[(\epsilon,\delta)\text{-PGS},\mathcal{MC}]$-algorithm that is sub-Gaussian with parameter $M^{d-1}C$, where $M$ is the constant in Theorem \ref{thm:multi-dim-uni}. Hence, we have
\[ T(\epsilon,\delta,\mathcal{MC}) = \tilde{O}\left[ \frac{M^d}{\epsilon^2}\log\left(\frac{1}{\delta}\right) \right]. \]
Furthermore, by choosing $\epsilon=c/2$, it holds that
\[ T(\delta,\mathcal{MC}_c) = \tilde{O}\left[ \frac{M^d}{c^2}\log\left(\frac{1}{\delta}\right) \right]. \]
\end{theorem}
}
We note that although the upper bound in Theorem \ref{thm:multi-dim-uni} is independent of the Lipschitz constant $L$ and independent of $N$ when $\delta \ll 1$, the dependence on $d$ is exponential. Hence, Algorithm \ref{alg:multi-dim-uni} is largely theoretical and only suitable for low-dimensional problems, e.g., problems with $d\leq3$. On the other hand, if the dimension $d$ is treated as a fixed constant, Algorithm \ref{alg:multi-dim-uni} attains the optimal asymptotic performance under the asymptotic criterion in \citet{kaufmann2016complexity}. 
We also mention that Algorithm \ref{alg:multi-dim-uni} does not make a full use of the properties of $L^\natural$-convex functions. Actually, Algorithm \ref{alg:multi-dim-uni} is an $(\epsilon,\delta)$-PGS algorithm for those functions that are convex in each direction.

\subsection{Stochastic Cutting-plane Methods: Stochastic Separation Oracles}
\label{sec:multi-weak}

Now, we consider designing simulation-optimization algorithms with simulation costs having a polynomial dependence on the problem parameters $d$ and $N$. In addition, we reiterate that the goal is to design algorithms that do not require the information about the Lipschitz constant $L$ and the simulation cost is upper bounded by a constant that is independent of $L$. Intuitively, the subgradient information is useful for high-dimensional problems, while the localization operation is good at utilizing the discrete nature of the problem and get rid of the dependence on the Lipschitz constant. Therefore, one may expect subgradient-based localization methods to satisfy the aforementioned requirements. 
Using the definitions and tools introduced in Section \ref{sec:Lnatural}, we are able to design the desired algorithm in two steps. In this subsection, we first introduce the definition of stochastic separation oracles and give a novel framework to design stochastic cutting-plane methods via deterministic cutting-plane methods. Straightforward extensions of deterministic cutting-plane methods require prior knowledge about $L$ and the simulation cost has a logarithmic dependence on $L$. Hence, the following assumption is required.
%
\begin{assumption}\label{asp:5}
The $\ell_\infty$-Lipschitz constant $L$ is known a priori. Namely, we have
\[ |f(x) - f(y)| \leq L,\quad\forall x,y\in\mathcal{X},\quad\mathrm{s.t.}~\|x-y\|_\infty \leq 1.  \]
\end{assumption}
%
%
In the next subsection, we incorporate the stochastic cutting-plane methods with the dimension reduction operation. The resulting algorithm, named as the dimension reduction algorithm, does not require prior information about $L$ and the simulation cost is upper bounded by a constant that is independent of $L$.
We note that the design of the dimension reduction algorithm is the main objective of this section and stochastic cutting-plane methods mainly serve as an example of our novel framework.
%


In each iteration of a cutting-plane algorithm, a cutting hyperplane is generated to shrink the subset of potentially optimal choices of decision variables. In other words, the cutting hyperplane is used to localize the optimal solution. When the volume is small enough, the Lipschitz continuity implies that the all points in the polytope have their objective values close to the optimal value.
Compared to subgradient-based search methods, cutting-plane methods are more sensitive to noise. Hence, more simulation runs are required to generate robust separation oracles and therefore the simulation cost has a higher-order dependence on the problem dimension compared to subgradient-based search methods. 
As a counterpart of separation oracles, we introduce the stochastic separation oracle, named as the ($\epsilon,\delta$)-separation oracle, to characterize the accuracy of separation oracles in the stochastic case.
\begin{definition}
A \textbf{($\epsilon,\delta$)-separation oracle (($\epsilon,\delta$)-$\mathcal{SO}$)} is a function on $[1,N]^d$ with the property that for any input $x\in[1,N]^d$, it outputs a stochastic vector $\hat{g}_x\in\mathbb{R}^d$ such that the inequality
\[ f(y) \geq f(x) - \epsilon,\quad \forall y\in [1,N]^d \cap H  \]
holds with probability at least $1-\delta$, where the half space $H$ is defined as $\{ z: \langle \hat{g}_x,z-x \rangle \geq 0 \}$.
\end{definition}

Before we state algorithms, we give a concrete example of $(\epsilon,\delta)$-$\mathcal{SO}$ oracles and provide an upper bound on the expected simulation cost of evaluating each oracle. We define the averaged subgradient estimator as
\begin{align}\label{eqn:stochastic-subgrad1}
    \hat{g}^n_{\alpha_x(i)} := \hat{F}_n\left(S^{x,i}\right) - \hat{F}_n\left(S^{x,i-1}\right),\quad\forall i\in[d],
\end{align}
where $\alpha_x$ is a consistent permutation of $x$, $n \geq 1$ is the number of samples, and $\hat{F}_n$ is the empirical mean of $n$ independent evaluations of $F$. The following lemma gives a lower bound on $n$ to guarantee that $\hat{g}^n$ is an $(\epsilon,\delta)$-$\mathcal{SO}$ oracle.
\revise{
\begin{lemma}\label{lem:weak-1}
Suppose that Assumptions \ref{asp:1}-\ref{asp:4} hold. If we choose $n$ such that
\[ n = {\Theta}\left[ \frac{d N^2}{\epsilon^2}\log\left( \frac{1}{\delta} \right) \right], \]
then $\hat{g}^n$ is an $(\epsilon,\delta)$-$\mathcal{SO}$ oracle. Moreover, the expected simulation cost of generating an $(\epsilon,\delta)$-$\mathcal{SO}$ oracle is at most
\[ O\left[ \frac{d^2N^2}{\epsilon^2}\log\left( \frac{1}{\delta} \right) + d\right] = \tilde{O}\left[ \frac{d^2N^2}{\epsilon^2}\log\left( \frac{1}{\delta} \right) \right]. \]
\end{lemma}
}
\begin{proof}{Proof of Lemma \ref{lem:weak-1}.}
The proof is provided in \ref{ec:weak-1}.
\hfill\Halmos\end{proof}

We note that the condition in Lemma \ref{lem:weak-1} provides a sufficient condition of the $\mathcal{SO}$ oracle. In practice, the value of $n$ can be much smaller than the bound in Lemma \ref{lem:weak-1}; see numerical examples in Section \ref{sec:numerical}.
To show the usefulness of the stochastic separation oracle, we extend Vaidya's cutting-plane method \citep{vaidya1996new} to a stochastic cutting-plane method that can find PGS solutions in the stochastic case. Vaidya's cutting-plane method maintains a polytope that contains the optimal points and iteratively reduces the volume of polytope by generating a separation oracle at the approximate volumetric center. We provide the pseudo-code of deterministic Vaidya's method in \ref{ec:cuttingplane} for the self-contained purpose. Other deterministic cutting-plane methods based on reducing the volume of a polytope can also be extended to the stochastic case using our novel framework, and we consider Vaidya's method mainly for its simplicity. 

It is desirable to prove that by substituting the separation oracles with stochastic separation oracles, Vaidya's cutting-plane method can be used to find PGS solutions. The pseudo-code of the stochastic cutting-plane method is given in Algorithm \ref{alg:multi-dim-weak}.
\bigskip
\begin{breakablealgorithm}
\caption{Stochastic cutting-plane method for the PGS guarantee}
\label{alg:multi-dim-weak}
\begin{algorithmic}[1]
\Require{Model $\mathcal{X},(\mathsf{Y},\mathcal{B}_\mathsf{Y}),F(x,\xi_x)$, optimality guarantee parameters $\epsilon$ and $\delta$, Lipschitz constant $L$, $(\epsilon,\delta)$-$\mathcal{SO}$ oracle $\hat{g}$.}
\Ensure{An $(\epsilon,\delta)$-PGS solution $x^*$ to problem \eqref{eqn:obj}.}
\State Set the initial polytope $P \leftarrow [1,N]^d$.
\State \revise{Set the constant $\rho \leftarrow 10^{-7}$. \Comment{Constant $\rho$ corresponds to $\epsilon$ in \citet{vaidya1996new}.}}
\State \revise{Set the number of iterations $T_{max}\leftarrow \lceil 2d / \rho \cdot \log[d N L / (\rho\epsilon)] \rceil$.}
\State Initialize the set of points used to query separation oracles $\mathcal{S} \leftarrow \emptyset$.
\State Initialize the volumetric center $z \leftarrow (N+1)/2 \cdot (1,1,\dots,1)^T$.
\For{$T=1,2,\dots,T_{max}$}
    \State Decide adding or removing a cutting plane by Vaidya's method.
    \If{add a cutting plane}
        \State Evaluate an $(\epsilon/8,\delta/4)$-$\mathcal{SO}$ oracle $\hat{g}_z$ at $z$.
        \If{$\hat{g}_z=0$}
            \State Round $z$ to an integral solution by Algorithm \ref{alg:multi-dim-round} and return the rounded solution.
        \EndIf
        \State Add the current point $z$ to $\mathcal{S}$.
    \ElsIf{remove a cutting plane}
        \State Remove corresponding point $z$ from $\mathcal{S}$.
    \EndIf
    \State Update the approximate volumetric center $z$ by a Newton-type method.
\EndFor \Comment{There are at most $O(d)$ points in $\mathcal{S}$ by Vaidya's method.}
\State Find an $(\epsilon/4,\delta/4)$-PGS solution $\hat{x}$ of problem $\min_{x\in\mathcal{S}}~ f(x)$.
\State Round $\hat{x}$ to an integral solution by Algorithm \ref{alg:multi-dim-round}.
\end{algorithmic}
\end{breakablealgorithm}
\bigskip
We note that
if the approximate volumetric center $z$ is not in $[1,N]^d$, then we choose a violated constraint $x_i \geq 1$ or $x_i \leq N$ and return $e_i$ or $-e_i$ as the separating vector, respectively. For arithmetic operations, each iteration of Algorithm \ref{alg:multi-dim-weak} requires $O(d)$ inversions and multiplications of $d\times d$ matrices. Each inversion and multiplication can be finished within $O(d^\omega)$ arithmetic operations, where $\omega < 2.373$ is the matrix exponent \citep{alman2020refined}. Hence, Algorithm \ref{alg:multi-dim-weak} needs $O(d^{\omega+1})$ arithmetic operations for each iteration. The calculation of the number of iterations $T_{max}$ is provided in \ref{ec:thm-weak-1}. The correctness and the expected simulation cost of Algorithm \ref{alg:multi-dim-weak} are studied in the following theorem.
\begin{theorem}\label{thm:weak-1}
Suppose that Assumptions \ref{asp:1}-\ref{asp:5} hold. Algorithm \ref{alg:multi-dim-weak} returns an $(\epsilon,\delta)$-PGS solution and we have
\[ T(\epsilon,\delta,\mathcal{MC}) = O\left[ \frac{d^3N^2}{\epsilon^2}\log\left(\frac{dLN}{\epsilon}\right)\log\left(\frac{1}{\delta}\right) + d^2\log\left(\frac{dLN}{\epsilon}\right) \right] = \tilde{O}\left[ \frac{d^3N^2}{\epsilon^2}\log\left(\frac{dLN}{\epsilon}\right)\log\left(\frac{1}{\delta}\right) \right]. \]
\end{theorem}
\begin{proof}{Proof of Theorem \ref{thm:weak-1}.}
The proof is provided in \ref{ec:thm-weak-1}.
\hfill\Halmos\end{proof}
\begin{remark}\label{rmk:3}
We note that another popular deterministic cutting-plane method, the random walk method \citep{bertsimas2004solving}, can also be extended to the stochastic case and achieves a better expected simulation cost
\[ \tilde{O}\left[ \frac{d^3N^2}{\epsilon^2}\log\left(\frac{LN}{\epsilon}\right)\log\left(\frac{1}{\delta}\right) \right] \]
at the expense of $\tilde{O}[d^6+\log^2(1/\delta)]$ arithmetic operations in each iteration. \revise{We provide the pseudo-code in \ref{ec:cuttingplane} for the self-contained purpose.} Here, the $O[\log^2(1/\delta)]$ factor is required to ensure the high-probability approximation to the centroid. Moreover, we note that the fast implementation of Vaidya's method in~\citet{jiang2020improved} reduces number of arithmetic operations in each iteration to $O(d^2)$.
\end{remark} 
\begin{remark}
Stochastic cutting-plane methods can also be applied to problems that are defined on $[N]^d$ with linear constraints $\{x\in\mathbb{Z}^d:Ax\leq b\}$, since we can choose the initial polytope to be $\mathcal{X}:=[1,N]^d \cap \{ Ax \leq b \}$. The results in this section still hold if we replace $N$ with $\max_{x,y\in\mathcal{X}}\|x-y\|_\infty$.
\end{remark}

\subsection{Stochastic Cutting-plane Methods: Dimension Reduction Algorithm}
\label{sec:multi-strong}

In this subsection, we develop the dimension reduction algorithm, which does not require the knowledge about the Lipschitz constant $L$ and whose simulation cost is upper bounded by a constant that is independent of $L$. 
The idea behind the dimension reduction algorithm is based on the following observation: if a convex body $P\subset\mathbb{R}^d$ has a volume \revise{$\mathrm{vol}(P)$} smaller than $(d!)^{-1}=O[\exp(-(d+1/2)\log(d)+d)]$, then all integral points inside $P$ must lie on a hyperplane. Otherwise, if there exist $d+1$ integral points $x_0,\dots,x_d \in P$ that are not on the same hyperplane, then the convex body $P$ contains the polytope $\mathrm{conv}\{x_0,\dots,x_d\}$, which has the volume
\[ \frac{1}{d!}\left| \mathrm{det}(x_1-x_0,\dots,x_d-x_0) \right| \geq \frac{1}{d!}, \]
where $\mathrm{conv}(\cdot)$ is the convex hull and $\mathrm{det}(\cdot)$ is the determinant of matrices. This leads to a contradiction \revise{since we assume that $\mathrm{vol}(P) < (d!)^{-1}$}. Hence, we may use Vaidya's method or the random walk method to reduce the volume of the search polytope $P$ to $O[\exp(-(d+1/2)\log(d)+d)]$, and then we reduce the problem dimension by projecting the polytope onto the hyperplane that all remaining points lie on. After $d-1$ dimension reductions, we have an one-dimensional convex problem and algorithms in Section \ref{sec:one-dim} can be applied. This idea is summarized in Algorithm \ref{alg:multi-dim-strong}.
\bigskip
\begin{breakablealgorithm}
\caption{Dimension reduction algorithm for the PGS guarantee}
\label{alg:multi-dim-strong}
\begin{algorithmic}[1]
\Require{Model $\mathcal{X},(\mathsf{Y},\mathcal{B}_\mathsf{Y}),F(x,\xi_x)$, optimality guarantee parameters $\epsilon$ and $\delta$, $(\epsilon,\delta)$-$\mathcal{SO}$ oracle $\hat{g}$.}
\Ensure{An $(\epsilon,\delta)$-PGS solution $x^*$ to problem \eqref{eqn:obj}.}
\State Set the initial polytope $P \leftarrow [1,N]^d$.
\State Initialize the set of points used to query separation oracles $\mathcal{S} \leftarrow \emptyset$.
\For{$d'=d,d-1,\dots,2$} \Comment{The current dimension $d'$ is gradually reduced.}
    \State Initialize Vaidya's cutting-plane method.
    \While{the volume of $P$ is larger than $(d'!)^{-1}$}
        \State Take one step of Vaidya's cutting-plane method with $(\epsilon/4,\delta/4)$-$\mathcal{SO}$ oracle.
        \Statex \Comment{Vaidya's cutting-plane method decides a suitable cutting plane $H$.}
        \State Add the point where the stochastic separation oracle is called to $\mathcal{S}$.
        \State Shrink the volume of $P$ using the cutting plane $H$.
    \EndWhile
    \State Find the hyperplane $H$ that contains all integral points in $P$.
    \Statex \Comment{If $P$ contains no integral points, then an arbitrary hyperplane works.}
    \State Project $P$ onto the hyperplane $H$. \Comment{Reduce the dimension by $1$.}
\EndFor
\State Find an $(\epsilon/4,\delta/4)$-PGS solution of the last one-dim problem and add the solution to $\mathcal{S}$.
\State Find the $(\epsilon/4,\delta/4)$-PGS solution $\hat{x}$ of problem $\min_{x\in\mathcal{S}}~ f(x)$.
\State Round $\hat{x}$ to an integral solution by Algorithm \ref{alg:multi-dim-round}.
\end{algorithmic}
\end{breakablealgorithm}
\bigskip

We note that the application of Vaidya's method in line 6 refers to implementing the cutting-plane algorithm for one iteration. Namely, only a single cutting hyperplane will be generated. Importantly, the implementation of Vaidya's method in this step does not require the knowledge about the Lipschitz constant, since the Lipschitz constant is only used to calculate the total number of steps in Algorithm \ref{alg:multi-dim-weak}. In addition, Vaidya's cutting-plane method can be replaced with other deterministic cutting-plane methods.
To make Algorithm \ref{alg:multi-dim-strong} more practical, we need to consider the following question:
\begin{itemize}
    \item How many arithmetic operations are required to identify the hyperplane given that the volume of $P$ is small enough?
\end{itemize}
%
%
The total number of arithmetic operations in each iteration is mainly determined by the answer to this question, since it is easy to show that other parts of the algorithm require only polynomially many arithmetic operations. Intuitively, the problem of identifying the hyperplane can be finished by finding a vector $c\in\mathbb{Z}^d$ such that
\[ \langle c, x - y\rangle \approx 0,\quad \forall x,y \in P, \]
where $P$ is the current polytope. In \citet{jiang2020minimizing}, the author reduced the problem to the Shortest Vector Problem in lattices and showed that the LLL algorithm~\citep{lenstra1982factoring} can be applied to find a set of LLL-reduced basis~\citep{lenstra1982factoring}, which contains the normal vector of the hyperplane when the volume of search set is small enough. We show that their results can be extended to the stochastic case and can be combined with the framework in Section \ref{sec:multi-weak} to generate the desired dimension reduction algorithm. Intuitively, the dimension reduction algorithm implements the stochastic cutting-plane method at each dimension from $d$ to $1$. 
Therefore, the total simulation cost is on the same order as the summation of $i^3$ for $i\in[d]$, which is on the order of $O(d^4)$. More rigorously, we provide the correctness and the simulation cost of Algorithm \ref{alg:multi-dim-strong} in the following theorem.
\begin{theorem}\label{thm:strong-1}
Suppose that Assumptions \ref{asp:1}-\ref{asp:4} hold. Algorithm \ref{alg:multi-dim-strong} returns an $(\epsilon,\delta)$-PGS solution and we have
\[ T(\epsilon,\delta,\mathcal{MC}) = O\left[ \frac{d^3N^2(d+\log(N))}{\epsilon^2}\log\left(\frac{1}{\delta}\right) + d^2(d+\log(N)) \right] = \tilde{O}\left[ \frac{d^3N^2(d+\log(N))}{\epsilon^2}\log\left(\frac{1}{\delta}\right) \right]. \]
\end{theorem}
\begin{proof}{Proof of Theorem \ref{thm:strong-1}.}
The proof is provided in \ref{ec:thm-strong-1}.
\hfill\Halmos\end{proof}
We note that the idea of gradually reducing the dimension is proposed in our work and \citet{jiang2020minimizing} independently, although the author of \citet{jiang2020minimizing} has made the algorithm more practical. 
More specifically, if we allow exponentially many arithmetic operations, the LLL algorithm is not necessary. In that case, we can reduce the number of separation oracles to $O(d^2)$ and the computational complexity can be reduced to $\tilde{O}[d^4N^2\epsilon^{-2}\log(1/\delta)]$.


\section{Adaptive Sub-Gaussian Parameter Estimator} \label{sec:var}

In this section, we provide a simple adaptive mean estimator to adaptively estimate the variance of each choice of decision variable under the assumption that the distribution of the randomness is Gaussian. The estimator can be used to further enhance our proposed algorithm and we hope the procedure to be useful for other optimization via simulation problems and algorithms that do not know the variances a priori. Using the adaptive estimator, the prior knowledge about the upper bound on the variance $\sigma^2$ is not necessary. In addition, for the multi-dimensional localization algorithms proposed in this work, the simulation cost for the unknown variance case is at most a constant factor larger than the case when an upper bound on the variance is known a priori. Therefore, the algorithm using the adaptive estimator, or the adaptive algorithm, is able to improve the performance of our proposed algorithms if an estimate of the upper bound $\sigma^2$ is much larger than the true variance. In this case, the original algorithms will implement an unnecessarily large number of simulation runs to shrink the confidence interval, while the adaptive algorithm is able to automatically learn the true variance and thus save the computational cost. Another situation where the adaptive algorithm is useful is when the variance of the system varies a lot at different choices of decision variable. In this case, the upper bound of the variance is usually attained at extremely choices of decision variable and is much larger than the variance of a majority of feasible choices. For example, we consider the case when the noise is multiplicative and Gaussian. Namely, the noisy evaluation is $F(x,\xi) = \eta(\xi) \cdot f(x)$ for all $x\in\mathcal{X}$, where $\eta(\xi)$ obeys the distribution $\mathcal{N}(1, \sigma^2)$. In this example, the tightest upper bound of the variance is $\sigma^2 \max_x f^2(x)$, which is much larger than the variance at points $x^0$ such that $f(x^0) \ll \max_x f(x)$. Therefore, using the upper bound at all points leads to a conservative mean estimator. Finally, we note that the adaptive algorithm also provides an explicit way to utilize the information about the variance $\sigma^2_x$ at each point $x\in\mathcal{X}$. Here, $\sigma^2_x$ refers to an upper bound on the variance at point $x$.

We now state the proposed adaptive mean estimator. To increase the generality of our results, we make a weaker assumption than the Gaussian case.
\begin{assumption}\label{asp:7}
The distribution of $F(x,\xi_x) - f(x)$ belongs to the family of sub-Gaussian distributions $\mathcal{F}_\kappa$, where $\kappa>0$ is a known constant. For any random variable $X$ whose distribution belongs to $\mathcal{F}_\kappa$, it holds that
%
\begin{align}\label{eqn:ada-1} 
\kappa \sigma_X^2 \leq  \mathrm{Var}(X),
\end{align}
where $\sigma_X^2$ is the sub-Gaussian parameter of the distribution.
\end{assumption}
We note that the inverse inequality $\mathrm{Var}(F(x, \xi_x)) \leq \sigma_x^2$ always holds for all sub-Gaussian distributions. However, there does not exist a universal constant $\kappa>0$ such that inequality \eqref{eqn:ada-1} holds for all sub-Gaussian distributions. Therefore, Assumption \ref{asp:7} cannot be implied by Assumption \ref{asp:3}. In the special case when the distribution of $F(x,\xi_x)$ is Gaussian, the constant $\kappa = 1$, i.e., we have the following relation:
\[ \sigma_X^2 = \mathrm{Var}(X). \]
Therefore, Assumption \ref{asp:7} includes the Gaussian distribution as a special case. Under the above assumption, we propose the adaptive mean estimator.
\begin{definition}
\label{def:ada}
Let $\epsilon>0$ be the precision and $\delta\in(0,1]$ be the failing probability. We construct the adaptive mean estimator of $f(x)$ in two steps:
\begin{enumerate}
    \item Sample $2n$ independent evaluations $F(x,\xi_i)$ for $i\in[2n]$, where $n:= \lceil 256\kappa^{-2}\log(2/\delta) \rceil$. Compute the variance estimator
    \[ \hat{\mathrm{Var}} := \frac{1}{n} \sum_{i=1}^n \left[F(x, \xi_{2i-1}) - F(x,\xi_{2i})\right]^2 \]
    and the parameter estimator
    \[ \hat{\sigma}^2 := \frac{1}{\kappa}\hat{\mathrm{Var}}. \]
    \item Let $m := \max\{\lceil 2\epsilon^{-2}\hat{\sigma}^2\log(2/\delta) \rceil, 2n\}$ and sample $m-2n$ independent evaluations $F(x,\xi_{2n+i})$ for $i\in[m-2n]$ and compute the empirical mean
    \[ \hat{F}(x; \delta) := \frac{1}{m} \sum_{i=1}^{m} F(x,\xi_i). \]
\end{enumerate}
\end{definition}
The construction of the adaptive mean estimator has two steps. In the first step, we estimate an upper bound for the sub-Gaussian parameter, and in the second step, we use the estimated upper bound to calculate the required number of simulation so that the sub-Gaussian parameter is less than a known constant. We note that the adaptive mean estimator is an online estimator. To be more concrete, if a smaller precision $\epsilon'<\epsilon$ is required, it suffices to add
\[ \lceil 2(\epsilon')^{-2}\hat{\sigma}^2\log(2/\delta) \rceil - \lceil 2\epsilon^{-2}\hat{\sigma}^2\log(2/\delta) \rceil \]
more evaluations into the empirical mean in step 2.
The following theorem verifies that $\hat{F}(\cdot;\delta)$ is an unbiased mean estimator for $f(x)$ and its tail is sub-Gaussian with a small failing probability.
\begin{theorem}
\label{thm:ada-1}
Suppose that Assumption \ref{asp:7} holds. Let $\delta\in(0,1]$ be the failing probability. For all $\epsilon\geq0$, the adaptive mean estimator satisfies
\begin{align}\label{eqn:ada-0}
\mathbb{P}\left[ |\hat{F}(x; \delta) - f(x)| \geq \epsilon \right] \leq \delta, 
\end{align}
%
In addition, the expected simulation cost of the adaptive mean estimator is $O[(\kappa^{-2}+\epsilon^{-2}\kappa^{-1}\sigma_x^2)\log(1/\delta)]$.
\end{theorem}
\begin{proof}{Proof of Theorem \ref{thm:ada-1}.}
The proof is provided in \ref{ec:thm-ada-1}.
\hfill\Halmos\end{proof}
If the sub-Gaussian parameter $\sigma_x^2$ is known, the Hoeffding bound shows that
\[ O[ \epsilon^{-2}\sigma_x^2\log(1/\delta) ] \]
samples are sufficient to generate an estimator for inequality \eqref{eqn:ada-0}. Therefore, the relative efficiency of the adaptive mean estimator is
\[ \frac{\epsilon^{-2}\sigma_x^2}{\kappa^{-2}+\epsilon^{-2}\kappa^{-1}\sigma_x^2} = \frac{1}{\kappa^{-1} + \kappa^{-2}\epsilon^2\sigma_x^{-2} }. \]
If the precision $\epsilon$ is small or the parameter $\sigma_x^2$ is large, the adaptive mean estimator is only a constant ($\kappa$) time less efficient than the known variance case.

Now, we estimate the expected simulation cost of our proposed simulation-optimization algorithms combined with the adaptive estimator. Intuitively, we need to implement the first step in Definition \ref{def:ada} once for all simulated choices of decision variable. Suppose that a simulation-optimization algorithm simulates $N(\epsilon,\delta)$ different choices of decision variable in expectation and the expected simulation cost is $T(\epsilon,\delta)$. Then, the expected simulation cost of the adaptive simulation-optimization algorithm is
\[ O\left[ \kappa^{-1}T(\epsilon,\delta) + \kappa^{-2} N(\epsilon,\delta) \log\left( N(\epsilon,\delta) / \delta \right) \right]. \]
%
For the localization algorithms, we usually have $T(\epsilon,\delta) = O[ N(\epsilon,\delta)\log\left( N(\epsilon,\delta) / \delta \right)]$. Therefore, the expected simulation cost of the adaptive algorithm is $O(T(\epsilon,\delta))$. More concretely, we have the following corollary.
\begin{corollary}
Suppose that Assumptions \ref{asp:1}, \ref{asp:4}-\ref{asp:7} hold. The following estimates hold:
\begin{itemize}
    \item The expected simulation cost of adaptive tri-section sampling algorithm (Algorithm \ref{alg:one-dim}) is 
    \[ O\left[ (1 + \epsilon^{-2})\log(N)\log\left( \frac{\log(N)}{\delta} \right) + \log(N) \right] = \tilde{O}\left[ (1 + \epsilon^{-2})\log(N)\log\left( \frac{1}{\delta} \right) \right]. \]
    \item The expected simulation cost of adaptive shrinking uniform sampling algorithm (Algorithm \ref{alg:one-dim-uni}) is 
    \[ O\left[ (N + \epsilon^{-2})\log\left( \frac{N}{\delta} \right) \right] = \tilde{O}\left[ (N + \epsilon^{-2})\log\left( \frac{1}{\delta} \right) \right]. \]
    \item The expected simulation cost of adaptive stochastic cutting-plane algorithm (Algorithm \ref{alg:multi-dim-weak}) is 
    \[ O\left[ \frac{d^3N^2}{\epsilon^2}\log\left(\frac{dLN}{\epsilon}\right)\log\left(\frac{1}{\delta}\right) + d^2\log\left(\frac{dLN}{\epsilon}\right) \right] = \tilde{O}\left[ \frac{d^3N^2}{\epsilon^2}\log\left(\frac{dLN}{\epsilon}\right)\log\left(\frac{1}{\delta}\right) \right]. \]
    \item The expected simulation cost of adaptive dimension reduction algorithm (Algorithm \ref{alg:multi-dim-strong}) is 
    \[ O\left[ \frac{d^3N^2(d+\log(N))}{\epsilon^2}\log\left(\frac{1}{\delta}\right) + d^2(d+\log(N)) \right] = \tilde{O}\left[ \frac{d^3N^2(d+\log(N))}{\epsilon^2}\log\left(\frac{1}{\delta}\right) \right]. \]
\end{itemize}
\end{corollary}
Here, constants $\sigma_x^2, \kappa$ are omitted in the $O(\cdot)$ and $\tilde{O}(\cdot)$ notations. We can see that the expected simulation cost of the stochastic cutting-plane method and the dimension reduction algorithm is only increased by a factor. Therefore, the adaptive mean estimator is useful in dropping the requirement of known parameter $\sigma$ for the multi-dimensional case. For the one-dimensional case, the expected simulation cost of the tri-section algorithm is also increased by a constant factor. On the other hand, the cost of the adaptive shrinking uniform sampling algorithm is larger than the original version, especially when $\epsilon^{-2} \ll N$. Therefore, in the one-dimensional unknown parameter case, we need to estimate the size of $\epsilon$ to decide whether to use the tri-section sampling algorithm or the shrinking uniform sampling algorithm, More specifically, if $\epsilon = O(\sqrt{\log{N}/N})$, then the shrinking uniform sampling algorithm is preferred; otherwise the tri-section sampling algorithm is preferred.

\section{Numerical Experiments}
\label{sec:numerical}

In this section, we implement our proposed simulation-optimization algorithms that are guaranteed to find high-confidence high-precision PGS solutions. \revisee{Through these numerical experiments, we show that the localization methods proposed in this manuscript outperform benchmark algorithms on large-scale problems.}
First, we consider the problem of finding the optimal allocation of a total number of $N$ staffs to two queues so that the average waiting time for all of the arrivals from the two queues is minimized. Given the optimality parameters $\epsilon$ and $\delta$, we empirically show that the tri-section sampling algorithm and the shrinking uniform sampling algorithm have respectively $O(\log N)$ and $O(1)$ dependence on the scale $N$, which supports our theoretical results. In addition, we construct a synthetic one-dimensional convex function with a similar landscape to show that the returned solution satisfies the high-probability guarantee. 
Second, we construct a multi-dimensional stochastic function, whose expectation is a separable convex function, i.e., functions of the form $f(x)=\sum_{i=1}^{d} f^i(x_i)$ for convex functions $f^1(x),\dots,f^d(x)$, to test and compare the subgradient descent algorithm \citep{zhang2020discrete} with the stochastic localization methods proposed in this work for different values of the scale $N$ and dimension $d$, especially for large $N$.
Similar to the one-dimensional case, we consider functions with a closed-form to check the coverage rate of the proposed algorithms. 
Finally, the multi-dimensional resource allocation problem in service systems is considered to compare the performance of proposed algorithms on practical problems.

\subsection{Staffing Two Queues under Resource Constraints}

Consider a service system that operates over a time horizon $[0,T]$ with two streams of customers arriving at the system. One example is that the system receives service requests from both online app-based customers and offline walk-in customers, and each stream needs dedicated servers assigned. The first stream of customers arrives according to a doubly stochastic non-homogeneous Poisson process $N_1 := (N_1(t):t\in [0,T])$, with the customer service times being independent and identically distributed according to a distribution $S_1$. The second stream of customers obeys the same model with the process $N_2 := (N_2(t):t\in [0,T])$ and distribution $S_2$.
The two streams of customers form two separate queues and their arrival processes can be correlated. Suppose that the decision maker needs to staff the two queues separately. There are in total a number of $N+1$ homogeneous servers that work independently in parallel. Each server can handle the service requested by customers from either stream, one at a time. 
Suppose that no change on the staffing plan can be made once the system starts working. Assume that the system operates based on a first-come-first-serve routine, with unlimited waiting room in each queue, and that customers never abandon.  

The decision maker's objective is to select the staffing level $x \in[N]$ for the first queue and the staffing level $N+1-x$ for the second queue, in order to minimize the expected average waiting time for all customers from the two streams over the time horizon $[0,T]$. In the numerical example, we consider $N\in\{10,20,\dots,150\}$ and $T = 2$. The arrival processes $N_1$ and $N_2$ are non-homogeneous processes with random intensity functions $\Gamma_1\cdot \lambda_1(t)$ and $\Gamma_2 \cdot \lambda_2(t)$, in which
\[
\lambda_1(t) := 75 + 25 \sin(0.3t),\quad \lambda_2(t) := 80 + 40\sin(0.2t).
\]
Positive-valued random variables $\Gamma_1$ and $\Gamma_2$ are defined as
\[ \Gamma_1 := X + Z, \quad \Gamma_2 := Y - Z, \]
where $X,Y$ are independent uniform random variables on $[0.75,1.25]$ and $Z$ is an independent uniform random variable on $[-0.5,0.5]$. The service time distribution $S_1$ is log-normal distributed with mean $0.75$ and variance $0.1$. The service time distribution $S_2$ is gamma distributed with mean $0.65$ and variance $0.1$. 
%
%
Figure~\ref{fig:landscape} plots an empirical average waiting time as a function of the discrete decision variable $x$. It can be observed that the landscape around the optimum is extremely flat and such property may cause challenges for algorithms that aim to exactly select the optimal solution (i.e., the PCS guarantee). In practice, the decision maker may be indifferent about a very small difference in the averaging waiting time performance, when the small difference does not impact much on customers satisfaction.
Instead, algorithms that are designed for the ($\epsilon,\delta$)-PGS guarantee do not suffer from the extremely flat landscape around the global optimum.

Moreover, we construct the convex objective function
\[ f(x;c,x^*) := \begin{cases}
c\left(\sqrt{\frac{x^*}{x}} - 1\right) & \text{if } x \leq x^*\\ c\left(\sqrt{\frac{N+1-x^*}{N+1-x}} - 1\right) & \text{if } x > x^*
\end{cases} ,\quad \forall x,x^* \in [N],
\]
where $c\in[0.75,1.25]$ and $x^* \in\{1,\dots,\lfloor0.3N\rfloor\}$. The objective function has a similar landscape as the average waiting time; see Figure \ref{fig:landscape}. We use this closed-form function to verify that the $\epsilon$-optimality is satisfied with high probability.

\begin{figure}[t]
\begin{center}
\begin{subfigure}{.49\textwidth}
    \centering
    \includegraphics[scale=0.5]{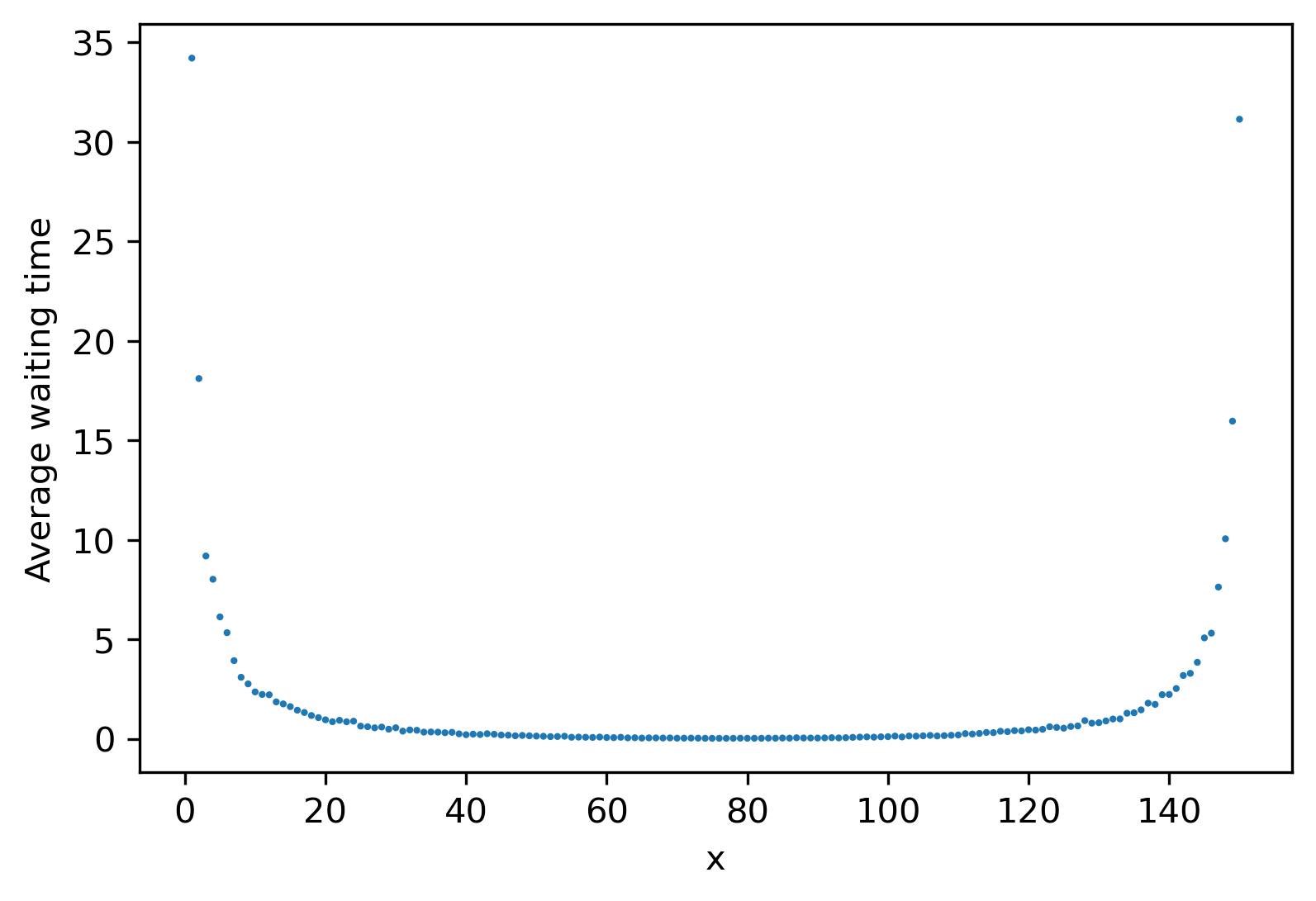}
    \caption{}
\end{subfigure}\hfill
\begin{subfigure}{.49\textwidth}
    \centering
    \includegraphics[scale=0.5]{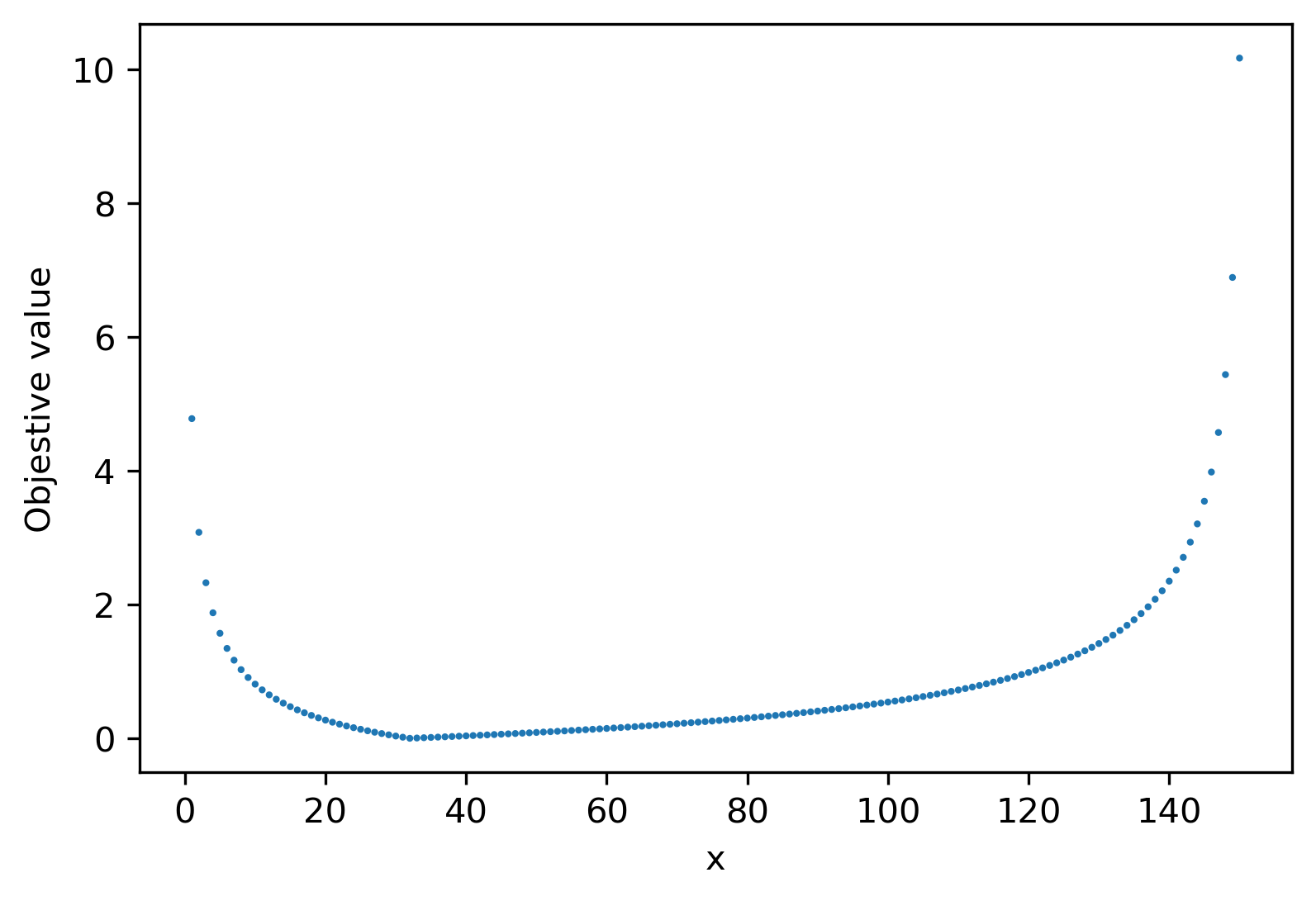}
    \caption{}
\end{subfigure}\hfill
\caption{The landscapes of objective functions in the one-dimensional case. \textbf{(a)} The empirical average waiting time with $N=150$. \textbf{(b)} The landscape of the synthetic convex function with scale $N=150$ and optimum $x^*=31$.} 
\label{fig:landscape}
\end{center}
\end{figure}

\subsection{Separable Convex Function Minimization}

We consider the problem of minimizing a stochastic function whose expectation is a separable $L^\natural$-convex function of the form
\[ f_{c,x^*}(x) := \sum_{i=1}^d c_i g(x_i;x_i^*), \]
where $c_i\in[0.75,1.25]$, $x^*_i \in\{1,\dots,\lfloor0.3N\rfloor\}$ for all $i\in[d]$ and
\[ g(x;x_i^*) := \begin{cases}
(x_i^* - x)^{1.2} & \text{if } x \leq x^*\\ (x - x_i^*)^{1.2} & \text{if } x > x^*
\end{cases} ,\quad \forall x,x^* \in [N].
\]
It can be observed that the function $f_{c,x^*}(x)$ is the sum of separable convex functions and therefore is $L^\natural$-convex. Moreover, the function $f_{c,x^*}(x)$ has the optimum $x^*$ associated with the optimal value $0$. For stochastic evaluations, we add Gaussian noise with mean $0$ and variance $1$. 
The advantage of this numerical example is that the expected objective function has a closed form, and we are able to exactly compute the optimality gap of the solutions returned by the proposed algorithms. 

\subsection{Resource Allocation Problem in Service Systems}

%
We consider the $24$-hour operation of a service system with a single stream of incoming customers. The customers arrive according to a doubly stochastic non-homogeneous Poisson process with the intensity function
\[ \Lambda(t) := 0.5\lambda N \cdot ( 1 - |t - 12| / 12 ),\quad\forall t\in[0,24], \]
where $\lambda$ is a positive constant and $N$ is a positive integer. Each customer requests a service with the service time independent and identically distributed according to the log-normal distribution with mean $1/\lambda$ and variance $0.1$. We divide the $24$-hour operation into $d$ time slots with length $24/d$ for some positive integer $d$. For the $i$-th time slot, there are $x_i \in[N]$ of homogeneous servers that work independently in parallel and the number of servers cannot be changed during the slot. Assume that the system operates based on a first-come first-serve routine, with an unlimited waiting room in each queue, and that customers never abandon. 

The decision maker's objective is to select the staffing level $x:=(x_1,\dots,x_d)$ such that the total waiting time of all customers is minimized. Namely, by letting $f(x)$ be the expected total waiting time under the staffing plan $x$, the optimization problem can be written as
\begin{align}\label{eqn:num-1} \min_{x\in[N]^d} f(x). \end{align}
It has been proved in \citet{altman2003discrete} that the function $f(\cdot)$ is multimodular. We define the linear transformation
\[ g(y) := ( y_1,y_2-y_1,\dots, y_d - y_{d-1} ) \quad \forall y\in\mathbb{R}^d. \]
Then, \citet{murota2003discrete} has proved that
\[ h(y) := f \circ g(y) = f( y_1,y_2-y_1,\dots, y_d - y_{d-1} ) \]
is a $L^\natural$-convex function on the $L^\natural$-convex set
\[ \mathcal{Y} := \{ y \in [Nd]^d ~|~ y_1\in[N],~ y_{i+1}-y_i \in[N],~i=1,\dots, N - 1 \}. \]
The optimization problem \eqref{eqn:num-1} has the trivial solution $x_1=\cdots=x_d=N$. However, in reality, it is also necessary to keep the staffing cost low.
Therefore, we add the staffing cost term $R(x_1,\dots,x_d) := C/d \cdot \sum_{i=1}^d x_i = C/d \cdot y_d$ 
to the objective function, where $C$ is a positive constant. The optimization problem can be written as
\begin{align}\label{eqn:num-3}
    \min_{y\in\mathcal{Y}} h(y) + C/d \cdot y_d.
\end{align}
%
The proposed algorithms can be extended to this problem by considering the \lovasz extension $\tilde{h}(y)$ on the set
\[ \tilde{\mathcal{Y}} := \{ y\in[1,Nd]^d ~|~ y_1\in[1,N],~ y_{i+1}-y_i \in[1,N],~i=1,\dots, N - 1 \}. \]

\subsection{Numerical Results: Tri-section Sampling Algorithm and Shrinking Uniform Sampling Algorithm}

\begin{figure}[t]
\begin{center}
\begin{subfigure}{.49\textwidth}
    \centering
    \includegraphics[scale=0.3]{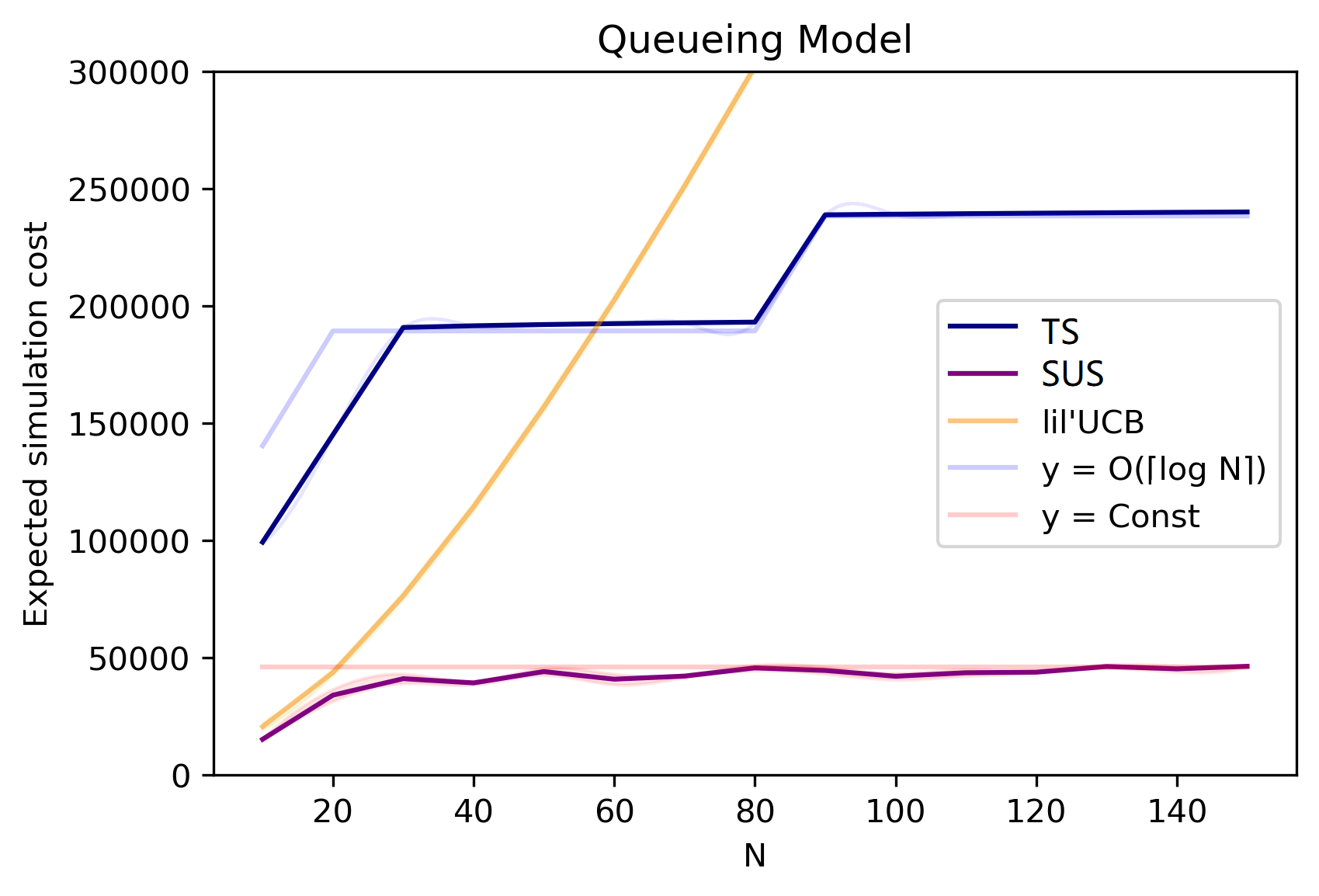}
    \caption{}
\end{subfigure}\hfill
\begin{subfigure}{.49\textwidth}
    \centering
    \includegraphics[scale=0.3]{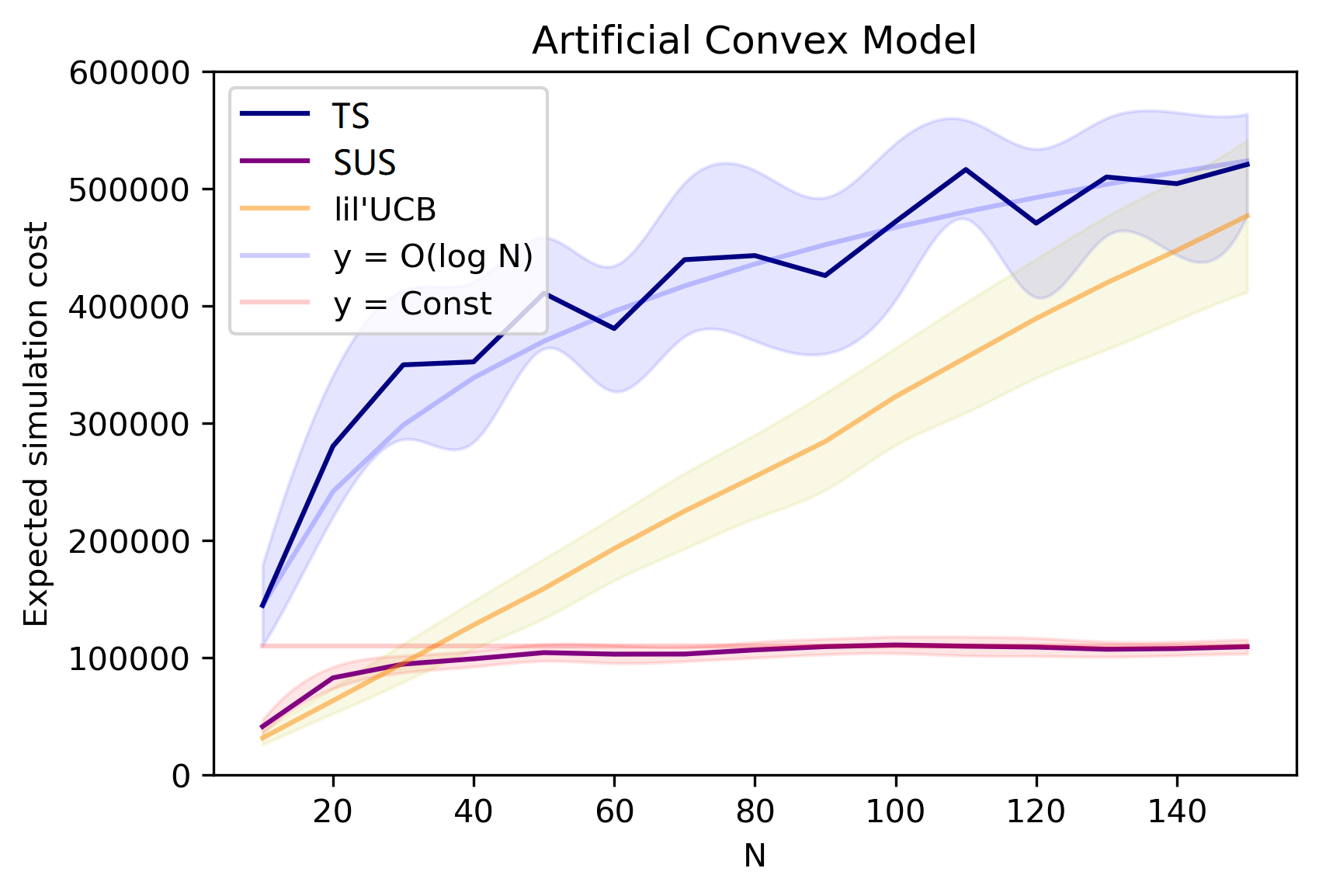}
    \caption{}
\end{subfigure}\hfill
\caption{The expected simulation cost of tri-section sampling (TS), shrinking uniform sampling (SUS) and lil'UCB algorithms in the one-dimensional case. \textbf{(a)} Optimal allocation problem. \textbf{(b)} One-dimensional separable convex function minimization.} 
\label{fig:one-dim}
\end{center}
\end{figure}

We first compare the performance of the tri-section sampling (TS) algorithm and the shrinking uniform sampling (SUS) algorithm on the optimal allocation problem in Section 6.1 and the closed-form convex function minimization problem in Section 6.2. As a comparison to the existing algorithms, we also implement the state-of-the-art algorithm for the best arm identification problem, namely the lil'UCB algorithm \citep{jamieson2014lil}. The best arm identification problem is equivalent to problem \eqref{eqn:obj} without any convexity structure.
We consider problems with dimension $d=1$ and scale $N\in\{10,20,\dots,150\}$. The expected simulation cost is computed by averaging $400$ independent solving processes. For the optimal allocation problem, we set the optimality parameters for the PGS guarantee as $\epsilon=1$ and $\delta=10^{-6}$. An upper bound on the variance is estimated as $\sigma^2=10$. For the convex function minimization problem, we generate each $c_i$ from the uniform distribution on $[0.75,1.25]$ and $x^*_i$ from the discrete uniform distribution on $\{1,2,\dots, \lfloor 0.3N\rfloor\}$. The optimality parameters are chosen as $\epsilon=0.2$ and $\delta=10^{-6}$ and the variance is set to be $\sigma^2=1$.

It is observed that both algorithms satisfy the given PGS guarantee on the synthetic convex function minimization problem, namely, the $\epsilon$-optimality is satisfied for all implementations. We then plot the estimated expected simulation costs in Figure~\ref{fig:one-dim}. For the optimal allocation problem, the expected simulation costs of the TS and SUS algorithms approximately have $O( \lceil \log N \rceil )$ and almost $O(1)$ dependence on the scale $N$, respectively. The expected simulation cost of the SUS algorithm is almost independent of $N$ and this verifies our theoretical analysis. For the synthetic convex function minimization problem, same as the queueing example, the estimated expected simulation costs of the TS algorithm and SUS algorithm have $O(\log N)$ and almost $O(1)$ dependence on $N$, respectively. This again verifies our theoretical analysis. Moreover, both algorithms outperform the lil'UCB algorithm, when $N$ is large. The numerical results show that our proposed algorithms can efficiently solve large-scale one-dimensional convex problems.




\subsection{Numerical Results: Subgradient Descent and Localization Methods}

We next compare the performances of the truncated stochastic subgradient descent algorithm \citep{zhang2020discrete} and stochastic localization methods proposed in this work. We first consider the separable convex function minimization problem, where we can compute the optimality gap and verify the $\epsilon$-optimality.
The dimension and scale of the separable convex model are chosen as $d\in\{2,6,10,15\}$ and $N\in\{50,500,5000\}$. The optimality guarantee parameters are chosen as $\epsilon=d$ and $\delta=10^{-6}$, respectively. 
The empirical choice of $\epsilon$ ensures that any $\epsilon$-optimal solution $x^0$ satisfies $\|x^0-x^*\|_1 \leq d^{5/6} \ll N$.
We compute the average simulation cost of $10$ independently generated models to estimate the expected simulation cost. 
Moreover, early stopping conditions are designed to terminate algorithms early when little progress is made at any iteration. For the subgradient descent method, we maintain the empirical mean of stochastic objective function values up to the current iteration and terminate the algorithm if the empirical mean does not decrease by $O(\epsilon/\sqrt{N})$ after $O[d\epsilon^{-2}\log(1/\delta)]$ consecutive iterations. For stochastic cutting-plane methods, we terminate the algorithm if the empirical mean of the objective function of the last $5$ iterations does not decrease by $\epsilon/d$. For the dimension reduction method, we terminate the algorithm early if the polytope is empty. Furthermore, we have observed that using $(N\epsilon/4,\delta/4)$-$\mathcal{SO}$ oracles in localization methods is sufficient for producing high-probability guarantees on this example.

We summarize the results in Table~\ref{tab:multi-dim}, where the coverage rate refers to the percentage of implementations that produce an $\epsilon$-optimal solution
Since the coverage rates of the algorithms are all equal to $100\%$, the PGS guarantee is likely to be satisfied by all of the algorithms. 
The performances of localization methods are better than the subgradient descent algorithm in all settings especially for the large-scale instances. The simulation cost of the random walk-based cutting-plane method is better than the Vaidya's cutting-plane method, which may be a result of the extra $\log(d)$ term in the simulation cost; see the discussion in Remark \ref{rmk:3}. The dimension reduction method has the best performance on examples with $N=500,\,5000$ and has the advantage of not requiring any knowledge about the Lipschitz constant. \revisee{From the experimental results, we can see that the empirical performances of proposed algorithms are sometimes better than their theoretical guarantees.}

\begin{table}[thp]
\small
\caption{Simulation cost and coverage rate of different algorithms on separable convex functions. }\label{tab:multi-dim}
  \begin{center}  
      \begin{tabular}{p{0.6cm}<{\centering}p{0.6cm}<{\centering}p{1.5cm}<{\centering}p{1.5cm}<{\centering}p{1.5cm}<{\centering}p{1.5cm}<{\centering}p{1.5cm}<{\centering}p{1.5cm}<{\centering}p{1.5cm}<{\centering}p{1.5cm}}
          \toprule[2pt]
          \multicolumn{2}{c}{} & \multicolumn{2}{c}{\textbf{Search Methods}}    & \multicolumn{6}{c}{\textbf{Localization Methods (this work)}}  \\
          \cline{5-10}
          \multicolumn{2}{c}{\textbf{Params.}} & \multicolumn{2}{c}{{SubGD}}    & \multicolumn{2}{c}{{Vaidya's}}  & \multicolumn{2}{c}{{Random Walk}} & \multicolumn{2}{c}{{Dim Reduction}} \\
          d & N     & Cost & Rate(\%)       & Cost & Rate(\%)       & Cost & Rate(\%)       & Cost & Rate(\%)  \\ 
          \midrule[1pt]
        2 & 50    & 1.08e3  & 100.0     & 2.74e2  & 100.0     & 1.66e2  & 100.0     & \textbf{1.56e2}  & 100.0       \\
        \hline
        2 & 500   & 2.54e4  & 100.0     & 6.54e2  & 100.0     & 2.32e2  & 100.0     & \textbf{2.08e2}  & 100.0       \\
        \hline
        2 & 5000  & 3.97e5  & 100.0     & 1.13e4  & 100.0     & 5.29e2  & 100.0     & \textbf{4.66e2}  & 100.0       \\
        \hline
        6 & 50    & 5.00e3  & 100.0     & 4.13e2  & 100.0     & \textbf{3.36e2}  & 100.0     & 4.05e2  & 100.0       \\
        \hline
        6 & 500   & 4.75e4  & 100.0     & 1.34e3  & 100.0     & 1.65e3  & 100.0     & \textbf{6.45e2}  & 100.0       \\
        \hline
        6 & 5000  & 2.72e6  & 100.0     & 8.15e4  & 100.0     & 4.75e3  & 100.0     & \textbf{8.25e2}  & 100.0       \\
        \hline
        10& 50    & 8.46e3  & 100.0     & 7.98e2  & 100.0     & \textbf{7.70e2}  & 100.0     & 8.34e2  & 100.0       \\
        \hline
        10& 500   & 6.32e4  & 100.0     & 6.57e3  & 100.0     & 2.16e3  & 100.0     & \textbf{1.48e3}  & 100.0       \\
        \hline
        10& 5000  & 7.76e6  & 100.0     & 2.42e5  & 100.0     & 8.03e3  & 100.0     & \textbf{2.02e3}  & 100.0       \\
        \hline
        15& 50    & 1.23e4  & 100.0     & 1.50e3  & 100.0     & \textbf{1.91e3}  & 100.0     & 2.18e3  & 100.0       \\
        \hline
        15& 500   & 2.83e5  & 100.0     & 2.66e4  & 100.0     & 1.06e4  & 100.0     & \textbf{3.19e3}  & 100.0       \\
        \hline
        15& 5000  & 1.85e7  & 100.0     & 1.96e6  & 100.0     & 1.55e5  & 100.0     & \textbf{4.85e3}  & 100.0       \\
          \bottomrule[2pt]
      \end{tabular}
  \end{center}
\end{table}

We then consider the multi-dimensional resource allocation problem. We first fix the dimension (number of time slots) to be $d=4$ and compare the performance with the scale $N\in\{10,20,30,40,50\}$, and we then fix the scale to be $N=10$ and compare the performance with the dimension $d\in\{4,8,12,16,20,24\}$. The parameters of the problem are chosen as $\lambda=1$ and $C=10$, and the optimality guarantee parameters are $\epsilon=N/2$ and $\delta=10^{-6}$. An upper bound on the variance is estimated as $\sigma^2=30\sqrt{N}$. For each problem setup, we average the results of $10$ independent implementations to estimate the expected simulation cost and the objective value of the returned solution. 
The results are summarized in Table \ref{tab:multi-dim-queue}. It is observed that the dimension reduction method achieves the best performance in all cases, although its simulation costs have a faster growth rate than other methods. The stochastic cutting-plane methods also outperform the subgradient descent algorithm when the dimension is $4$. The truncated stochastic subgradient descent method returns the smallest objective values except the case when $(d,N)=(4,50)$, and the objective values returned by other algorithms are not much larger than the truncated stochastic subgradient descent method. This is possible since we are searching for PGS solutions and an optimality gap smaller than $\epsilon=N/2$ is acceptable.

\begin{table}[t]
\revise{
\small
\caption{Simulation cost and objective value of different algorithms on the resource allocation problem. }\label{tab:multi-dim-queue}
  \begin{center}  
      \begin{tabular}{p{0.6cm}<{\centering}p{0.6cm}<{\centering}p{1.5cm}<{\centering}p{1.5cm}<{\centering}p{1.5cm}<{\centering}p{1.5cm}<{\centering}p{1.5cm}<{\centering}p{1.5cm}<{\centering}p{1.5cm}<{\centering}p{1.5cm}}
          \toprule[2pt]
          \multicolumn{2}{c}{} & \multicolumn{2}{c}{\textbf{Search Methods}}    & \multicolumn{6}{c}{\textbf{Localization Methods (this work)}}  \\
          \cline{5-10}
          \multicolumn{2}{c}{\textbf{Params.}} & \multicolumn{2}{c}{{SubGD}}    & \multicolumn{2}{c}{{Vaidya's}}  & \multicolumn{2}{c}{{Random Walk}} & \multicolumn{2}{c}{{Dim Reduction}} \\
          d & N     & Cost & Obj.       & Cost & Obj. & Cost & Obj. & Cost & Obj.  \\ 
          \midrule[1pt]
          4 & 10    & 3.06e5  & 2.13e1  & 9.89e4  & 2.19e1  & 6.92e4  & 2.47e1  & \textbf{2.42e4}  & 2.40e1  \\
          \hline
          4 & 20    & 1.08e5  & 3.41e1  & 3.64e4  & 3.42e1  & 2.45e4  & 3.73e1  & \textbf{1.40e4}  & 3.44e1  \\
          \hline
          4 & 30    & 7.79e4  & 4.59e1  & 1.94e4  & 4.65e1  & 1.33e4  & 5.10e1  & \textbf{9.21e3}  & 4.59e1  \\
          \hline
          4 & 40    & 5.06e4  & 5.73e1  & 1.24e4  & 5.86e1  & 8.68e3  & 6.35e1  & \textbf{6.31e3}  & 5.75e1  \\
          \hline
          4 & 50    & 4.50e4  & 6.91e1  & 9.24e3  & 6.98e1  & 6.22e3  & 7.49e1  & \textbf{4.03e3}  & 6.67e1  \\
          \midrule[1pt]
          8 & 10    & 1.20e6  & 2.01e1  & 7.27e5  & 2.12e1  & 5.53e5  & 2.17e1  & \textbf{1.48e5}  & 2.12e1  \\
          \hline
          12& 10    & 2.69e6  & 1.90e1  & 2.49e6  & 2.07e1  & 1.86e6  & 2.13e1  & \textbf{6.10e5}  & 2.01e1  \\
          \hline
          16& 10    & 4.78e6  & 1.83e1  & 6.64e6  & 2.02e1  & 4.43e6  & 2.04e1  & \textbf{1.59e6}  & 1.91e1  \\
          \hline
          20& 10    & 7.45e6  & 1.78e1  & 1.38e7  & 2.01e1  & 8.65e6  & 2.04e1  & \textbf{3.21e6}  & 1.81e1  \\
          \hline
          24& 10    & 1.43e7  & 1.71e1  & 2.42e7  & 1.99e1  & 1.49e7  & 2.04e1  & \textbf{8.54e6}  & 1.76e1  \\
          \bottomrule[2pt]
      \end{tabular}
  \end{center}
  }
\end{table}

In summary, based on the results from numerical results, the shrinking uniform sampling algorithm and the dimension reduction method provide a more efficient choice for large-scale convex discrete optimization via simulation problems, and they have the advantage that no prior information about the objective function is required except the $L^\natural$-convexity.

\section{Conclusion}
\label{sec:cls}

In this paper, algorithms based on the idea of localization are proposed for large-scale convex discrete optimization via simulation problems. The simulation-optimization algorithms are theoretically guaranteed to identify a solution whose corresponding objective value is close to the optimal objective value up to a given precision with high probability. Moreover, the efficiency of the developed algorithms is evaluated by obtaining upper bounds on the expected simulation cost. Specifically, in the one-dimensional case, we propose the shrinking uniform sampling method, which has an expected simulation cost as $O[ \epsilon^{-2}(\log(N) + \log(1/\delta))]$, which attains the best achievable performance under the asymptotic criterion \citep{kaufmann2016complexity}, i.e., when $\delta \rightarrow 0$. For the multi-dimensional case, we combine the idea of localization with subgradient information. The dimension reduction algorithm is designed using a new framework to extend deterministic cutting-plane methods. The expected simulation cost is proven to be upper bounded by a constant that is independent of the Lipschitz constant. In addition, all proposed algorithms do not require prior knowledge about the Lipschitz constant. Finally, an adaptive algorithm is designed to avoid the requirement that the variance of the noise should be estimated a priori. Numerical results on both synthetic and queueing models demonstrate that the proposed algorithms have better performances compared to benchmark methods especially when the problem scale is large.

\bibliographystyle{informs2014} 
\bibliography{ref}

\begin{thebibliography}{54}
\providecommand{\natexlab}[1]{#1}
\providecommand{\url}[1]{\texttt{#1}}
\providecommand{\urlprefix}{URL }

\bibitem[{Agarwal et~al.(2011)Agarwal, Foster, Hsu, Kakade, \protect\BIBand{}
  Rakhlin}]{agarwal2011stochastic}
Agarwal A, Foster DP, Hsu DJ, Kakade SM, Rakhlin A (2011) Stochastic convex
  optimization with bandit feedback. \emph{Advances in Neural Information
  Processing Systems}, 1035--1043.

\bibitem[{Alman \protect\BIBand{} Williams(2020)}]{alman2020refined}
Alman J, Williams VV (2020) A refined laser method and faster matrix
  multiplication. \emph{arXiv preprint arXiv:2010.05846} .

\bibitem[{Altman et~al.(2003)Altman, Gaujal, \protect\BIBand{}
  Hordijk}]{altman2003discrete}
Altman E, Gaujal B, Hordijk A (2003) \emph{Discrete-event control of stochastic
  networks: Multimodularity and regularity} (springer).

\bibitem[{Bechhofer(1954)}]{bechhofer1954single}
Bechhofer RE (1954) A single-sample multiple decision procedure for ranking
  means of normal populations with known variances. \emph{The Annals of
  Mathematical Statistics} 16--39.

\bibitem[{Bertsimas \protect\BIBand{} Vempala(2004)}]{bertsimas2004solving}
Bertsimas D, Vempala S (2004) Solving convex programs by random walks.
  \emph{Journal of the ACM (JACM)} 51(4):540--556.

\bibitem[{Burnetas \protect\BIBand{} Katehakis(1996)}]{burnetas1996optimal}
Burnetas AN, Katehakis MN (1996) Optimal adaptive policies for sequential
  allocation problems. \emph{Advances in Applied Mathematics} 17(2):122--142.

\bibitem[{Chen et~al.(2016)Chen, Gupta, \protect\BIBand{} Li}]{chen2016pure}
Chen L, Gupta A, Li J (2016) Pure exploration of multi-armed bandit under
  matroid constraints. \emph{Conference on Learning Theory}, 647--669.

\bibitem[{Chen \protect\BIBand{} Li(2020)}]{chen2020discrete}
Chen X, Li M (2020) Discrete convex analysis and its applications in
  operations: A survey. \emph{Production and Operations Management} .

\bibitem[{Dyer \protect\BIBand{} Proll(1977)}]{dyer1977note}
Dyer M, Proll L (1977) Note—on the validity of marginal analysis for
  allocating servers in m/m/c queues. \emph{Management Science}
  23(9):1019--1022.

\bibitem[{Eckman \protect\BIBand{} Henderson(2018)}]{eckman2018fixed}
Eckman DJ, Henderson SG (2018) Fixed-confidence, fixed-tolerance guarantees for
  selection-of-the-best procedures. Technical report, Working paper, Cornell
  University, School of Operations Research and~….

\bibitem[{Eckman et~al.(2020)Eckman, Plumlee, \protect\BIBand{}
  Nelson}]{Nelson2020}
Eckman DJ, Plumlee M, Nelson BL (2020) Plausible screening using functional
  properties for simulations with large solution spaces, working paper.

\bibitem[{Eckman et~al.(2021)Eckman, Plumlee, \protect\BIBand{}
  Nelson}]{eckman2021flat}
Eckman DJ, Plumlee M, Nelson BL (2021) Flat chance! using stochastic gradient
  estimators to assess plausible optimality for convex functions. \emph{2021
  Winter Simulation Conference (WSC)}, 1--18 (IEEE).

\bibitem[{Even-Dar et~al.(2002)Even-Dar, Mannor, \protect\BIBand{}
  Mansour}]{even2002pac}
Even-Dar E, Mannor S, Mansour Y (2002) Pac bounds for multi-armed bandit and
  markov decision processes. \emph{International Conference on Computational
  Learning Theory}, 255--270 (Springer).

\bibitem[{Freund et~al.(2017)Freund, Henderson, \protect\BIBand{}
  Shmoys}]{freund2017minimizing}
Freund D, Henderson SG, Shmoys DB (2017) Minimizing multimodular functions and
  allocating capacity in bike-sharing systems. \emph{International Conference
  on Integer Programming and Combinatorial Optimization}, 186--198 (Springer).

\bibitem[{Fujishige(2005)}]{fujishige2005submodular}
Fujishige S (2005) \emph{Submodular functions and optimization} (Elsevier).

\bibitem[{Gong \protect\BIBand{} Chao(2013)}]{gong2013optimal}
Gong X, Chao X (2013) Optimal control policy for capacitated inventory systems
  with remanufacturing. \emph{Operations Research} 61(3):603--611.

\bibitem[{Hong et~al.(2020)Hong, Fan, \protect\BIBand{} Luo}]{hong2020review}
Hong LJ, Fan W, Luo J (2020) Review on ranking and selection: A new
  perspective. \emph{arXiv preprint arXiv:2008.00249} .

\bibitem[{Hong \protect\BIBand{} Nelson(2006)}]{hong2006discrete}
Hong LJ, Nelson BL (2006) Discrete optimization via simulation using compass.
  \emph{Operations Research} 54(1):115--129.

\bibitem[{Hong et~al.(2015)Hong, Nelson, \protect\BIBand{}
  Xu}]{hong2015discrete}
Hong LJ, Nelson BL, Xu J (2015) Discrete optimization via simulation.
  \emph{Handbook of simulation optimization}, 9--44 (Springer).

\bibitem[{Huh \protect\BIBand{} Janakiraman(2010)}]{huh2010optimal}
Huh WT, Janakiraman G (2010) On the optimal policy structure in serial
  inventory systems with lost sales. \emph{Operations Research} 58(2):486--491.

\bibitem[{Jamieson et~al.(2014)Jamieson, Malloy, Nowak, \protect\BIBand{}
  Bubeck}]{jamieson2014lil}
Jamieson K, Malloy M, Nowak R, Bubeck S (2014) lil’ucb: An optimal
  exploration algorithm for multi-armed bandits. \emph{Conference on Learning
  Theory}, 423--439.

\bibitem[{Jian et~al.(2016)Jian, Freund, Wiberg, \protect\BIBand{}
  Henderson}]{jian2016simulation}
Jian N, Freund D, Wiberg HM, Henderson SG (2016) Simulation optimization for a
  large-scale bike-sharing system. \emph{2016 Winter Simulation Conference
  (WSC)}, 602--613 (IEEE).

\bibitem[{Jiang(2020)}]{jiang2020minimizing}
Jiang H (2020) Minimizing convex functions with integral minimizers.
  \emph{arXiv preprint arXiv:2007.01445} .

\bibitem[{Jiang et~al.(2020)Jiang, Lee, Song, \protect\BIBand{}
  Wong}]{jiang2020improved}
Jiang H, Lee YT, Song Z, Wong SCw (2020) An improved cutting plane method for
  convex optimization, convex-concave games, and its applications.
  \emph{Proceedings of the 52nd Annual ACM SIGACT Symposium on Theory of
  Computing}, 944--953.

\bibitem[{Karnin et~al.(2013)Karnin, Koren, \protect\BIBand{}
  Somekh}]{karnin2013almost}
Karnin Z, Koren T, Somekh O (2013) Almost optimal exploration in multi-armed
  bandits. \emph{International Conference on Machine Learning}, 1238--1246.

\bibitem[{Kaufmann et~al.(2016)Kaufmann, Capp{\'e}, \protect\BIBand{}
  Garivier}]{kaufmann2016complexity}
Kaufmann E, Capp{\'e} O, Garivier A (2016) On the complexity of best-arm
  identification in multi-armed bandit models. \emph{The Journal of Machine
  Learning Research} 17(1):1--42.

\bibitem[{Lai \protect\BIBand{} Robbins(1985)}]{lai1985asymptotically}
Lai TL, Robbins H (1985) Asymptotically efficient adaptive allocation rules.
  \emph{Advances in applied mathematics} 6(1):4--22.

\bibitem[{Lee \protect\BIBand{} Valiant(2020)}]{lee2020optimal}
Lee JC, Valiant P (2020) Optimal sub-gaussian mean estimation in $\mathbb{R}$.
  \emph{arXiv preprint arXiv:2011.08384} .

\bibitem[{Lee et~al.(2015)Lee, Sidford, \protect\BIBand{} Wong}]{lee2015faster}
Lee YT, Sidford A, Wong SCw (2015) A faster cutting plane method and its
  implications for combinatorial and convex optimization. \emph{2015 IEEE 56th
  Annual Symposium on Foundations of Computer Science}, 1049--1065 (IEEE).

\bibitem[{Lenstra et~al.(1982)Lenstra, Lenstra, \protect\BIBand{}
  Lov{\'a}sz}]{lenstra1982factoring}
Lenstra AK, Lenstra HW, Lov{\'a}sz L (1982) Factoring polynomials with rational
  coefficients. \emph{Mathematische annalen} 261(ARTICLE):515--534.

\bibitem[{Liang et~al.(2014)Liang, Narayanan, \protect\BIBand{}
  Rakhlin}]{liang2014zeroth}
Liang T, Narayanan H, Rakhlin A (2014) On zeroth-order stochastic convex
  optimization via random walks. \emph{arXiv preprint arXiv:1402.2667} .

\bibitem[{Lov{\'a}sz(1983)}]{lovasz1983submodular}
Lov{\'a}sz L (1983) Submodular functions and convexity. \emph{Mathematical
  programming the state of the art}, 235--257 (Springer).

\bibitem[{Luo et~al.(2015)Luo, Hong, Nelson, \protect\BIBand{}
  Wu}]{luo2017fully}
Luo J, Hong LJ, Nelson BL, Wu Y (2015) Fully sequential procedures for
  large-scale ranking-and-selection problems in parallel computing
  environments. \emph{Operations Research} 63(5):1177--1194.

\bibitem[{Ma \protect\BIBand{} Henderson(2017)}]{ma2017efficient}
Ma S, Henderson SG (2017) An efficient fully sequential selection procedure
  guaranteeing probably approximately correct selection. \emph{2017 Winter
  Simulation Conference (WSC)}, 2225--2236 (IEEE).

\bibitem[{Ma \protect\BIBand{} Henderson(2019)}]{mahen17b}
Ma S, Henderson SG (2019) Predicting the simulation budget in ranking and
  selection procedures. \emph{{ACM} Transactions on Modeling and Computer
  Simulation} 29(3):Article 14, 1--25.

\bibitem[{Murota(2003)}]{murota2003discrete}
Murota K (2003) Discrete convex analysis. \emph{Society for Industrial and
  Applied Mathematics} (Citeseer).

\bibitem[{Nelson(2010)}]{nelson2010optimization}
Nelson BL (2010) Optimization via simulation over discrete decision variables.
  \emph{Risk and Optimization in an Uncertain World}, 193--207 (Informs).

\bibitem[{Ni et~al.(2017)Ni, Ciocan, Henderson, \protect\BIBand{}
  Hunter}]{nietal17}
Ni EC, Ciocan DF, Henderson SG, Hunter SR (2017) Efficient ranking and
  selection in high performance computing environments. \emph{Operations
  Research} 65(3):821--836.

\bibitem[{Pang et~al.(2012)Pang, Chen, \protect\BIBand{} Feng}]{pang2012note}
Pang Z, Chen FY, Feng Y (2012) A note on the structure of joint
  inventory-pricing control with leadtimes. \emph{Operations Research}
  60(3):581--587.

\bibitem[{Ragavan et~al.(2021)Ragavan, Hunter, Pasupathy, \protect\BIBand{}
  Taaffe}]{ragavan2021adaptive}
Ragavan PK, Hunter SR, Pasupathy R, Taaffe MR (2021) Adaptive sampling line
  search for local stochastic optimization with integer variables.
  \emph{Mathematical Programming} 1--30.

\bibitem[{Shaked \protect\BIBand{} Shanthikumar(1988)}]{shaked1988stochastic}
Shaked M, Shanthikumar JG (1988) Stochastic convexity and its applications.
  \emph{Advances in Applied Probability} 20(2):427--446.

\bibitem[{Shi et~al.(2000)}]{shi2000nested}
Shi L, et~al. (2000) Nested partitions method for stochastic optimization.
  \emph{Methodology and Computing in Applied probability} 2(3):271--291.

\bibitem[{Singhvi et~al.(2015)Singhvi, Singhvi, Frazier, Henderson, O'Mahony,
  Shmoys, \protect\BIBand{} Woodard}]{singhvi2015predicting}
Singhvi D, Singhvi S, Frazier PI, Henderson SG, O'Mahony E, Shmoys DB, Woodard
  DB (2015) Predicting bike usage for new york city's bike sharing system.
  \emph{AAAI Workshop: Computational Sustainability} (Citeseer).

\bibitem[{Sun et~al.(2014)Sun, Hong, \protect\BIBand{} Hu}]{sun2014balancing}
Sun L, Hong LJ, Hu Z (2014) Balancing exploitation and exploration in discrete
  optimization via simulation through a gaussian process-based search.
  \emph{Operations Research} 62(6):1416--1438.

\bibitem[{Vaidya(1996)}]{vaidya1996new}
Vaidya PM (1996) A new algorithm for minimizing convex functions over convex
  sets. \emph{Mathematical programming} 73(3):291--341.

\bibitem[{Wainwright(2019)}]{wainwright2019high}
Wainwright MJ (2019) \emph{High-dimensional statistics: A non-asymptotic
  viewpoint}, volume~48 (Cambridge University Press).

\bibitem[{Wang et~al.(2021)Wang, Xu, Hu, \protect\BIBand{}
  Chen}]{wang2020optimal}
Wang T, Xu J, Hu JQ, Chen CH (2021) Optimal computing budget allocation for
  regression with gradient information. \emph{Automatica} 134:109927.

\bibitem[{Wolff \protect\BIBand{} Wang(2002)}]{wolff2002convexity}
Wolff RW, Wang CL (2002) On the convexity of loss probabilities. \emph{Journal
  of applied probability} 402--406.

\bibitem[{Xu et~al.(2010)Xu, Nelson, \protect\BIBand{} Hong}]{xu2010industrial}
Xu J, Nelson BL, Hong JL (2010) Industrial strength compass: A comprehensive
  algorithm and software for optimization via simulation. \emph{ACM
  Transactions on Modeling and Computer Simulation (TOMACS)} 20(1):1--29.

\bibitem[{Xu \protect\BIBand{} Nelson(2013)}]{xu2013empirical}
Xu WL, Nelson BL (2013) Empirical stochastic branch-and-bound for optimization
  via simulation. \emph{Iie Transactions} 45(7):685--698.

\bibitem[{Xu et~al.(2016)Xu, Lin, \protect\BIBand{} Yang}]{xu2016accelerated}
Xu Y, Lin Q, Yang T (2016) Accelerated stochastic subgradient methods under
  local error bound condition. \emph{arXiv preprint arXiv:1607.01027} .

\bibitem[{Zhang et~al.(2020)Zhang, Zheng, \protect\BIBand{}
  Lavaei}]{zhang2020discrete}
Zhang H, Zheng Z, Lavaei J (2020) Discrete convex simulation optimization.
  \emph{arXiv preprint arXiv:2010.16250} .

\bibitem[{Zhong \protect\BIBand{} Hong(2019)}]{zhong2019knockout}
Zhong Y, Hong LJ (2019) Knockout-tournament procedures for large-scale ranking
  and selection in parallel computing environments. \emph{accepted} .

\bibitem[{Zipkin(2008)}]{zipkin2008structure}
Zipkin P (2008) On the structure of lost-sales inventory models.
  \emph{Operations research} 56(4):937--944.

\end{thebibliography}

\ECSwitch


\ECHead{Proofs of Statements}

\section{Algorithms and Complexity Analysis for the PCS-IZ Guarantee}\label{ec:pcsiz}

In this section, we provide modified simulation-optimization algorithms for the PCS-IZ guarantee. We assume that the objective value of any sub-optimal choice of decision variables is at least $c$ larger than the optimal objective value, where the indifference zone parameter $c>0$ is known a priori.

\subsection{Modified Tri-section Sampling Algorithm for the PCS-IZ Guarantee}
\label{ec:one-dim-ada}

We first consider the one-dimensional case. When the prior information about the indifference zone parameter $c$ is available, we can modify the tri-section sampling (TS) algorithm to achieve a better simulation cost. The modified algorithm also consists of two parts: the shrinkage of intervals and a sub-problem with at most $3$ points. The improvement is achieved by a weaker condition for the comparison of objective values at two $3$-quantiles. We give the modified algorithm in Algorithm \ref{alg:one-dim-iz} and omit those lines that are the same as Algorithm \ref{alg:one-dim}. 
\bigskip
\begin{breakablealgorithm}
\caption{Tri-section sampling algorithm for the PCS-IZ guarantee}
\label{alg:one-dim-iz}
\begin{algorithmic}[1]
\Require{Model $\mathcal{X}=[N],(\mathsf{Y},\mathcal{B}_\mathsf{Y}),F(x,\xi_x)$, optimality guarantee parameter $\delta$, indifference zone parameter $c$.}
\Ensure{An $(c,\delta)$-PCS-IZ solution $x^*$ to problem \eqref{eqn:obj}.}
\State Set upper and lower bounds of current interval $\revise{x_L}\leftarrow 1,\revise{x_U}\leftarrow N$. 
\State Set maximal number of comparisons $T_{max}\leftarrow\log_{1.5}(N) + 2$.
\While{$\revise{x_U} - \revise{x_L} > 2$} \Comment{Iterate until there are at most $3$ decisions.}
     \Statex ...
    \makeatletter
    \setcounter{ALG@line}{7}
    \makeatother
    \State Stop sampling if one of the following conditions holds:
    \begin{align*}  (i)\quad& \hat{F}_n(\revise{q_{1/3}})-h_{1/3}\geq\hat{F}_n(\revise{q_{2/3}})+h_{2/3},\\
                    (ii)\quad& \hat{F}_n(\revise{q_{1/3}})+h_{1/3}\leq\hat{F}_n(\revise{q_{2/3}})-h_{2/3},\\
                    (iii)\quad& h_{1/3} \leq (\revise{q_{2/3}} - \revise{q_{1/3}})\cdot c/5\text{ and } h_{2/3} \leq (\revise{q_{2/3}} - \revise{q_{1/3}})\cdot c/5. \end{align*}
    \Statex ...
    \makeatletter
    \setcounter{ALG@line}{12}
    \makeatother
    \If{$h_{1/3} \leq (\revise{q_{2/3}} - \revise{q_{1/3}})\cdot c/5$ and $h_{2/3} \leq (\revise{q_{2/3}} - \revise{q_{1/3}})\cdot c/5$}
        \State Update $\revise{x_L}\leftarrow \revise{q_{1/3}}$ and $\revise{x_U}\leftarrow \revise{q_{2/3}}$.
    \EndIf
\EndWhile
\State Simulate $F(x,\xi_x)$ for all $x\in\{\revise{x_L},\dots, \revise{x_U}\}$ until the $1-\delta/(2T_{max})$ confidence half-widths are smaller than $c/3$.
\Statex \Comment{Now $\revise{x_U}-\revise{x_L}\leq 2$.}
\State Return the point in $\{\revise{x_L},\dots, \revise{x_U}\}$ with minimal empirical mean.
\end{algorithmic}
\end{breakablealgorithm}
\bigskip
The following theorem proves the correctness and the expected simulation cost of the modified tri-section sampling algorithm.

\begin{theorem}\label{thm:one-dim-iz}
Suppose that Assumptions \ref{asp:1}-\ref{asp:4} hold. The modified tri-section sampling algorithm is a $[(c,\delta)\text{-PCS-IZ},\mathcal{MC}_c]$-algorithm. Furthermore, we have
\[ T(\delta,\mathcal{MC}_{c}) = O\left[ \frac{1}{c^2}\log\left( \frac{\log(N)}{\delta} \right) + \log(N) \right] = \tilde{O}\left[ \frac{1}{c^2}\log\left( \frac{1}{\delta} \right) \right]. \]
%
\end{theorem}
\begin{proof}{Proof of Theorem \ref{thm:one-dim-iz}.}
The proof of is provided in \ref{ec:one-dim-iz}.
\hfill\Halmos\end{proof}

By Theorem \ref{thm:one-dim-iz}, the expected simulation cost for the PCS-IZ guarantee is asymptotically independent of the number of points $N$, when the failing probability $\delta$ is sufficiently small. 
%
%
If the shrinking uniform sampling (SUS) algorithm is used for the PCS-IZ guarantee, the algorithm also achieves an $\tilde{O}(c^{-2}\log(1/\delta))$ expected simulation cost by setting the optimality parameter $\epsilon=c/2$. This is because the objective values of sub-optimal solutions are larger than that of the optimal solution by at least $c$ and because the solution satisfying the $(c/2,\delta)$-PGS guarantee also satisfies the $(c,\delta)$-PCS-IZ guarantee. Hence, for both the modified TS algorithm and the SUS algorithm, the asymptotic simulation cost has an upper bound that is independent of $N$. However, we note that the space complexity of the modified TS algorithm is only $\tilde{O}(\log(N))$, whereas the SUS algorithm requires $O(N)$ memory space. Therefore, the modified TS algorithm is preferred for the PCS-IZ guarantee.

\subsection{Modified Stochastic Cutting-plane Methods for the PCS-IZ Guarantee}
\label{ec:multi-dim-weak}

In the multi-dimensional case, we develop modified stochastic cutting-plane methods for the PCS-IZ guarantee. Using the same adaptive acceleration scheme as in~\citet{zhang2020discrete}, the indifference zone parameter can help reduce the dependence of the simulation cost on the problem scale $N$. We give the pseudo-code of the accelerated stochastic cutting-plane method in Algorithm \ref{alg:multi-dim-weak-iz}. 
\bigskip
\begin{breakablealgorithm}
\caption{Stochastic cutting-plane method for the PCS-IZ guarantee}
\label{alg:multi-dim-weak-iz}
\begin{algorithmic}[1]
\Require{Model $\mathcal{X},(\mathsf{Y},\mathcal{B}_\mathsf{Y}),F(x,\xi_x)$, optimality guarantee parameter $\delta$, indifference zone parameter $c$, Lipschitz constant $L$, $(\epsilon,\delta)$-$\mathcal{SO}$ oracle $\hat{g}$.}
\Ensure{An $(c,\delta)$-PCS-IZ solution $x^*$ to problem \eqref{eqn:obj}.}
\State Set the initial guarantee $\epsilon_0\leftarrow cN/4$.
\State Set the number of epochs $E\leftarrow \lceil \log_2(N) \rceil + 1$.
\State Set the initial searching space $\mathcal{Y}_0 \leftarrow [1,N]^d$.
\For{$e=0,\dots,E-1$}
    \State Use Algorithm \ref{alg:multi-dim-weak} to get an $(\epsilon_e,\delta/(2E))$-PGS solution $x_e$ in $\mathcal{Y}_e$.
    \State Update guarantee $\epsilon_{e+1}\leftarrow \epsilon_e / 2$.
    \State Update the searching space $\mathcal{Y}_{e+1}\leftarrow \mathcal{N}(x_e, 2^{-e-2}N)$.
\EndFor
\State Round $x_{E-1}$ to an integral point by Algorithm \ref{alg:multi-dim-round}.
\end{algorithmic}
\end{breakablealgorithm}
\bigskip
\revise{We can prove the correctness and estimate the expected simulation cost of the accelerated algorithm in the same way as Theorem 8 in \citet{zhang2020discrete}. Thus, we omit the proof.}
\begin{theorem}\label{thm:weak-2}
Suppose that Assumptions \ref{asp:1}-\ref{asp:5} hold. The accelerated stochastic cutting-plane method returns a $(c,\delta)$-PCS-IZ solution and we have
\begin{align*} 
T(c,\delta,\mathcal{MC}_c) &= O\left[ \frac{d^3\log(N)}{\epsilon^2}\log(\frac{dLN}{\epsilon})\log\left(\frac{1}{\delta}\right) + d^2\log(N)\log(\frac{dLN}{\epsilon}) \right]\\
&= \tilde{O}\left[ \frac{d^3\log(N)}{\epsilon^2}\log(\frac{dLN}{\epsilon})\log\left(\frac{1}{\delta}\right) \right].
\end{align*}
\end{theorem}
%
By substituting Algorithm \ref{alg:multi-dim-weak} with Algorithm \ref{alg:multi-dim-strong} in the above algorithm, \revise{the acceleration scheme can be applied to Algorithm \ref{alg:multi-dim-strong} to reduce the number of required simulation runs when the indifference zone parameter $c$ is known}. We give the reduced expected simulation cost for achieving the PCS-IZ guarantee and omit the proof.
\begin{theorem}\label{thm:strong-2}
Suppose that Assumptions \ref{asp:1}-\ref{asp:4} hold. The accelerated dimension reduction method returns an $(c,\delta)$-PCS-IZ solution and we have
\begin{align*} 
T(c,\delta,\mathcal{MC}_c) &= O\left[ \frac{d^3\log(N)(d+\log(N))}{\epsilon^2}\log\left(\frac{1}{\delta}\right) + d^2\log(N)(d+\log(N)) \right]\\
&= \tilde{O}\left[ \frac{d^3\log(N)(d+\log(N))}{\epsilon^2}\log\left(\frac{1}{\delta}\right) \right].
\end{align*}
\end{theorem}

\subsection{Proof of Theorem \ref{thm:one-dim-iz}}
\label{ec:one-dim-iz}

We first estimate the simulation cost of each iteration and the sub-problem.
\begin{lemma}\label{lem:one-dim-3}
Suppose that Assumptions \ref{asp:1}-\ref{asp:4} hold. The simulation cost for each iteration of Algorithm \ref{alg:one-dim-iz} is at most $100\sigma^2c^{-2}(\revise{q_{2/3}}-\revise{q_{1/3}})^{-2}\log[4T_{max}/\delta]$, where $T_{max}:=\log_{1.5}(N) + 2$. The simulation cost of the sub-problem is at most $54\sigma^2c^{-2}\log[4T_{max}/\delta]$.
\end{lemma}
\begin{proof}{Proof.}
The proof is similar to the proof of Lemma \ref{lem:one-dim2} and we only give a sketch of the proof. By the Hoeffding bound, simulating
\[ \frac{50\sigma^2}{(\revise{q_{2/3}}-\revise{q_{1/3}})^{2}c^2}\log\left(\frac{4T_{max}}{\delta}\right) \]
times on quantiles $\revise{q_{1/3}}$ and $\revise{q_{2/3}}$ is enough to ensure that the confidence half-width is at most $(\revise{q_{2/3}} - \revise{q_{1/3}})\cdot c/5$. It implies that the last condition in line 8 is satisfied and the simulation cost of each iteration is at most $100\sigma^2c^{-2}(\revise{q_{2/3}}-\revise{q_{1/3}})^{-2}\log[4T_{max}/\delta]$. For the sub-problem, Hoeffding bound gives that simulating
\[ \frac{18\sigma^2}{c^2}\log\left(\frac{4T_{max}}{\delta}\right) \]
times for each point is enough to ensure that the confidence half-width is at most $c/3$. Since there are at most $3$ points in the sub-problem, the simulation cost for the sub-problem is at most $54\sigma^2c^{-2}\log[4T_{max}/\delta]$. 
\hfill\Halmos\end{proof}
Using Lemma \ref{lem:one-dim-3}, we can estimate the total simulation cost of Algorithm \ref{alg:one-dim-iz}.
%
%
\begin{lemma}
Suppose that Assumptions \ref{asp:1}-\ref{asp:4} hold. The expected simulation cost of Algorithm \ref{alg:one-dim-iz} is bounded by
\[ \frac{459\sigma^2}{c^{2}}\log\left(\frac{4T_{max}}{\delta}\right) = O\left[ \frac{1}{c^2}\log\left( \frac{1}{\delta} \right) \right], \]
where $T_{max}:=\log_{1.5}(N) + 2$.
\end{lemma}
\begin{proof}{Proof.}
We denote the upper bound and the lower bound at the beginning of the $k$-th iteration as $\revise{x_{U_k}}$ and $\revise{x_{L_k}}$, respectively. By Lemma \ref{lem:one-dim-3}, the simulation cost for the $k$-th iteration is at most $100\sigma^2c^{-2}(\revise{q^k_{2/3}-q^k_{1/3}})^{-2}\log[4T_{max}/\delta]$, where \revise{$q^k_{1/3}$} and \revise{$q^k_{2/3}$} are the $3$-quantiles for the $k$-th iteration. By the definition of $3$-quantiles, it follows that $\revise{q^k_{2/3}-q^k_{1/3}} \geq (\revise{x_{U_k}} - \revise{x_{L_k}}) / 3$ and therefore
\begin{align}\label{eqn:one-dim-4} \frac{100\sigma^2}{(\revise{q^k_{2/3}-q^k_{1/3}})^{2}c^{2}}\log\left(\frac{4T_{max}}{\delta}\right) \leq \frac{900\sigma^2}{(\revise{x_{U_k}} - \revise{x_{L_k}})^{2}c^{2}}\log\left(\frac{4T_{max}}{\delta}\right). \end{align}
Hence, we only need to bound the sum $\sum_{k=1}^T(\revise{x_{U_k}} - \revise{x_{L_k}})^{-2}$, where $T$ is the number of iterations of Algorithm \ref{alg:one-dim-iz}. By inequality \eqref{eqn:one-dim-3}, we know
\[ \revise{x_{U_k}} - \revise{x_{L_k}} \geq \frac32(\revise{x_{U_{k+1}}} - \revise{x_{L_{k+1}}}) - 1 ,\quad\forall k\in\{1,2,\dots,T-1\}. \]
We can rewrite the above inequality as $\revise{x_{U_k}} - \revise{x_{L_k}} - 2 \geq 3/2\cdot(\revise{x_{U_{k+1}}} - \revise{x_{L_{k+1}}} - 2)$. Since $T$ is the last iteration, it holds that $x_{U_T} - x_{L_T} \geq 4$ and therefore
\[ \revise{x_{U_k}} - \revise{x_{L_k}} - 2 \geq \left(\frac32\right)^{T-k}(\revise{x_{U_T}} - \revise{x_{L_T}} - 2) \geq 2\cdot\left(\frac32\right)^{T-k}. \]
Summing over $k=1,2,\dots,T$, we get the bound
\[ \sum_{k=1}^T~(\revise{x_{U_k}} - \revise{x_{L_k}})^{-2} \leq \sum_{k=1}^T~\left(2\cdot\left(\frac32\right)^{T-k} + 2\right)^{-2} \leq \sum_{k=1}^T~\frac14\cdot\left(\frac32\right)^{-2(T-k)} = \frac9{20}\left[1-\left(\frac49\right)^T\right] \leq \frac9{20}. \]
Combining with inequality \eqref{eqn:one-dim-4}, the simulation cost for $T$ iterations is at most
\[ \frac{900\sigma^2}{c^{2}}\log\left(\frac{4T_{max}}{\delta}\right) \cdot \sum_{k=1}^T~(\revise{x_{U_k}} - \revise{x_{L_k}})^{-2} \leq \frac{405\sigma^2}{c^{2}}\log\left(\frac{4T_{max}}{\delta}\right). \]
Considering the simulation cost of the sub-problem, the total simulation cost of Algorithm \ref{alg:one-dim-iz} is at most
\[ \frac{405\sigma^2}{c^{2}}\log\left(\frac{4T_{max}}{\delta}\right) + \frac{54\sigma^2}{c^{2}}\log\left(\frac{4T_{max}}{\delta}\right) = \frac{459\sigma^2}{c^{2}}\log\left(\frac{4T_{max}}{\delta}\right).  \]
\hfill\hfill\Halmos\end{proof}
Finally, we verify the correctness of Algorithm \ref{alg:one-dim-iz} and get an upper bound on $T(\delta,\mathcal{MC}_{c})$.
\begin{proof}{Proof of Theorem \ref{thm:one-dim-iz}.}
Similar to the proof of Theorem \ref{thm:one-dim}, we use the induction method to prove that Event-I happens for the $k$-th iteration with probability at least $1-(k-1)\delta/T_{max}$. For the first iteration, the solution to problem \eqref{eqn:obj} is in $\mathcal{X}=\{1,2,\dots,N\}$ with probability $1$. We assume that the claim is true for the first $k-1$ iterations, and consider the $k$-th iteration. If one of the first two conditions in line 8 holds when the sampling process terminates, then, by the same analysis as the proof of Theorem \ref{thm:one-dim}, we know that Event-I happens for the $k$-th iteration with probability at least $1-(k-1)\delta/T_{max}$. Hence, we only need consider the case when only the last condition in line 8 holds when the sampling process terminates. Since the first two conditions in line 8 do not hold, we know
\begin{align}\label{eqn:one-dim-5} \left|\hat{F}_n(\revise{q_{1/3}}) - \hat{F}_n(\revise{q_{2/3}})\right| \leq (\revise{q_{2/3}} - \revise{q_{1/3}})\cdot 2c / 5. \end{align}
In addition, it holds that
\[ \left| f(\revise{q_{1/3}}) - \hat{F}_n(\revise{q_{1/3}}) \right| \leq (\revise{q_{2/3}} - \revise{q_{1/3}})\cdot c / 5,\quad \left| f(\revise{q_{2/3}}) - \hat{F}_n(\revise{q_{2/3}}) \right| \leq (\revise{q_{2/3}} - \revise{q_{1/3}})\cdot c / 5 \]
with probability at least $1-\delta/T_{max}$. Combining with inequality \eqref{eqn:one-dim-5}, we know that
\begin{align}\label{eqn:one-dim-6} \left| f(\revise{q_{1/3}}) - f(\revise{q_{2/3}}) \right| \leq (\revise{q_{2/3}} - \revise{q_{1/3}})\cdot 4c / 5 < (\revise{q_{2/3}} - \revise{q_{1/3}})\cdot c \end{align}
holds with probability at least $1-\delta/T_{max}$. We assume that the above event and Event-I for the $(k-1)$-th iteration both hold, which has a joint probability of at least $1-\delta/T_{max} - (k-2)\delta/T_{max} = 1 - (k-1)\delta/T_{max}$. If the solution to problem \eqref{eqn:obj} is not in $\{\revise{q_{1/3}},\dots,\revise{q_{2/3}}\}$, then function $f(x)$ is monotone on $\{\revise{q_{1/3}},\dots,\revise{q_{2/3}}\}$ and
\[ \left| f(\revise{q_{1/3}}) - f(\revise{q_{2/3}}) \right| = \sum_{x=\revise{q_{1/3}}}^{\revise{q_{2/3}}-1}~ \left| f(x) - f(x+1) \right|. \]
Since the indifference zone parameter is $c$ and the function $f(x)$ is convex, the function value difference between any two neighbouring points is at least $c$, which implies that
\[ \sum_{x=\revise{q_{1/3}}}^{\revise{q_{2/3}}-1}~ \left| f(x) - f(x+1) \right| \geq (\revise{q_{2/3}}) - \revise{q_{1/3}})\cdot c. \]
However, the above inequality contradicts inequality \eqref{eqn:one-dim-6} and thus the solution to problem \eqref{eqn:obj} is in $\{\revise{q_{1/3}},\dots,\revise{q_{2/3}}\}$. Hence, Event-I happens for the $k$-th iteration with probability at least $1-(k-1)\delta/T_{max}$.

Suppose that there are $T$ iterations in Algorithm \ref{alg:one-dim-iz}. Since the updating rule of intervals is not changed, Lemma \ref{lem:one-dim1} gives $T \leq T_{max}-1$. By the induction method, the solution to problem \eqref{eqn:obj} is in $\{x_{L_{T+1}},\dots,\revise{x_{U_{k+1}}}\}$ with probability at least $1-T\cdot \delta/T_{max}\geq 1-\delta+\delta/T_{max}$. Using the same analysis as Theorem \ref{thm:one-dim}, the point returned by the sub-problem is at most $2c/3$ larger than the optimal value with probability at least $1-\delta$. By the assumption that the indifference zone parameter is $c$, all feasible points have function values at least $c$ larger than the optimal value. This implies that the solution returned by Algorithm \ref{alg:one-dim-iz} is optimal with probability at least $1-\delta+\delta/T_{max}-\delta/T_{max} \geq 1-\delta$ and Algorithm \ref{alg:one-dim-iz} is a $[(c,\delta)\text{-PCS-IZ},\mathcal{MC}_c]$-algorithm. 
\hfill\Halmos\end{proof}

\section{Proofs in Section \ref{sec:one-dim}}

\subsection{Proof of Theorem \ref{thm:one-dim}}
\label{ec:one-dim}

We first estimate the simulation cost of Algorithm \ref{alg:one-dim}. The following lemma gives an upper bound on the total number of iterations.
\begin{lemma}\label{lem:one-dim1}
Suppose that Assumptions \ref{asp:1}-\ref{asp:4} hold. The number of iterations of Algorithm \ref{alg:one-dim} is at most $\log_{1.5}(N) + 1$.
\end{lemma}
\begin{proof}{Proof.}
If the total number of points $N$ is at most $3$, then there is no iteration. In the following proof, we assume $N \geq 4$. We first calculate the shrinkage of interval length after each iteration. We denote the upper and the lower bound at the beginning of the $k$-th iteration as $\revise{x_{U_k}}$ and $\revise{x_{L_k}}$, respectively. Then, we know there are $n_k:=\revise{x_{U_k}}-\revise{x_{L_k}}+1$ points in the $k$-th iteration and the algorithm starts with $\revise{x_{L_1}}=1,\revise{x_{U_1}}=N$. We define the $3$-quantiles $\revise{q_{1/3}}:=\lfloor 2\revise{x_{L_k}}/3+\revise{x_{U_k}}/3\rfloor$ and $\revise{q_{2/3}}:=\lceil \revise{x_{L_k}}/3+2\revise{x_{U_k}}/3\rceil$. By the updating rule, the next interval is
\[ [\revise{x_{L_k}},\revise{q_{2/3}}]\quad\text{ or }\quad[\revise{q_{1/3}},\revise{x_{U_k}}]\quad\text{ or }\quad [\revise{q_{1/3}},\revise{q_{2/3}}]. \]
By discussing three cases when $n_k\in3\mathbb{Z}$, $n_k\in3\mathbb{Z}+1$ and $n_k\in3\mathbb{Z}+2$, we know the next interval has at most $2n_k/3+1$ points, i.e., 
\begin{align}\label{eqn:one-dim-3} n_{k+1} \leq 2n_k/3+1. \end{align}
Rewriting the inequality, we get the relation $n_{k+1} - 3 \leq 2(n_k - 3)/3$. Combining with the fact that $n_1 = N$, it follows that
\[ n_k \leq \left(\frac23\right)^{k-1}(N-3) + 3. \]
Suppose Algorithm \ref{alg:one-dim} terminates after $T$ iterations. Then, it holds that $n_T \geq 4$ and $n_{T+1} \leq 3$. Hence, we know
\[ 4 \leq n_T \leq \left(\frac23\right)^{T-1}(N-3) + 3, \]
which implies
\[ T \leq \log_{1.5}(N-3) + 1 < \log_{1.5}(N) + 1. \]
\hfill\Halmos\end{proof}
In the next lemma, we estimate the simulation cost of each iteration.
\begin{lemma}\label{lem:one-dim2}
Suppose that Assumptions \ref{asp:1}-\ref{asp:4} hold. The simulation cost of each iteration of Algorithm \ref{alg:one-dim} is at most $256\sigma^2\epsilon^{-2}\log\left(4T_{max}/\delta\right)$, where $T_{max}:=\log_{1.5}(N) + 2$. The simulation cost of the sub-problem is at most $24\sigma^2\epsilon^{-2}\log\left(4T_{max}/\delta\right)$.
\end{lemma}
\begin{proof}{Proof.}
We first calculate the confidence interval at each point $x\in\mathcal{X}$. By Assumption \ref{asp:3}, the distribution of $F(x,\xi_x)$ is sub-Gaussian with parameter $\sigma^2$. Hence, the distribution of the empirical mean $\hat{F}_n(x)$ is sub-Gaussian with parameter $\sigma^2/n$ and the Hoeffding bound gives
\[ \mathbb{P}\left[|\hat{F}_n(x) - f(x)| \geq t \right] \leq 2\mathrm{exp}\left(-\frac{nt^2}{2\sigma^2}\right),\quad\forall t\geq0. \]
Next, we estimate the simulation cost of each iteration. Choosing
\[ t_\epsilon = \frac{\epsilon}{8} ,\quad n_{\epsilon,\delta} = \frac{128\sigma^2}{\epsilon^2}\log\left(\frac{4T_{max}}{\delta}\right), \]
we get the deviation bound
\[ \mathbb{P}\left[|\hat{F}_{n_{\epsilon,\delta}}(x) - f(x)| \geq \frac{\epsilon}{8} \right] \leq \frac{\delta}{2T_{max}}. \]
If we simulate $F(\revise{q_{1/3}},\xi_{1/3})$ and $F(\revise{q_{2/3}},\xi_{2/3})$ for $n_{\epsilon,\delta}$ times, the confidence half-width is $\epsilon/4$ and the third condition in Line 8 is satisfied. Hence, the simulation cost of each iteration is at most $2n_{\epsilon,\delta}$. Finally, for the sub-problem, we choose
\[ t_\epsilon = \frac{\epsilon}{2} ,\quad \tilde{n}_{\epsilon,\delta} = \frac{8\sigma^2}{\epsilon^2}\log\left(\frac{4T_{max}}{\delta}\right) \]
and it follows that
\[ \mathbb{P}\left[|\hat{F}_{\tilde{n}_{\epsilon,\delta}}(x) - f(x)| \geq \frac{\epsilon}{2} \right] \leq \frac{\delta}{2T_{max}}. \]
Hence, simulating $\tilde{n}_{\epsilon,\delta}$ times on each point are sufficient and the simulation cost is at most $3\tilde{n}_{\epsilon,\delta}$. 
\hfill\Halmos\end{proof}
Combining Lemmas \ref{lem:one-dim1} and \ref{lem:one-dim2}, we get the total simulation cost of Algorithm \ref{alg:one-dim}.
\begin{lemma}
Suppose that Assumptions \ref{asp:1}-\ref{asp:4} hold. The expected simulation cost of Algorithm \ref{alg:one-dim} is at most
\[ \frac{256T_{max}\sigma^2}{\epsilon^{2}}\log\left(\frac{4T_{max}}{\delta}\right) = O\left[ \frac{\log(N)}{\epsilon^2}\log\left( \frac{1}{\delta} \right) \right], \]
where $T_{max}:=\log_{1.5}(N) + 2$.
\end{lemma}
\begin{proof}{Proof.}
By Lemmas \ref{lem:one-dim1} and \ref{lem:one-dim2}, the total simulation cost of the first part is at most
\[ [\log_{1.5}(N)+1]\cdot\frac{256\sigma^2}{\epsilon^{2}}\log\left(\frac{4T_{max}}{\delta}\right) \leq [T_{max} - 1] \cdot \frac{256\sigma^2}{\epsilon^{2}}\log\left(\frac{4T_{max}}{\delta}\right) \]
and the simulation cost of the second part is at most
\[ \frac{24\sigma^2}{\epsilon^{2}}\log\left(\frac{4T_{max}}{\delta}\right). \]
Combining two parts, we know the total simulation cost is at most
\[ [T_{max} - 1] \cdot \frac{256\sigma^2}{\epsilon^{2}}\log\left(\frac{4T_{max}}{\delta}\right) + \frac{24\sigma^2}{\epsilon^{2}}\log\left(\frac{4T_{max}}{\delta}\right) \leq \frac{256T_{max}\sigma^2}{\epsilon^{2}}\log\left(\frac{4T_{max}}{\delta}\right). \]
\hfill\Halmos\end{proof}
Finally, we verify the correctness of Algorithm \ref{alg:one-dim} and get an upper bound on $T(\epsilon,\delta\mathcal{MC})$. 
\begin{proof}{Proof of Theorem \ref{thm:one-dim}.}

We denote $T_{max}:=\log_{1.5}(N) + 2$. We also denote the upper and the lower bound at the beginning of the $k$-th iteration as $\revise{x_{U_k}}$ and $\revise{x_{L_k}}$, respectively. We use the induction method to prove that, for the $k$-th iteration, at least one of the following two events happens with probability at least $1-(k-1)\cdot\delta/T_{max}$:
\begin{itemize}
    \item \textbf{Event-I.} A solution to problem \eqref{eqn:obj} is in $\{\revise{x_{L_k}},\dots,\revise{x_{U_k}}\}$,
    \item \textbf{Event-II.} For any $x\in\{\revise{x_{L_k}},\dots,\revise{x_{U_k}}\}$, it holds $f(x) \leq \min_{y\in\mathcal{X}} f(y) + \epsilon$.
\end{itemize}
When $k=1$, all solutions to problem \eqref{eqn:obj} are in $\mathcal{X} = \{\revise{x_{L_1}},\dots,\revise{x_{U_1}}\}$ and Event-I happens with probability $1$. Suppose the claim is true for the first $k-1$ iterations. We consider the $k$-th iteration. For the $(k-1)$-th iteration, if Event-II happens with probability at least $1-(k-2)\cdot\delta/T_{max}$, then Event-II happens for the $k$-th iteration with the same probability. This is because the interval $\{\revise{x_{L_k}},\dots,\revise{x_{U_k}}\}$ is a subset of $\{\revise{x_{L_{k-1}}},\dots,\revise{x_{U_{k-1}}}\}$ and all points in the new interval satisfy the condition of Event-II.

Hence, we only need to consider the case when only Event-I for the $(k-1)$-th iteration happens with probability at least $1-(k-2)\delta/T_{max}$. We assume Event-I happens and consider conditional probabilities in the following of the proof. We denote
\[ \revise{q_{1/3}} := \lfloor 2\revise{x_{L_{k-1}}}/3+\revise{x_{U_{k-1}}}/3\rfloor,\quad \revise{q_{2/3}} := \lfloor \revise{x_{L_{k-1}}}/3+2\revise{x_{U_{k-1}}}/3\rfloor \]
and discuss by two different cases.
\paragraph{Case I.} Suppose one of the first two conditions in line 8 holds when the sampling process terminates. Since the two conditions are symmetrical, we assume without loss of generality that the first condition holds. Then, the new interval is $[\revise{q_{1/3}},\revise{x_{U_{k-1}}}]$ and, by the definition of confidence interval, we know
\[ f(\revise{q_{1/3}}) \geq f(\revise{q_{2/3}}) \]
holds with probability at least $1-\delta/T_{max}$. We assume the above event and Event-I for the $(k-1)$-th iteration both happen, which has joint probability at least $1-\delta/T_{max}-(k-2)\delta/T_{max} = 1 - (k-1)\delta/T_{max}$. By the convexity of $f(x)$, it holds that
\[ f(x) \geq f(\revise{q_{1/3}}) + \frac{x - \revise{q_{1/3}}}{\revise{q_{1/3}} - \revise{q_{2/3}}} \left[ f(\revise{q_{1/3}}) - f(\revise{q_{2/3}}) \right] \geq f(\revise{q_{1/3}}),\quad\forall x\in\{\revise{x_{L_{k-1}}},\dots,\revise{q_{1/3}}\}. \]
Hence, the minimum of $f(x)$ in $\{\revise{x_{L_{k-1}}},\dots,\revise{x_{U_{k-1}}}\}$ is attained by a point in $\{\revise{q_{1/3}},\dots,\revise{x_{U_{k-1}}}\}$. Combining with the assumption that there exists a solution to problem \eqref{eqn:obj} in $\{\revise{x_{L_{k-1}}},\dots,\revise{x_{U_{k-1}}}\}$, we know that there exists a solution to problem \eqref{eqn:obj} in $\{\revise{q_{1/3}},\dots,\revise{x_{U_{k-1}}}\}$. Thus, Event-I for the $k$-th iteration happens with probability at least $1-(k-1)\delta/T_{max}$. 

\paragraph{Case II.} Suppose only the last condition in line 8 holds when the sampling process terminates. Since the first two conditions in line 8 do not hold, we have
\begin{align}\label{eqn:one-dim-1} \left|\hat{F}_n(\revise{q_{1/3}}) - \hat{F}_n(\revise{q_{2/3}})\right| \leq \epsilon / 4. \end{align}
In addition, by the definition of confidence interval, it holds
\[ \left| f(\revise{q_{1/3}}) - \hat{F}_n(\revise{q_{1/3}}) \right| \leq \epsilon / 8,\quad \left| f(\revise{q_{2/3}}) - \hat{F}_n(\revise{q_{2/3}}) \right| \leq \epsilon / 8 \]
with probability at least $1-\delta/T_{max}$. Combining with inequality \eqref{eqn:one-dim-1}, we know
\begin{align}\label{eqn:one-dim-2} \left| f(\revise{q_{1/3}}) - f(\revise{q_{2/3}}) \right| \leq \epsilon / 2 \end{align}
holds with probability at least $1-\delta/T_{max}$. We assume that the above event and Event-I for the $(k-1)$-th iteration both happen, which has joint probability at least $1-\delta/T_{max}-(k-2)\delta/T_{max} = 1 - (k-1)\delta/T_{max}$. We prove that if Event-I for the $k$-th iteration does not happen, then Event-II for the $k$-th iteration happens. Under the condition that Event-I does not happen, we assume without loss of generality that solutions to problem \eqref{eqn:obj} are in $\{\revise{x_{L_{k-1}}},\dots,\revise{q_{1/3}}-1\}$.
%
Using the convexity of function $f(x)$, we know
\[ f(x) \geq f(\revise{q_{1/3}}) - \frac{\revise{q_{1/3}} - x}{\revise{q_{2/3}} - \revise{q_{1/3}}} \left[f(\revise{q_{1/3}}) - f(\revise{q_{2/3}}) \right] ,\quad\forall x\in\{\revise{x_{L_{k-1}}},\dots, \revise{q_{1/3}}\}. \]
Choosing
\[ x \in \left(\argmin_{y\in\mathcal{X}}~f(y)\right) \cap \{\revise{x_{L_{k-1}}},\dots,\revise{x_{U_{k-1}}}\} \neq \emptyset, \]
we get
\begin{align*}
    \min_{y\in\mathcal{X}}~f(y) &\geq f(\revise{q_{1/3}}) - \frac{\revise{q_{1/3}} - x}{\revise{q_{2/3}} - \revise{q_{1/3}}} \left[f(\revise{q_{1/3}}) - f(\revise{q_{2/3}}) \right]\\
    &\geq f(\revise{q_{1/3}}) - \frac{\revise{q_{1/3}} - \revise{x_{L_{k-1}}}}{\revise{q_{2/3}} - \revise{q_{1/3}}} \cdot \epsilon / 2 \geq f(\revise{q_{1/3}}) - \epsilon / 2,
\end{align*}
where the last inequality is from the definition of $3$-quantiles. Combining with inequality \eqref{eqn:one-dim-2}, we get
\[ \min_{y\in\mathcal{X}}~f(y) \geq f(\revise{q_{2/3}}) - \epsilon. \]
By the convexity of $f(x)$, it holds that
\[ \max_{x\in\{\revise{q_{1/3}},\dots,\revise{q_{2/3}}\}}~f(x) = \max\{ f(\revise{q_{1/3}}) , f(\revise{q_{2/3}}) \} \leq \min_{y\in\mathcal{X}}~f(y) + \epsilon, \]
which means Event-II for the $k$-th iteration happens. 

Combining the two cases, we know the claim holds for the $k$-th iteration. Suppose there are $T$ iterations in Algorithm \ref{alg:one-dim}. By Lemma \ref{lem:one-dim1}, we have $T\leq T_{max}-1$. By the induction method, the last interval $\{ \revise{x_{L_{T+1}}},\dots,\revise{x_{U_{T+1}}} \}$ satisfies the condition in Event-I or Event-II with probability at least $1-T\cdot\delta/T_{max} \geq 1-\delta + \delta/T_{max}$. If Event-II happens with probability at least $1-\delta + \delta/T_{max}$, then regardless of the point chosen in the sub-problem, the solution returned by the algorithm has value at most $\epsilon$ larger than the optimal value with probability at least $1-\delta + \delta/T_{max} \geq 1-\delta$. Hence, the solution satisfies the $(\epsilon,\delta)$-PGS guarantee. Otherwise, we assume Event-I happens with probability at least $1-\delta + \delta/T_{max}$. Then, a solution to problem \eqref{eqn:obj} is in $\{\revise{x_{L_{T+1}}},\dots,\revise{x_{U_{T+1}}}\}$. We choose
\[ x^* \in \left(\argmin_{x\in\mathcal{X}}~f(x)\right) \cap \{\revise{x_{L_{T+1}}},\dots,\revise{x_{U_{T+1}}}\} \]
and suppose the algorithm returns
\[ x^{**}\in \argmin_{x\in\{\revise{x_{L_{T+1}}},\dots,\revise{x_{U_{T+1}}}\}}~\hat{F}_n(x). \]
By the definition of confidence interval, it holds 
\[ f(x^{**}) \leq \hat{F}_n(x^{**}) + \epsilon/2,\quad f(x^{*}) \geq \hat{F}_n(x^{*}) - \epsilon/2 \]
with probability at least $1-\delta/T_{max}$. Under the above event, we get
\[ f(x^{**}) \leq \hat{F}_n(x^{**}) + \epsilon/2 \leq \hat{F}_n(x^{*}) + \epsilon/2 \leq f(x^*) + \epsilon. \]
Recalling that Event-I happens with probability at least $1-\delta + \delta/T_{max}$, the point $x^{**}$ satisfies the above relation with probability at least $1-\delta$ and therefore satisfies the $(\epsilon,\delta)$-PGS guarantee. Combining with the first case, we know Algorithm \ref{alg:one-dim} is an $[(\epsilon,\delta)$-PGS$,\mathcal{MC}]$-algorithm.
\hfill\Halmos\end{proof}

\subsection{Proof of Theorem \ref{thm:one-dim-uni}}
\label{ec:one-dim-uni}

We first estimate the simulation cost of Algorithm \ref{alg:one-dim-uni}. 
\begin{lemma}
Suppose that Assumptions \ref{asp:1}-\ref{asp:4} hold. The expected simulation cost for Algorithm \ref{alg:one-dim-uni} is at most
\[ \frac{25600\sigma^2}{\epsilon^2}\log\left[\frac{4N}{\delta}\right] = O\left[ \frac{1}{\epsilon^{2}}\log\left(\frac{1}{\delta}\right) \right]. \]
\end{lemma}
\begin{proof}{Proof.}
Denote $T_{max}:=N$. Suppose there are $T$ iterations in Algorithm \ref{alg:one-dim-uni}. We denote $\mathcal{S}_k$ as the active set at the beginning of the $k$-th iteration. Since each iteration reduces the size of $\mathcal{S}_k$ by at least $1$, it follows that
\[ |\mathcal{S}_k| \geq |\mathcal{S}_{T+1}| + T + 1 - k ,\quad\forall k\in[T+1]. \]
By the same analysis as Lemma \ref{lem:one-dim2}, we know that for the $k$-th iteration, simulating
\[ n(|\mathcal{S}_k|) := \frac{12800\sigma^2}{|\mathcal{S}_k|^2\epsilon^2}\log\left(\frac{4T_{max}}{\delta}\right) \]
times is sufficient to achieve $1-\delta/(2T_{max})$ confidence half-width $|\mathcal{S}_k|/80\cdot\epsilon$. Considering the condition on line 8, each point discarded during the $k$-th iteration is simulated at most $n(|\mathcal{S}_k|)$ times. Hence, the total number of simulations on points discarded during the $k$-th iteration is at most
\begin{align*} 
\left(|\mathcal{S}_k| - |\mathcal{S}_{k+1}|\right) \cdot n(|\mathcal{S}_k|) &= \frac{|\mathcal{S}_k| - |\mathcal{S}_{k+1}|}{|\mathcal{S}_k|^2} \cdot \frac{12800\sigma^2}{\epsilon^2}\log\left(\frac{4T_{max}}{\delta}\right)\\
&\leq \left(\frac{1}{|\mathcal{S}_{k+1}|} - \frac{1}{|\mathcal{S}_k|}\right) \cdot \frac{12800\sigma^2}{\epsilon^2}\log\left(\frac{4T_{max}}{\delta}\right),
\end{align*}
where the inequality is because of $|\mathcal{S}_k| \geq |\mathcal{S}_{k+1}|$. Summing over $k=1,2,\dots,T$, we get the number of simulations on all discarded points during iterations is at most
\begin{align*} 
&\sum_{k=1}^T~\left(\frac{1}{|\mathcal{S}_{k+1}|} - \frac{1}{|\mathcal{S}_k|}\right) \cdot \frac{12800\sigma^2}{\epsilon^2}\log\left(\frac{4T_{max}}{\delta}\right) = \left(\frac{1}{|\mathcal{S}_{T+1}|} - \frac{1}{|\mathcal{S}_1|}\right) \cdot \frac{12800\sigma^2}{\epsilon^2}\log\left(\frac{4T_{max}}{\delta}\right)\\
& \hspace{16em}\leq \left(1 - \frac{1}{N}\right) \cdot \frac{12800\sigma^2}{\epsilon^2}\log\left(\frac{4T_{max}}{\delta}\right)
\leq \frac{12800\sigma^2}{\epsilon^2}\log\left(\frac{4T_{max}}{\delta}\right).
\end{align*}
For points in the last active set $\mathcal{S}_{T+1}$, the number of simulations is bounded by
\[ |\mathcal{S}_{T+1}|\cdot n(|\mathcal{S}_{T+1}|) = \frac{12800\sigma^2}{|\mathcal{S}_{T+1}|\epsilon^2}\log\left(\frac{4T_{max}}{\delta}\right) \leq \frac{12800\sigma^2}{\epsilon^2}\log\left(\frac{4T_{max}}{\delta}\right). \]
Considering two parts, we know that the simulation cost of Algorithm \ref{alg:one-dim-uni} is at most
\[ \frac{25600\sigma^2}{\epsilon^2}\log\left(\frac{4T_{max}}{\delta}\right). \]

\hfill\Halmos\end{proof}
Then, we prove the correctness of Algorithm \ref{alg:one-dim-uni}. The following lemma plays a critical role in verifying the correctness of Type-II Operations.
\begin{lemma}\label{lem:one-dim-uni3}
Suppose function $h(x)$ is convex on $[1,M+a]$, where integer $M\geq 3$ and constant $a\in[0,1]$. Then, the restriction of function $h(x)$ to $[M]$, which we denote as $\tilde{h}(x)$, is also convex. Furthermore, given a constant $\epsilon>0$, if it holds that
\begin{align}\label{eqn:one-dim-cor-1} \max_{x\in[M]}~\tilde{h}(x) - \min_{x\in[M]}~\tilde{h}(x) \leq M/20 \cdot \epsilon, \end{align}
then we know
\[ \min_{y\in[M']}~\tilde{h}(2y-1) - \min_{x\in[1,M+a]}~{h}(x) \leq \epsilon/2. \]
where we define $M' := \lceil M/2 \rceil$.
\end{lemma}
\begin{proof}{Proof.}
Since the midpoint convexity of $h(x)$ implies the discrete midpoint convexity of $\tilde{h}$, we know $\tilde{h}(x)$ is also convex. We prove the second claim in three steps. 

\paragraph{Step 1.} We first prove that
\begin{align}\label{eqn:one-dim-cor-4} \min_{x\in[1,M]}~{h}(x) - \min_{x\in[1,M+a]}~{h}(x) \leq \epsilon / 10. \end{align}
Suppose $x^*$ is a minimizer of $h(x)$ on $[1,M+a]$. If $x^*\in[1,M]$, then inequality \eqref{eqn:one-dim-cor-4} holds trivially. We assume that $x^* \in (M,M+a]$. By the convexity of $h(x)$, we have
\[ h(M) - h(x^*) \leq \frac{x^* - M}{M - 1}\cdot\left[ h(M) - h(1) \right] \leq \frac{1}{M - 1}\cdot\left[ h(M) - h(1) \right] \leq \frac{M/20\cdot\epsilon}{M/2} = \epsilon / 10. \]
Hence, we know
\[  \min_{x\in[1,M]}~{h}(x) - \min_{x\in[1,M+a]}~{h}(x) = \min_{x\in[1,M]}~{h}(x) - {h}(x^*) \leq h(M) - h(x^*) \leq \epsilon / 10. \]

\paragraph{Step 2.} Next, we prove that
\begin{align}\label{eqn:one-dim-cor-2} \min_{x\in[M]}~\tilde{h}(x) - \min_{x\in[1,M]}~{h}(x) \leq \epsilon/5. \end{align}
Let $x^*$ be a minimizer of $\tilde{h}(x)$ on $[M]$. By inequality \eqref{eqn:one-dim-cor-1}, we know
\[ \max_{x\in[M]}~\tilde{h}(x) = \max\{\tilde{h}(1),\tilde{h}(M)\} \leq \tilde{h}(x^*) + M/20\cdot\epsilon. \]
By the convexity of $h(x)$, there exists a minimizer $x^{**}\in[1,M]$ of $h(x)$ in $(x^*-1,x^*+1)$. If $x^{**}=x^*$, then $\min \tilde{h}(x) = \min h(x)$ and inequality \eqref{eqn:one-dim-cor-2} holds. Hence, we assume that $x^{**}\neq x^*$ and, without loss of generality, $x^{**}\in(x^*,x^*+1)$. Since $(M-x^*-1) + (x^*-1) = M-2$, we know $\max\{M-x^*-1,x^*-1\} \geq \lceil (M-2)/2 \rceil$. We first consider the case when
\[ M-x^*-1 \geq \lceil (M-2)/2 \rceil. \]
By the convexity of $h(x)$, we have
\begin{align*}
    h(x^*+1) - h(x^{**}) &\leq \frac{x^*+1 - x^{**}}{M-x^*-1} \cdot \left[ h(M) - h(x^* + 1) \right]\\
    &\leq \frac{1}{\lceil (M-2)/2 \rceil} \cdot \left[ h(M) - h(x^*) \right] \leq \frac{M/20\cdot \epsilon}{\lceil (M-2)/2 \rceil}.
\end{align*}
By simple calculations, we get $M / 4 \leq \lceil (M-2)/2 \rceil$ for all $M\geq3$ and therefore
\[ h(x^*) - h(x^{**}) \leq h(x^*+1) - h(x^{**}) \leq \epsilon / 5, \]
which means inequality \eqref{eqn:one-dim-cor-2} holds. Now we consider the case when
\[ x^* - 1 \geq \lceil (M-2)/2 \rceil. \]
Similarly, by the convexity of $h(x)$, we have
\[ h(x^*) - h(x^{**}) \leq \frac{x^{**} - x^*}{x^*-1} \cdot \left[ h(x^*) - h(1) \right] \leq \frac{1}{\lceil (M-2)/2 \rceil} \cdot \left[ h(x^*) - h(1) \right] \leq \frac{M/20\cdot \epsilon}{\lceil (M-2)/2 \rceil} \leq \epsilon / 5. \]
Combining the two cases, we know inequality \eqref{eqn:one-dim-cor-2} holds.

\paragraph{Step 3.} Finally, we prove that
\begin{align}\label{eqn:one-dim-cor-3} \min_{y\in[M']}~\tilde{h}(2y-1) - \min_{x\in[M]}~\tilde{h}(x) \leq \epsilon/5. \end{align}
Let $x^*$ be a minimizer of $\tilde{h}(x)$. If $x^*$ is an odd number, then $\min_{y\in[M']}~\tilde{h}(2y-1) = \min_{x\in[M]}~\tilde{h}(x)$ and inequality \eqref{eqn:one-dim-cor-3} holds. Otherwise, we assume $x^* = 2y^*$ is an even number. Then, by the convexity of $\tilde{h}(x)$, there exists a minimizer of $\tilde{h}(2y-1)$ in $\{y^*, y^* + 1\}$. Without loss of generality, we assume $y^* + 1$ is a minimizer of $\tilde{h}(2y-1)$. Since $(M-x^*) + (x^*-1) = M-1$, we have $\max\{ M-x^*, x^* - 1 \} \geq \lceil (M-1)/2\rceil$. We first consider the case when
\[ M-x^* \geq \lceil (M-1)/2\rceil. \]
By the convexity of $\tilde{h}(x)$, we have
\[ \tilde{h}(2y^*+1) - \tilde{h}(2y^*) \leq \frac{1}{M-2y^*}\cdot\left[ \tilde{h}(M) - \tilde{h}(2y^*) \right] \leq \frac{M/20\cdot\epsilon}{\lceil (M-1)/2\rceil}. \]
We can verify that $M/4 \leq \lceil (M-1)/2\rceil$ for all $M\geq3$. Hence, it holds that
\[ \tilde{h}(2y^*+1) - \tilde{h}(2y^*) \leq \epsilon / 5. \]
Then, we consider the case when
\[ x^* - 1 \geq \lceil (M-1)/2\rceil. \]
Similarly, using the convexity of $\tilde{h}(x)$, we have
\[ \tilde{h}(2y^*-1) - \tilde{h}(2y^*) \leq \frac{1}{2y^*-1}\cdot\left[ \tilde{h}(2y^*) - \tilde{h}(1) \right] \leq \frac{M/20\cdot\epsilon}{\lceil (M-1)/2\rceil} \leq \epsilon / 5, \]
which implies that
\[ \tilde{h}(2y^*+1) - \tilde{h}(2y^*) \leq \tilde{h}(2y^*-1) - \tilde{h}(2y^*) \leq \epsilon/5. \]
Combining the two cases, we know inequality \eqref{eqn:one-dim-cor-3} holds.

By inequalities \eqref{eqn:one-dim-cor-4}, \eqref{eqn:one-dim-cor-2} and \eqref{eqn:one-dim-cor-3}, we have
\[ \min_{y\in[M']}~\tilde{h}(2y-1) - \min_{x\in[1,M+a]}~{h}(x) \leq \epsilon/10 + \epsilon / 5 + \epsilon / 5 = \epsilon/2. \]
\hfill\Halmos\end{proof}
We denote $\mathcal{S}_k$ and $d_k$ as the active set and the step size at the beginning of the $k$-th iteration, respectively. We define the upper bound and the lower bound for the $k$-th iteration as
\begin{align*} 
&\revise{x_{\revise{x_{L_1}}}} := 1,\quad \revise{x_{L_{k+1}}} := \begin{cases} y+d_k&\text{if the second case of Type-I Operation happens}\\ \revise{x_{L_k}}&\text{otherwise}, \end{cases}\\
&\revise{x_{\revise{x_{U_1}}}} := N,\quad \revise{x_{U_{k+1}}} := \begin{cases} y-d_k&\text{if the first case of Type-I Operation happens}\\ \revise{x_{U_k}}&\text{otherwise}. \end{cases}
\end{align*}
Although not explicitly defined in the algorithm, the interval $\{\revise{x_{L_k}},\dots,\revise{x_{U_k}}\}$ plays a similar role as in the tri-section sampling algorithm and characterizes the set of possible solutions. In the following lemma, we prove that the active set $\mathcal{S}_k$ is a good approximation to the interval $\{\revise{x_{L_k}},\dots,\revise{x_{U_k}}\}$. We note that the following lemma is deterministic.
\begin{lemma}\label{lem:one-dim-uni4}
For any iteration $k$, we have
\begin{align}\label{eqn:one-dim-cor-5} \revise{x_{L_k}} = \min~\mathcal{S}_k \quad\text{and}\quad \revise{x_{U_k}} \leq \max~\mathcal{S}_k + d_k. \end{align}
\end{lemma}
\begin{proof}{Proof.}
We use the induction method to prove the result. When $k=1$, we know $\revise{x_{L_1}}=1,\revise{x_{U_1}}=N,\mathcal{S} = [N]$ and $d_1 = 1$. Hence, the relations in \eqref{eqn:one-dim-cor-5} hold. We assume these relations hold for the first $k-1$ iterations. We discuss by two different cases.

\paragraph{Case I.} Type-I Operation is implemented during the $(k-1)$-th iteration. If the first case of Type-I Operation happens, then we know $\revise{x_{L_k}} = \revise{x_{L_{k-1}}}$ and $\revise{x_{U_k}} = y-d_{k-1}$. By the updating rule, the step size $d_{k-1}$ is not changed and all points in $\mathcal{S}_{k-1}$ that are at least $y$ are discarded from $\mathcal{S}_{k-1}$. Hence, it follows that $\max~\mathcal{S}_k = \revise{x_{U_k}}$ and the inequality $\revise{x_{U_k}} \leq \max~\mathcal{S}_k + d_k$ holds. Moreover, since both $\revise{x_{L_k}}$ and $\min~\mathcal{S}_{k-1}$ are not changed, the equality $\revise{x_{L_k}} = \min~\mathcal{S}_k$ still holds. 

Otherwise if the second case of Type-I Operation happens, then we know $\revise{x_{L_k}} = y + d_{k-1}$ and $\revise{x_{U_k}} = \revise{x_{U_{k-1}}}$. Similarly, we can prove that $\revise{x_{L_k}} = \min~\mathcal{S}_k$. Moreover, since $d_k = d_{k-1}$ and $\max~\mathcal{S}_{k-1}+d_{k-1} = \max~\mathcal{S}_{k} + d_{k-1}$, it holds
\[ \revise{x_{U_k}} = \revise{x_{U_{k-1}}} \leq \max~\mathcal{S}_{k-1}+d_{k-1} = \max~\mathcal{S}_{k}+d_{k}. \]

\paragraph{Case II.} Type-II Operation is implemented during the $(k-1)$-th iteration. In this case, bounds $\revise{x_{L_{k-1}}}$ and $\revise{x_{U_{k-1}}}$ are not changed. By the update rule, we know the step size $d_k=2d_{k-1}$ and
\begin{align}\label{eqn:one-dim-cor-6} \min~\mathcal{S}_{k} = \min~\mathcal{S}_{k-1},\quad \max~\mathcal{S}_k \in \{ \max~\mathcal{S}_{k-1} - d_{k-1}, \max~\mathcal{S}_{k-1}\}. \end{align}
Thus, the equality $\revise{x_{L_k}} = \min~\mathcal{S}_k$ still holds. By the induction assumption, we know that
\[ \revise{x_{U_k}} \leq \revise{x_{U_{k-1}}} \leq \max~\mathcal{S}_{k-1} + d_{k-1}. \]
Combining with the latter relation in \eqref{eqn:one-dim-cor-6}, we get
\[ \revise{x_{U_k}} \leq \max~\mathcal{S}_{k-1}+d_{k-1} \leq \max~\mathcal{S}_{k}+2d_{k-1} = \max~\mathcal{S}_{k}+d_{k}. \]

Combining the two cases, we know the relations in \eqref{eqn:one-dim-cor-5} hold for the $k$-th iteration. By the induction method, the relations hold for all iterations.
\hfill\Halmos\end{proof}
Finally, utilizing Lemmas \ref{lem:one-dim-uni3} and \ref{lem:one-dim-uni4}, we can prove the correctness of Algorithm \ref{alg:one-dim-uni} and get a better upper bound on $T_{0}(\epsilon,\delta,\mathcal{MC})$.
\begin{proof}{Proof of Theorem \ref{thm:one-dim-uni}.}
Denote $T_{max}:=N$. We use the induction method to prove that, for any iteration $k$, the two events
\begin{itemize}
    \item $\min_{x\in\mathcal{S}_k}~f(x) \leq \min_{x\in\mathcal{X}}~f(x) + \epsilon / 2$.
    \item $\min_{x\in\{\revise{x_{L_k}},\revise{x_{L_k}}+1,\dots,\revise{x_{U_k}}\}}~f(x) = \min_{x\in\mathcal{X}}~f(x)$.
\end{itemize}
happen jointly with probability at least $1-(k-1)\delta/T_{max}$. When $k=1$, we know $\mathcal{S}_1=\mathcal{X}$ and $\revise{x_{L_1}}=1,\revise{x_{U_1}}=N$. Hence, the two events happen with probability $1$. Suppose the claim is true for the first $k-1$ iterations. We assume the two events happen for the $(k-1)$-th iteration and consider conditional probabilities in the following proof. We discuss by two different cases.

\paragraph{Case I.} Type-I Operation is implemented in the $(k-1)$-th iteration. In this case, there exists $x,y\in\mathcal{S}_{k-1}$ such that $\hat{F}_{n_x}(x) + h_{x} \leq \hat{F}_{n_y}(y) - h_{y}$. By the definition of confidence intervals, we know $f(x) \leq f(y)$ holds with probability at least $1-\delta/T_{max}$. We assume event $f(x) \leq f(y)$ happens jointly with the claim for the $(k-1)$-th iteration, which has probability at least $1-(k-2)\delta/T_{max}-\delta/T_{max} = 1-(k-1)\delta/T_{max}$. If $x < y$, then using the convexity of $f(x)$, we know
\[ f(z) \geq f(y) \geq f(x) ,\quad\forall z \in[N]\quad\mathrm{s.t.}~z\geq y, \]
which means all discarded points have function values at least $f(x)$. Hence, the minimums in the claim are not changed, i.e., we have
\[ \min_{x\in\mathcal{S}_k}~f(x) = \min_{x\in\mathcal{S}_{k-1}}~f(x) \leq \min_{x\in\mathcal{X}}~f(x) + \epsilon / 2 \]
and
\[ \min_{x\in\{\revise{x_{L_k}},\revise{x_{L_k}}+1,\dots,\revise{x_{U_k}}\}}~f(x) = \min_{x\in\{\revise{x_{L_{k-1}}},\revise{x_{L_{k-1}}}+1,\dots, \revise{x_{U_{k-1}}}\}}~f(x) = \min_{x\in\mathcal{X}}~f(x). \]
The two events happen with probability at least $1-(k-1)\delta/T_{max}$. If $y < x$, the proof is the same and therefore the claim holds for the $k$-th iteration.

\paragraph{Case II.} Type-II Operation is implemented in the $(k-1)$-th iteration. Since $\revise{x_{L_{k-1}}}$ and $\revise{x_{U_{k-1}}}$ are not changed, the first event happens for the $k$-th iteration. Hence, we only need to verify that the second event happens with high probability. Let $x^*$ and $x^{**}$ be a minimizer and a maximizer of $f(x)$ on $\mathcal{S}_{k-1}$, respectively. By the condition of Type-II Operations, we know
\[ | \hat{F}_{n}~(x^*) - \hat{F}_n(x^{**}) | \leq h_{x^*} + h_{x^{**}} \]
and
\[ h_{x} \leq |\mathcal{S}_{k-1}|/80\cdot \epsilon ,\quad\forall x\in\mathcal{S}_{k-1}. \]
By the definition of confidence intervals, it holds
\[ | f(x^*) - \hat{F}_{n}~(x^*) | \leq h_{x^*},\quad | f(x^{**}) - \hat{F}_{n}~(x^{**}) | \leq h_{x^{**}} \]
with probability at least $1-\delta/T_{max}$. Under the above event, we have
\begin{align}\label{eqn:one-dim-cor-7} |f(x^*) - f(x^{**})| &\leq | f(x^*) - \hat{F}_{n}~(x^*) | + | \hat{F}_{n}~(x^*) - \hat{F}_n(x^{**}) | + | f(x^{**}) - \hat{F}_{n}~(x^{**}) |\\
\nonumber&\leq 2(h_{x^*}) + h_{x^{**}} \leq M/20\cdot\epsilon. \end{align}
We assume the above event happens jointly with with the claim for the $(k-1)$
-th iteration, which has probability at least $1-(k-1)\delta/T_{max}$. By the induction assumption, the original problem \eqref{eqn:obj} is equivalent to
\[ \min_{x\in\{\revise{x_{L_{k-1}}},\revise{x_{L_{k-1}}}+1,\dots,\revise{x_{U_{k-1}}}\}}~f(x). \]
Moreover, if we denote $\tilde{f}$ as the linear interpolation of $f(x)$ defined in \eqref{eqn:linear-int}, then the above problem is equivalent to
\begin{align}\label{eqn:one-dim-cor-8} \min_{x\in\{\revise{x_{L_{k-1}}},\revise{x_{L_{k-1}}}+1,\dots,\revise{x_{U_{k-1}}}\}}~f(x) = \min_{x\in[\revise{x_{L_{k-1}}},\revise{x_{U_{k-1}}}]}~\tilde{f}(x). \end{align}
We define the constant $\tilde{M}:= (\revise{x_{U_{k-1}}}-\revise{x_{L_{k-1}}})/d_{k-1}+1$ and the linear transformation
\[ T(x) := \revise{x_{L_{k-1}}} + d_{k-1}(x-1),\quad\forall x \in[1,\tilde{M}]. \]
The inverse image $T^{-1}\left([\revise{x_{L_{k-1}}},\revise{x_{U_{k-1}}}]\right)$ is $[1,\tilde{M}]$. Defining the composite function
\[ \tilde{g}(x) := \tilde{f}(T(x)),\quad\forall x\in[1,\tilde{M}], \]
we know that the problem \eqref{eqn:one-dim-cor-8} is equivalent to
\begin{align}\label{eqn:one-dim-cor-9} \min_{x\in[1,\tilde{M}]}~\tilde{g}(x). \end{align}
The inverse image $T^{-1}(\mathcal{S}_{k-1})$ is $[M]$, where $M:=|\mathcal{S}_{k-1}|$ is the number of points in $\mathcal{S}_{k-1}$. Lemma \ref{lem:one-dim-uni4} implies that $a:=\tilde{M}-M\in[0,1]$. Recalling inequality \eqref{eqn:one-dim-cor-7}, we get
\[ \max_{x\in[M]}~\tilde{g}(x) - \min_{x\in[M]}~\tilde{g}(x) \leq M/20\cdot\epsilon. \]
Now we can apply Lemma \ref{lem:one-dim-uni3} to get
\[ \min_{y\in [M']}~\tilde{g}(2y-1) - \min_{x\in[1,M+a]}~\tilde{g}(x) = \min_{y\in [M']}~\tilde{g}(2y-1) - \min_{x\in[1,\tilde{M}]}~\tilde{g}(x) \leq \epsilon/2, \]
where $M' := \lceil M/2\rceil$. Since problem \eqref{eqn:one-dim-cor-9} is equivalent to problem \eqref{eqn:one-dim-cor-8} and further equivalent to problem \eqref{eqn:obj}, it holds that
\[ \min_{y\in [M']}~\tilde{g}(2y-1) - \min_{x\in[1,\tilde{M}]}~\tilde{g}(x) = \min_{y\in [M']}~\tilde{g}(2y-1) - \min_{x\in\mathcal{X}}~f(x) \leq \epsilon/2. \]
By the definition of $\mathcal{S}_{k}$, we know $T(2[M'] - 1)$ is $\mathcal{S}_{k}$ and therefore
\[ \min_{y\in [M']}~\tilde{g}(2y-1) - \min_{x\in\mathcal{X}}~f(x) = \min_{x\in\mathcal{S}_k}~f(x) - \min_{x\in\mathcal{X}}~f(x) \leq \epsilon/2, \]
which implies that the second case happens for the $k$-th iteration with probability at least $1-(k-1)\delta/T_{max}$.

Combining the above two cases, we know the claim holds for all iterations. Suppose there are $T$ iterations in Algorithm \ref{alg:one-dim-uni}. Since each iteration will decrease the active set $\mathcal{S}$ by at least $1$, we get $T \leq N - 1$. Then after the $T$ iterations, we have
\begin{align}\label{eqn:one-dim-cor-10} \min_{x\in\mathcal{S}_{T+1}}~f(x) \leq \min_{x\in\mathcal{X}}~f(x) + \epsilon / 2 \end{align}
holds with probability at least $1-T\cdot\delta/T_{max} \geq 1-\delta + \delta/T_{max}$. For the sub-problem, using the same analysis as Theorem \ref{thm:one-dim}, the point returned by Algorithm \ref{alg:one-dim-uni} satisfies the $(\epsilon/2,\delta/T_{max})$-PGS guarantee. Combining with the relation \eqref{eqn:one-dim-cor-10}, we know the algorithm returns a solution satisfying the $(\epsilon,\delta)$-PGS guarantee. 
\hfill\Halmos\end{proof}

\subsection{Proof of Theorem \ref{thm:one-dim-low}}
\label{ec:one-dim-low}

\begin{proof}{Proof of Theorem \ref{thm:one-dim-low}.}
We construct the two models $M_1,M_2\in\mathcal{MC}$ as
\[ \nu_{1,x} := \mathcal{N}\left[c x,\sigma^2\right],\quad \nu_{2,x} := \mathcal{N}\left[ c\left(|x-2| + 2\right),\sigma^2\right],\quad\forall x\in\mathcal{X}. \]
Given a $[(c,\delta)\text{-PCS-IZ},\mathcal{MC}_c]$-algorithm, the algorithm returns point $1$ with probability at least $1-\delta$ when applied to model $M_1$, and returns point $2$ with probability at least $1-\delta$ when applied to model $M_2$. We choose $\mathcal{E}$ as the event that the algorithm returns point $1$ as the solution. Then, we know
\[ \mathbb{P}_{M_1}(\mathcal{E}) \geq 1-\delta,\quad \mathbb{P}_{M_2}(\mathcal{E}) \leq \delta. \]
Using the monotonicity of function $d(x,y)$, we get
\begin{align}\label{eqn:one-dim-low2} d(\mathbb{P}_{M_1}(\mathcal{E}),\mathbb{P}_{M_2}(\mathcal{E}))\geq d(1-\delta,\delta) \geq \log(1/2.4\delta). \end{align}
Since the distributions $\nu_{1,x}$ and $\nu_{2,x}$ are Gaussian with variance $\sigma^2$, the KL divergence can be calculated as
\[ \mathrm{KL}(\nu_{1,x},\nu_{2,x}) = \frac{\left[c x - c\left(|x-2| + 2\right)\right]^2}{2\sigma^2} = \begin{cases}
{2c^2}\sigma^{-2}&\text{if } x = 1\\
0&\text{otherwise}.
\end{cases} \]
Hence, the summation can be calculated as
\begin{align}\label{eqn:one-dim-low3} \sum_{x\in\mathcal{X}}~\mathbb{E}_{M_1}\left[N_x(\tau)\right]\mathrm{KL}(\nu_{1,x},\nu_{2,x}) = \frac{2c^2}{\sigma^2}\cdot\mathbb{E}_{M_1}\left[N_1(\tau)\right]. \end{align}
Substituting \eqref{eqn:one-dim-low2} and \eqref{eqn:one-dim-low3} into inequality \eqref{eqn:one-dim-low1}, we know
\[ \frac{2c^2}{\sigma^2}\cdot\mathbb{E}_{M_1}\left[N_1(\tau)\right] \geq \log(1/2.4\delta), \]
which implies that
\[ \mathbb{E}_{M_1}\left[\tau\right]\geq\mathbb{E}_{M_1}\left[N_1(\tau)\right]\geq \frac{\sigma^2}{2c^2}\cdot\log(\frac{1}{2.4\delta}) = \revise{\Theta\left[ \frac{1}{c^{2}}\log\left(\frac{1}{\delta}\right) \right].} \]
\hfill\Halmos\end{proof}

\section{Deterministic cutting-plane methods}
\label{ec:cuttingplane}

In this section, we give the pseudo-codes of Vaidya's cutting-plane method \citep{vaidya1996new} and the random walk-based cutting-plane method \citep{bertsimas2004solving} for the self-contained purpose.

\subsection{Vaidya's cutting-plane method}

We first give the pseudo-code for Vaidya's cutting-plane method \citep{vaidya1996new}. We note that examples of Newton-type methods include the original Newton method, quasi-Newton methods and the cubic-regularized Newton method.
\bigskip
\begin{breakablealgorithm}
\caption{Vaidya's cutting-plane method}
\label{alg:vaidya}
\begin{algorithmic}[1]
\Require{Model $\mathcal{X},f(x)$, optimality guarantee parameter $\epsilon$, Lipschitz constant $L$, $\mathcal{SO}$ oracle $\hat{g}$.}
\Ensure{An $\epsilon$-solution $x^*$ to problem \eqref{eqn:obj}.}
\State Set the initial polytope $P \leftarrow [1,N]^d$.
\State \revise{Set the constant $\rho \leftarrow 10^{-7}$. \Comment{Constant $\rho$ corresponds to $\epsilon$ in \citet{vaidya1996new}.}}
\State \revise{Set the number of iterations $T_{max}\leftarrow \lceil 2d / \rho \cdot \log[d N L / (\rho\epsilon)] \rceil$.}
\State Initialize the set of points used to query separation oracles $\mathcal{S} \leftarrow \emptyset$.
\State Initialize the volumetric center $z \leftarrow (N+1)/2 \cdot (1,1,\dots,1)^T$.
\For{$T=1,2,\dots,T_{max}$}
    \State Decide adding or removing a cutting plane by checking $\sigma_i(z)$ for $i\in[T]$ \citep{vaidya1996new}.
    \If{add a cutting plane}
        \State Evaluate the $\mathcal{SO}$ oracle $\hat{g}$ at $z$.
        \If{$\hat{g}=0$}
            \State Return $z$ as the approximate solution.
        \EndIf
        \State Add the current point $z$ to $\mathcal{S}$.
    \ElsIf{remove a cutting plane}
        \State Remove corresponding point $z$ from $\mathcal{S}$.
    \EndIf
    \State Update the approximate volumetric center $z$ by a Newton-type method.
\EndFor \Comment{There are at most $O(d)$ points in $\mathcal{S}$ by Vaidya's method.}
\State Return the solution $\hat{x}$ to problem $\min_{x\in\mathcal{S}}~ f(x)$.
\end{algorithmic}
\end{breakablealgorithm}
\bigskip

\subsection{Random walk-based cutting-plane method}

Next, we list the pseudo-code for the random walk-based cutting-plane method in \citet{bertsimas2004solving}.
\bigskip
\begin{breakablealgorithm}
\caption{Deterministic random walk-based cutting-plane method}
\label{alg:randomwalk}
\begin{algorithmic}[1]
\Require{Model $\mathcal{X},f(x)$, optimality guarantee parameter $\epsilon$, Lipschitz constant $L$, $\mathcal{SO}$ oracle $\hat{g}$.}
\Ensure{An $\epsilon$-solution $x^*$ to problem \eqref{eqn:obj}.}
\State Set the initial polytope $P \leftarrow [1,N]^d$.
\State Set the number of iterations $T_{max}\leftarrow O[d\log(d L N / \epsilon)]$.
\State Set the number of samples required to calculate the center $M\leftarrow O(d)$.
\State Initialize the set of points used to query separation oracles $\mathcal{S} \leftarrow \emptyset$.
\State Initialize the volumetric center $z \leftarrow (N+1)/2 \cdot (1,1,\dots,1)^T$.
\For{$T=1,2,\dots,T_{max}$}
    \State Evaluate the $\mathcal{SO}$ oracle $\hat{g}$ at $z$.
    \State Add the current point $z$ to $\mathcal{S}$.
    \State Add the cutting plane using $\hat{g}$ to $P$.
    \If{$P=\emptyset$} \Comment{This step requires solving a linear feasibility problem}
        \State \textbf{break}
    \EndIf
    \State Uniformly sample $M$ points from the new polytope $P$ via random walk.
    \State Update the approximate volumetric center $z$ to the average of all sampled points.
\EndFor 
\State Return the solution $\hat{x}$ to problem $\min_{x\in\mathcal{S}}~ f(x)$.
\end{algorithmic}
\end{breakablealgorithm}
\bigskip%

\section{Proofs in Section \ref{sec:multi-dim}}

\subsection{Proof of Lemma \ref{lem:multi-dim-uni}}
\label{ec:lem-multi-dim-uni}

\begin{proof}{Proof of Lemma \ref{lem:multi-dim-uni}.}

Let $k\in\{2,3,\dots,N-1\}$. By the definition of $f^{d-1}(x)$, there exists vectors $y_{k-1},y_{k+1}\in[N]^{d-1}$ such that
\[ f^{d-1}(k-1) = f(y_{k-1}, k-1),\quad f^{d-1}(k+1) = f(y_{k+1}, k+1). \]
By the $L^\natural$-convexity of $f(x)$, we have
\begin{align*} 
f^{d-1}(k-1) + f^{d-1}(k+1) &= f(y_{k-1}, k-1) + f(y_{k+1}, k+1)\\
&\geq f\left( \left\lceil \frac{y_{k-1} + y_{k+1}}{2} \right\rceil, k \right) + f\left( \left\lfloor \frac{y_{k-1} + y_{k+1}}{2} \right\rfloor, k \right)\\
&\geq 2\min_{y\in[N]^{d-1}}~f(y,k) = 2f^{d-1}(k),
\end{align*}
which means the discrete midpoint convexity holds at point $k$. Since we can choose $k$ arbitrarily, we know function $f^{d-1}(x)$ is convex on $[N]$.
\hfill\Halmos\end{proof}

\subsection{ Proof of Theorem \ref{thm:multi-dim-uni} }
\label{ec:multi-dim-uni}

\begin{proof}{Proof of Theorem \ref{thm:multi-dim-uni}.}
We first verify the correctness of Algorithm \ref{alg:multi-dim-uni}. The algorithm is the same as Algorithm \ref{alg:one-dim-uni} except the condition for implementing Type-II Operations. Hence, if we can prove that, when Type-II Operations are implemented, it holds
\begin{align}\label{eqn:multi-dim-uni-1} h \leq |\mathcal{S}|\cdot \epsilon/80, \end{align}
then the proof of Theorem \ref{thm:one-dim-uni} can be directly applied to this case. If the confidence interval is updated at the beginning of current iteration, then we have
\[ h = |\mathcal{S}|\cdot\epsilon / 160 < |\mathcal{S}|\cdot\epsilon / 80. \]
Otherwise, if the confidence interval is not updated in the current iteration. Then, we have $|\mathcal{S}| > N_{cur}/2$ and therefore
\[ h = N_{cur}\cdot \epsilon / 160 < 2|\mathcal{S}|\cdot\epsilon / 160 = |\mathcal{S}|\cdot\epsilon / 80. \]
Combining the two cases, we have inequality \eqref{eqn:multi-dim-uni-1} and the correctness of Algorithm \ref{alg:multi-dim-uni}.

Next, we estimate the simulation cost of Algorithm \ref{alg:multi-dim-uni}. Denote the active sets when we update the confidence interval as $\mathcal{S}_1,\dots,\mathcal{S}_m$, where $m\geq1$ is the number of times when the confidence interval is updated. Then, we know $|\mathcal{S}_1|=N$ and $|\mathcal{S}_m| \geq 3$. By the condition for updating the confidence interval, it holds
\[ | \mathcal{S}_{k+1} | \leq | \mathcal{S}_k | / 2,\quad\forall k\in[m-1], \]
which implies
\[ | \mathcal{S}_k | \geq 2^{m-k}| \mathcal{S}_m | \geq 3\cdot 2^{m-k},\quad\forall k\in[m]. \]
Since the algorithm $\mathcal{A}$ is sub-Gaussian with parameter $C$, for each $x\in\mathcal{S}_k$, the simulation cost for generating $\hat{f}^{d-1}(x)$ is at most
\[ \frac{2C}{h^2}\log(\frac{2T_{max}}{\delta}) = \frac{2C}{160^{-2}|\mathcal{S}_k|^2\epsilon^2}\log(\frac{2T_{max}}{\delta}) = |\mathcal{S}_k|^{-2}\cdot \frac{51200C}{\epsilon^2}\log(\frac{2T_{max}}{\delta}). \]
Hence, the total simulation cost for the $k$-th update of confidence intervals is at most
\[ |\mathcal{S}_k| \cdot |\mathcal{S}_k|^{-2}\cdot \frac{51200C}{\epsilon^2}\log(\frac{2T_{max}}{\delta}) = |\mathcal{S}_k|^{-1}\cdot \frac{51200C}{\epsilon^2}\log(\frac{2T_{max}}{\delta}) \leq 2^{k-m}/3 \cdot \frac{51200C}{\epsilon^2}\log(\frac{2T_{max}}{\delta}). \]
Summing over all iterations, we have the simulation cost of all iterations of Algorithm \ref{alg:multi-dim-uni} is at most
\[ \sum_{k=1}^m~2^{k-m}/3 \cdot \frac{51200C}{\epsilon^2}\log(\frac{2T_{max}}{\delta}) = \left( 2 - 2^{1-m} \right) \cdot \frac{51200C}{3\epsilon^2}\log(\frac{2T_{max}}{\delta}) < \frac{102400C}{3\epsilon^2}\log(\frac{2T_{max}}{\delta}). \]
Now we consider the simulation cost of the last subproblem. Since the algorithm \ref{alg:multi-dim-uni} is sub-Gaussian with parameter $C$, the simulation cost of the subproblem is at most
\[ 2 \cdot \frac{2C}{(\epsilon/4)^2}\log(\frac{2T_{max}}{\delta}) = \frac{64C}{\epsilon^2}\log(\frac{2T_{max}}{\delta}). \]
Hence, the total simulation cost of Algorithm \ref{alg:multi-dim-uni} is at most
\[ \frac{102400C}{3\epsilon^2}\log(\frac{2T_{max}}{\delta}) + \frac{64C}{\epsilon^2}\log(\frac{2T_{max}}{\delta}) < 17099 \cdot \frac{2C}{\epsilon^2}\log(\frac{2T_{max}}{\delta}). \]
When $\delta$ is small enough, we can choose $M=17100$ and the asymptotic simulation cost of Algorithm \ref{alg:multi-dim-uni} is at most
\[ \frac{2MC}{\epsilon^2}\log(\frac{2T_{max}}{\delta}), \]
which implies that Algorithm \ref{alg:multi-dim-uni} is sub-Gaussian with dimension $d$ and parameter $MC$.
\hfill\Halmos\end{proof}

\subsection{Proof of Lemma \ref{lem:weak-1}}
\label{ec:weak-1}

\begin{proof}{Proof of Lemma \ref{lem:weak-1}.}
By the assumption that $F(x,\xi_x)-f(x)$ is sub-Gaussian with parameter $\sigma^2$ for any $x$, we know that $\hat{g}_{\alpha_x(i)} - g_{\alpha_x(i)}$ is the difference of two independent sub-Gaussian random variables and therefore
\[ \hat{g}_{\alpha_x(i)} - {g}_{\alpha_x(i)} \sim \mathrm{subGaussian}\left(2\sigma^2\right) ,\quad \forall i \in [d], \]
where $g$ is the subgradient of $f(x)$ defined in \eqref{eqn:subgrad}. Then, using the properties of sub-Gaussian random variables, it holds that
\[ \hat{g}^n_{\alpha_x(i)} - {g}_{\alpha_x(i)} \sim \mathrm{subGaussian}\left(\frac{2\sigma^2}{n}\right) ,\quad \forall i \in [d]. \]
Recalling that components of $\hat{g}^n$ are mutually independent, we know
\[ \langle \hat{g}^n - g, y - x \rangle = \sum_i~(\hat{g}^n_{\alpha_x(i)} - {g}_{\alpha_x(i)}) \cdot (y-x)_{\alpha_x(i)} \sim \mathrm{subGaussian}\left( \frac{2\sigma^2}{n} \cdot \|y-x\|_2^2 \right). \]
Since $\|y-x\|_2^2 \leq dN^2$, we know
\[ \langle \hat{g}^n - g, y - x \rangle \sim \mathrm{subGaussian}\left( \frac{2dN^2\sigma^2}{n} \right). \]
By the Hoeffding bound, it holds
\[ \left| \langle \hat{g}^n - g, y - x \rangle \right| \leq \sqrt{\frac{4dN^2\sigma^2}{n} \log\left( \frac{2}{\delta} \right)} \]
with probability at least $1-\delta$. If we choose 
\[ n = \left\lceil\frac{4dN^2\sigma^2}{\epsilon^2}\log\left(\frac{2}{\delta}\right)\right\rceil \leq \frac{4dN^2\sigma^2}{\epsilon^2}\log\left(\frac{2}{\delta}\right) + 1, \]
it follows that
\begin{align}\label{eqn:weak-1} \left| \langle \hat{g}^n - g, y - x \rangle \right| \leq \epsilon. \end{align}
Since $f(x)$ is a convex function and $g$ is a subgradient at point $x$, we have $f(y) \geq f(x) + \langle g, y - x \rangle$ for all $y \in [1,N]^d$. Combining with inequality \eqref{eqn:weak-1} gives
\[ f(y) \geq f(x) + \langle \hat{g}^n, y - x \rangle + \langle g - \hat{g}^n, y - x \rangle \geq f(x) + \langle \hat{g}^n, y - x \rangle - \epsilon ,\quad \forall y \in [1,N]^d \]
holds with probability at least $1-\delta$. Then, considering the half space $H = \{ y: \langle \hat{g}^n , y - x \rangle \leq 0 \}$, it holds 
\[ f(y) \geq f(x) + \langle \hat{g}^n, y - x \rangle - \epsilon \geq f(x) - \epsilon ,\quad \forall y \in [1,N]^d \cap H^c \]
with the same probability. Taking the minimum over $[1,N]^d \cap H^c$, it follows that the averaged stochastic subgradient provides an $(\epsilon,\delta)$-$\mathcal{SO}$ oracle. Finally, the expected simulation cost of each oracle evaluation is at most
\[ d\cdot n \leq \frac{4d^2N^2\sigma^2}{\epsilon^2}\log\left(\frac{2}{\delta}\right) + d = \tilde{O}\left[ \frac{d^2N^2}{\epsilon^2}\log\left(\frac{1}{\delta}\right) \right]. \]
\hfill\Halmos\end{proof}

\subsection{Proof of Theorem \ref{thm:weak-1}}
\label{ec:thm-weak-1}

\revise{
Before we provide the proof of Theorem \ref{thm:weak-1}, we show the calculation of the number of iterations $T_{max}$. With a slight abuse of notations, we use the same notations as \citet{vaidya1996new} only in this calculation. Before the first iteration, we have the volumetric center as 
\[ \omega = \frac{N+1}{2} \cdot (1,\dots,1)^T \in \mathbb{R}^d.  \]
Therefore, we can calculate that
\[ H(\omega) = \frac{8}{(N-1)^2} \cdot I_{d}, \quad \rho^0 = \frac{d}{2}\log\left( \frac{8}{(N-1)^2} \right),  \]
where $I_{d}$ is the $d\times d$ identity matrix. By \citet{vaidya1996new}, the volume of the polytope at the beginning of the $t$-th iteration satisfies
\begin{align}\label{eqn:ec-101-1}
\nonumber\log(\pi^t) &\leq d\log\left(\frac{2d}{\rho}\right) - \rho^0 - \frac{\rho}{2} \cdot t = d\log\left(\frac{2d}{\rho}\right) - \frac{d}{2}\log\left( \frac{8}{(N-1)^2} \right) - \frac{\rho}{2} \cdot t\\
&\leq d\log\left(\frac{2d}{\rho}\right) + d\log\left( \frac{N}{2} \right) - \frac{\rho}{2} \cdot t = d\log\left(\frac{Nd}{\rho}\right) - \frac{\rho}{2} \cdot t.
\end{align}
The target set consists of points in the set 
\[ P(\epsilon) := \left\{ x\in[1,N]^d ~:~ \|x - x^*\|_1 \leq \epsilon / L \right\}, \]
where $x^*$ is the optimal solution of problem \eqref{eqn:obj} and $L$ is the Lipschitz constant of $f(\cdot)$. By a simple analysis, we know that the volume of $P(\epsilon)$ satisfies
\[ \mathrm{vol}( P(\epsilon) ) \geq \left( \frac{\epsilon}{L} \right)^d. \]
Therefore, we can terminate the algorithm when
\[ \log(\pi^{T_{max}}) \leq d\log\left( \frac{\epsilon}{L} \right). \]
Combining with inequality \eqref{eqn:ec-101-1}, we know
\[ T_{max} \geq \frac{2d}{\rho} \cdot \log\left( \frac{NdL}{\rho \epsilon} \right) \]
is sufficient for $\epsilon$-approximate solutions.
}

\begin{proof}{Proof of Theorem \ref{thm:weak-1}.}

We first prove the correctness of Algorithm \ref{alg:multi-dim-weak}. If $\hat{g} = 0$ for some iteration, the half space $H=\mathbb{R}^d$ and the definition of $(\epsilon/8,\delta/4)$-$\mathcal{SO}$ implies that
\[ f(y) \geq f(z) - \epsilon/8 ,\quad \forall y \in [1,N]^d \]
holds with probability at least $1-\delta/4$, where $z$ is the point that the separation oracle is called. Hence, we know $z$ is an $(\epsilon/8,\delta/4)$-PGS solution and obviously satisfies the $(\epsilon/2,\delta/2)$-PGS guarantee. Then, by Theorem \ref{thm:round}, the integral solution after the round process is an $(\epsilon,\delta)$-PGS solution.

In the following of the proof, we assume $\hat{g}\neq 0$ for all iterations. Let $x^* \in \mathcal{X}$ be a minimizer of problem \eqref{eqn:obj}. We consider the set
\[ Q := \left(x^* + \left[ -\frac{\epsilon}{8L}, \frac{\epsilon}{8L} \right]^d \right) \cap [1,N]^d. \]
We can verify that set $Q$ is not empty and has volume at least $({\epsilon}/(8L))^d$. Moreover, for any $x \in Q$, it holds
\[ f(x) \leq f(x^*) + L \| x - x^*\|_\infty \leq f(x^*) + \frac{\epsilon}{8}. \]
By the analysis in~\citet{vaidya1996new}, the volume of the polytope $P$ is smaller than $({\epsilon}/(8L))^d$ after
\[ T_{max} := O\left[ d\log\left( \frac{8dLN}{\epsilon} \right) \right] \]
iterations. Hence, after $T_{max}$ iterations, the volume of $P$ is smaller than the volume of $Q$ and it must hold $Q \backslash P \neq \emptyset$. Since $Q \subset [1,N]^d$, the constraint $1\leq x_i \leq N$ is not violated for all $i\in[d]$. Thus, if we choose $x \in Q \backslash P$, there exists a cutting plane $-\hat{g}^T y \geq \beta $ in $P$ such that
\[ -\hat{g}^T x < \beta \leq -\hat{g}^T z, \]
where $z$ is the point that the $(\epsilon/8,\delta/4)$-$\mathcal{SO}$ oracle $\hat{g}$ is evaluated and $\beta$ is the value chosen by Vaidya's method. This implies that $x$ is not in the half space
\[ H:=\{ y: \hat{g}^Ty \leq \hat{g}^Tz \}. \]
Then, by the definition of $(\epsilon/8,\delta/4)$-$\mathcal{SO}$ oracle and the claim that $x\in [1,N]^d \cap H^c$, we know
\[ f(x) \geq f(z) - \epsilon / 8 \]
holds with probability at least $1-\delta/4$. On the other hand, the condition $x\in P$ leads to
\[ f(x) \leq f(x^*) + \epsilon / 8. \]
Combining the last two inequalities gives that
\[ \min_{y\in\mathcal{S}}~f(y) \leq f(z) \leq f(x^*) + \epsilon / 4 \]
holds with probability at least $1-\delta/4$. Hence, the $(\epsilon/4,\delta/4)$-PGS solution $\hat{x}$ of problem $\min_{y\in\mathcal{S}}~f(y)$ satisfies
\[ f(\hat{x}) \leq f(x^*) + \epsilon / 2 \]
with probability at least $1-\delta/2$. Equivalently, the solution $\hat{x}$ is an $(\epsilon/2,\delta/2)$-PGS solution. Using Theorem \ref{thm:round}, the integral solution returned by Algorithm \ref{alg:multi-dim-weak} is an $(\epsilon,\delta)$-PGS solution. 

Now, we estimate the expected simulation cost of Algorithm \ref{alg:multi-dim-weak}. By Lemma \ref{lem:weak-1}, the simulation cost of each $(\epsilon/8,\delta/4)$-$\mathcal{SO}$ oracle is at most
\[ {O}\left[\frac{d^2N^2}{\epsilon^2}\log\left(\frac{1}{\delta}\right) + d\right]. \]
Since at most one separation oracle is evaluated in each iteration, the total simulation cost of $T_{max}$ iterations is at most
\[ {O}\left[ \left(\frac{d^2N^2}{\epsilon^2}\log\left(\frac{1}{\delta}\right) + d \right) \cdot d\log\left( \frac{8dLN}{\epsilon} \right) \right] = \tilde{O}\left[ \frac{d^3N^2}{\epsilon^2}\log(\frac{dLN}{\epsilon})\log\left(\frac{1}{\delta}\right) \right]. \]
By the property of Vaidya's method, there are $O(d)$ cutting planes in the polytope $P$. Then, using the same analysis as~\citet{zhang2020discrete}, the expected simulation cost of finding an $(\epsilon/4,\delta/4)$-PGS solution of the sub-problem $\min_{y\in\mathcal{S}}~f(y)$ is at most
\[ \tilde{O}\left[ \frac{d^2}{\epsilon^2}\log\left(\frac{1}{\delta}\right) \right]. \]
We note that the evaluation of $(\epsilon/8,\delta/4)$-$\mathcal{SO}$ oracles at points in $\mathcal{S}$ provides enough simulations for the sub-problem and therefore the simulation cost of this part can be avoided. Finally, the expected simulation cost of the rounding process is bounded by
\[ \tilde{O}\left[ \frac{d}{\epsilon^2}\log\left(\frac{1}{\delta}\right) \right]. \]
Combining the three parts, the total expected simulation cost of Algorithm \ref{alg:multi-dim-weak} is at most
\[ \tilde{O}\left[ \frac{d^3N^2}{\epsilon^2}\log(\frac{dLN}{\epsilon})\log\left(\frac{1}{\delta}\right) \right]. \]
\hfill\Halmos\end{proof}

\subsection{Proof of Theorem \ref{thm:strong-1}}
\label{ec:thm-strong-1}

\begin{proof}{Proof of Theorem \ref{thm:strong-1}.}
We first verify the correctness of Algorithm \ref{alg:multi-dim-strong}. If the optimal solution has been removed during the dimension reduction process, we claim that the optimal solutions are removed from the search set by some cutting plane. This is because the dimension reduction steps will not remove integral points from the current search set~\citep{jiang2020minimizing}. Then, by the same proof as Theorem \ref{thm:weak-1}, it holds
\begin{align}\label{eqn:strong-1} \min_{x\in\mathcal{S}}~f(x) \leq \min_{x\in\mathcal{X}}~f(x) + \epsilon / 4 \end{align}
with probability at least $1-\delta/4$. Otherwise if the optimal solution has not been removed from the search set throughout the dimension reduction process, we know the last one-dimensional problem contains the optimal solution. Hence, the $(\epsilon/4,\delta/4)$-PGS solution to the one-dimensional problem is also an $(\epsilon/4,\delta/4)$-PGS solution to the original problem. Since the PGS solution is also added to the set $\mathcal{S}$, we also have relation \eqref{eqn:strong-1} holds with probability at least $1-\delta/4$. Then, the $(\epsilon/4,\delta/4)$-PGS solution $\bar{x}$ to problem $\min_{x\in\mathcal{S}}f(x)$ satisfies
\[ f(\bar{x}) \leq \min_{x\in\mathcal{X}}~f(x) + \epsilon / 2 \]
with probability at least $1-\delta/2$, or equivalently $\bar{x}$ is an $(\epsilon/2,\delta/2)$-PGS solution to problem \eqref{eqn:obj}. Using the results of Theorem \ref{thm:round}, the solution returned by Algorithm \ref{alg:multi-dim-strong} is an $(\epsilon,\delta)$-PGS solution.

Next, we estimate the expected simulation cost of Algorithm \ref{alg:multi-dim-strong}. By the results in \citet{jiang2020minimizing}, $(\epsilon/4,\delta/4)$-$\mathcal{SO}$ oracles are called at most $O[d(d+\log(N))]$ times. Hence, the size of $\mathcal{S}$ is at most $O[d(d+\log(N))]$. By the estimates in Lemma \ref{lem:weak-1}, the total simulation cost of the dimension reduction process is at most
\[ O\left[ \frac{d^3N^2(d+\log(N))}{\epsilon^2}\log\left(\frac{1}{\delta}\right) + d^2(d+\log(N)) \right] = \tilde{O}\left[ \frac{d^3N^2(d+\log(N))}{\epsilon^2}\log\left(\frac{1}{\delta}\right) \right]. \]
Moreover, the one-dimensional convex problem has at most $N$ feasible points and Theorem \ref{thm:one-dim-uni} implies that the expected simulation cost for this problem is at most
\[ \tilde{O}\left[\frac{1}{\epsilon^2}\log\left(\frac{1}{\delta}\right)\right]. \]
Since the size of $\mathcal{S}$ is at most $O[d(d+\log(N))]$, the sub-problem for the set $\mathcal{S}$ takes at most
\[ O\left[ \frac{d^2(d+\log(N))}{\epsilon^2}\log\left(\frac{1}{\delta}\right) + d^2(d+\log(N)) \right] = \tilde{O}\left[ \frac{d^2(d+\log(N))}{\epsilon^2}\log\left(\frac{1}{\delta}\right) \right] \]
simulation runs. Finally, Theorem \ref{thm:round} shows that the expected simulation cost of the rounding process is at most
\[ \tilde{O}\left[ \frac{d}{\epsilon^2}\log\left(\frac{1}{\delta}\right) \right]. \]
In summary, the total expected simulation cost of Algorithm \ref{alg:multi-dim-strong} is at most
\[ \tilde{O}\left[ \frac{d^3N^2(d+\log(N))}{\epsilon^2}\log\left(\frac{1}{\delta}\right) \right]. \]
\hfill\Halmos\end{proof}

\section{Proofs in Section \ref{sec:var}}

\subsection{Proof of Theorem \ref{thm:ada-1}}
\label{ec:thm-ada-1}

We first prove that $\hat{\sigma}$ serves as an upper bound on the sub-Gaussian parameter. The proof is based on the property that the lower tail of a squared sub-Gaussian random variable is sub-Gaussian.
\begin{lemma}
\label{lem:ada-1}
Let $\delta\in(0,1]$ be the failing probability. The parameter estimator $\hat{\sigma}^2$ in Definition \ref{def:ada} satisfies
\[ \mathbb{P}\left( \hat{\sigma}^2 \leq \sigma_x^2 \right) \leq \delta / 2.  \]
\end{lemma}

\begin{proof}{Proof of Lemma \ref{lem:ada-1}.}
By the definitions of $\hat{\mathrm{Var}}$ and $\hat{\sigma}^2$, we only need to prove that
\begin{align}\label{eqn:ada-3}
\mathbb{P}\left( \frac{1}{n} \sum_{i=1}^n \left[F(x, \xi_{2i-1}) - F(x,\xi_{2i})\right]^2 \leq \kappa\sigma_x^2 \right) \leq \delta / 2.
\end{align}
By the independence between $\xi_{2i-1}$ and $\xi_{2i}$, the random variable $F_i := F(x, \xi_{2i-1}) - F(x,\xi_{2i})$ is zero-mean and sub-Gaussian with parameter $2\sigma_x^2$ for all $i\in[n]$. Using the fact that $-F_i^2\leq 0$ almost surely and the one-sided Bernstein's inequality, we have
\[ \mathbb{P}\left[ \frac{1}{n} \sum_{i=1}^n \left(- F_i^2 + \mathbb{E}(F_i^2)\right) \geq \kappa\sigma_x^2 \right] \leq \exp\left[ -\frac{n\kappa^2\sigma_x^4}{2/n \sum_{i=1}^n\mathbb{E}(F_i^4) } \right]. \]
Since $\{F_i,i\in[n]\}$ are i.i.d. and zero-mean random variables, the above inequality is equivalent to
\[ \mathbb{P}\left[ \frac{1}{n} \sum_{i=1}^n F_i^2 - \mathrm{Var}(F_1) \leq -\kappa\sigma_x^2 \right] \leq \exp\left[ -\frac{n\kappa^2\sigma_x^4}{2 \mathbb{E}(F_1^4) } \right]. \]
Now, recalling the assumption in \eqref{eqn:ada-1} and $\mathrm{Var}(F_1) = 2\mathrm{Var}(F(x,\xi_1))$, we get
\begin{align}\label{eqn:ada-2}
\mathbb{P}\left[ \frac{1}{n} \sum_{i=1}^n F_i^2 - 2\kappa \sigma_x^2 \leq -\kappa\sigma_x^2 \right] \leq \mathbb{P}\left[ \frac{1}{n} \sum_{i=1}^n F_i^2 - \mathrm{Var}(F_1) \leq -\kappa\sigma_x^2 \right] \leq \exp\left[ -\frac{n\kappa^2\sigma_x^4}{2 \mathbb{E}(F_1^4) } \right].
\end{align}
To estimate the fourth moment of $F_1$, we calculate that
\begin{align*}
\mathbb{E}(F_1^4) &= \int_0^\infty t \mathbb{P}(F_1^4 \geq t)~dt = \int_0^\infty 4s^3 \mathbb{P}(F_1^4 \geq s^4)~ds = \int_0^\infty 4s^3 \mathbb{P}(|F_1| \geq s)~ds\\
&\leq \int_0^\infty 4s^3\cdot 2\exp[-s^2/(8\sigma_x^2)] ~ds = 128\sigma_x^4,
\end{align*}
where the second equality is from the substitution $t=s^4$ and the last inequality is from the fact that $F_1$ is $2\sigma_x^2$-sub-Gaussian. Substituting into inequality \eqref{eqn:ada-2}, we get
\[ \mathbb{P}\left[ \frac{1}{n} \sum_{i=1}^n F_i^2  \leq \kappa\sigma_x^2 \right] \leq \exp\left( -\frac{n\kappa^2\sigma_x^4}{256 \sigma_x^4 } \right) = \exp\left( -\frac{n\kappa^2}{256} \right) \leq \frac{\delta}{2}, \]
where the last inequality is from the choice of $n$. The above inequality is equivalent to inequality \eqref{eqn:ada-3} and the proof is done.
\hfill\Halmos\end{proof}

With the help of Lemma \ref{lem:ada-1}, we now prove the theorem.
\begin{proof}{Proof of Theorem \ref{thm:ada-1}.}
By Lemma \ref{lem:ada-1}, we know that the event $\mathcal{E} := \{\sigma_x^2\leq \hat{\sigma}^2\}$ happens with probability at least $1-\delta/2$. By the Hoeffding's inequality and the definitions of $n$ and $m$, we have
\begin{align*} 
\mathbb{P}\left[ |\hat{F}(x; \delta) - f(x)| \geq \epsilon, \mathcal{E} ~|~ \hat{\sigma}^2 \right] &\leq \exp\left[ - \frac{m\epsilon^2}{2\sigma_x^2} \right] \leq \exp\left[ - \frac{2\epsilon^{-2}\hat{\sigma}^2\cdot\log(2/\delta)\epsilon^2}{2\sigma_x^2} \right]\\
&\leq \exp\left[ - \log(2/\delta)\epsilon^2 \right] = \frac{\delta}{2}.
\end{align*}
Taking expectation over $\hat{\sigma}^2$ leads to
\[ \mathbb{P}\left[ |\hat{F}(x; \delta) - f(x)| \geq \epsilon, \mathcal{E} \right] \leq \frac{\delta}{2}. \]
Therefore, we get
\begin{align*} 
\mathbb{P}\left[ |\hat{F}(x; \delta) - f(x)| \geq \epsilon \right] &= \mathbb{P}\left[ |\hat{F}(x; \delta) - f(x)| \geq \epsilon, \mathcal{E} \right] + \mathbb{P}\left[ |\hat{F}(x; \delta) - f(x)| \geq \epsilon, \mathcal{E}^c \right]\\
&\leq \mathbb{P}\left[ |\hat{F}(x; \delta) - f(x)| \geq \epsilon, \mathcal{E} \right] + \mathbb{P}\left[ \mathcal{E}^c \right] \leq \delta,
\end{align*}
where $\mathcal{E}^c$ is the complementary set of $\mathcal{E}$. 

The estimation of the expected simulation cost is from the fact that $\hat{\mathrm{Var}}$ is an unbiased estimator and the bound
\[ m = \max\{\lceil 2\epsilon^{-2}\hat{\sigma}^2\log(2/\delta) \rceil, \lceil 512\kappa^{-2}\log(2/\delta) \rceil\} \leq 2\epsilon^{-2}\hat{\sigma}^2\log(2/\delta) \rceil + \lceil 512\kappa^{-2}\log(2/\delta) \rceil. \]
\hfill\Halmos\end{proof}


\end{document}